\documentclass[11pt,letterpaper]{amsart}

\usepackage{amssymb,amsmath,amsfonts,amsthm}
\usepackage{graphicx}
\usepackage[normalem]{ulem}
\graphicspath{ {./images/} }
\usepackage{setspace}

\usepackage{mathrsfs}  
\usepackage{psfrag}
\usepackage{mathtools}
\usepackage{color}
\usepackage{todonotes}
\usepackage{enumitem}
\definecolor{Red}{rgb}{1,0,0}
\definecolor{Blue}{rgb}{0,0,1}
\definecolor{Green}{rgb}{0,1,0}
\definecolor{magenta}{rgb}{1,0,.6}
\definecolor{gold}{rgb}{.6,.5,0}
\definecolor{orange}{rgb}{1,0.4,0}
\definecolor{darkgreen1}{rgb}{0, .35, 0}
\definecolor{darkgreen}{rgb}{0, .6, 0}
\definecolor{darkred}{rgb}{.75,0,0}

\usepackage{hyperref}
\hypersetup{colorlinks=true, citecolor=darkgreen, linkcolor=blue, hypertexnames=false}

\usepackage{chngcntr}
\counterwithin{equation}{section}

\DeclarePairedDelimiter{\floor}{\lfloor}{\rfloor}
\DeclarePairedDelimiter{\roof}{\lceil}{\rceil}
\theoremstyle{plain}
\newtheorem{main}{Theorem}

\newtheorem{theorem}{Theorem}[section]
\newtheorem{lemma}[theorem]{Lemma}
\newtheorem{proposition}[theorem]{Proposition}

\theoremstyle{remark}
\newtheorem{remark}[theorem]{Remark}
\newtheorem{definition}{Definition}

\newtheorem{conjecture}[theorem]{Conjecture}

\newcommand\numberthis{\addtocounter{equation}{1}\tag{\theequation}}

\newcommand{\Sing}{\operatorname{Sing}}

\newcommand{\diam}{\operatorname{diam}}

\newcommand{\cgs}{\mathcal G_{\mbox{spec}}}

           \def\ea{\end{array}}
          \def\ec{\end{center}}
     \def\ed{\end{description}}
        \def\ee{\end{equation}}
       \def\eea{\end{eqnarray}}
     \def\eeaa{\end{eqnarray*}}
 \def\et{\end{thebibliography}}
\def\bib{\bibitem}

\def\bM{{\bf{M}}}

\def\Orb{{\rm Orb}}

\def\Cl{{\rm Cl}}

\def\Reg{{\rm Reg}}
\def\Sing{{\rm Sing}}

\def\CP{{\rm CP}}

\newcommand{\interior}[1]{%
	{\kern0pt#1}^{\mathrm{o}}%
}

\def\supp{\operatorname{supp}}

\def\cG{{\mathcal G}}
\def\cA{{\mathcal A}}
\def\cD{{\mathcal D}}
\def\cC{{\mathcal C}}
\def\cO{{\mathcal O}}

\def\cU{{\mathcal U}}

\def\cR{{\mathcal R}}
\def\cB{{\mathcal B}}

\def\cF{{\mathcal F}}
\def\cM{{\mathcal M}}
\def\cN{{\mathcal N}}
\def\cP{{\mathcal P}}

\def\cR{{\mathcal R}}

\def\cS{{\mathcal S}}

\def\id{\operatorname{id}}
\def\length{\operatorname{length}}

\def\vep{\varepsilon}

\def\TT{{\mathbb T}}
\def\RR{{\mathbb R}}

\def\ZZ{{\mathbb Z}}
\def\NN{{\mathbb N}}

\def\Var{\operatorname{Var}}
\def\Exp{\operatorname{Exp}}

\def\dxt{\delta_{(x,t)}}

\title[Equilibrium states for sectional-hyperbolic attractors]{Equilibrium states for the classical Lorenz attractor and sectional-hyperbolic attractors in higher dimensions}
\date{\today}
\author{Maria Jose Pacifico, Fan Yang and Jiagang Yang}

\address{Instituto de Matem\'atica, Universidade Federal do Rio de Janeiro, C. P. 68.530, CEP 21.945-970,  Rio de Janeiro, RJ, Brazil.}
 \email{pacifico@im.ufrj.br }

\address{Department of Mathematics, Wake Forest University, Winston-Salem, NC, USA.}
\email{yangf@wfu.edu}

\address{Departamento de Geometria, Instituto de Matem\'atica e Estat\'\i stica, Universidade Federal Fluminense, Niter\'oi, Brazil}
\email{yangjg\@@impa.br}

\thanks{MJP and JY were partially supported by CAPES-Finance Code 001 and CNPq-Brazil No. 307776/2019-0 and 302975/2019-5. respectively. MJP was partially supported by FAPERJ Grant- CNE E-26/202.850/2018(239069). FY was partially supported by NSF Award Number 2418590.}

\begin{document}

\begin{abstract}
	It has long been conjectured that the classical Lorenz attractor supports a unique measure of maximal entropy. In this article, we give a positive answer to this conjecture and its higher-dimensional counterpart by considering the uniqueness of equilibrium states for H\"older continuous functions on a sectional-hyperbolic attractor $\Lambda$. We prove that in a $C^1$-open and dense family of vector fields (including the classical Lorenz attractor), if the point masses at singularities are not equilibrium states, then there exists a unique equilibrium state supported on $\Lambda$. In particular, there exists a unique measure of maximal entropy for the flow $X|_\Lambda$.  
\end{abstract}

\maketitle

\tableofcontents

	\section{Introduction}
	The discovery and study of the Lorenz equation
	\begin{equation}\label{e.lorenz}
		(\dot{x}, \dot{y}, \dot{z})= (10(y-x), 28x-y-xz, xy-8/3z) 
	\end{equation}	
	by Lorenz in his celebrated paper~\cite{Lo63} was one of the most significant milestones in the study of chaotic dynamics and natural science as a whole. It in fact predates Smale's famous horseshoe~\cite{Smale}. Similar to the horseshoe which captures the nature of discrete-time, differentiable chaotic dynamical systems, the Lorenz attractor indeed captures the very essence of chaotic vector fields with {\em singularities}, i.e., points where the vector field becomes degenerate.\footnote{In many books, singularities are also called {\em equilibria}. However, we will not use this terminology in this paper to avoid confusion with equilibrium states.} 
	
	Numerical simulations carried out by Lorenz for an open neighborhood of the chosen parameters suggested that almost all points in the phase space tend to a zero volume chaotic attractor. Based on this fact, Lorenz conjectured that such an attractor indeed exists for the flow generated by equations (\ref{e.lorenz}).
	
	
	About a decade later, the difficulty of proving the Lorenz conjecture led
	Afraimovich, Bykov, Shil'nikov and Guckenheimer, Williams to introduce the geometric Lorenz attractor independently, which reproduces the behavior predicted by Lorenz, \cite{ABS, Gu76,GW79}. In particular, in \cite{GW79} the authors prove the main topological properties of the geometric Lorenz attractor. 
	%
	
	While the discovery of the Lorenz attractor leveraged fundamental developments in dynamical systems, the Lorenz equations themselves continued to resist all attempts to prove that they exhibit a sensitive attractor. 
	
	
	The quest to rigorously prove the Lorenz conjecture continued. The main difficulty is the unavoidable fact that solutions slow down
	as they pass near the singularity at the origin, which means unbounded return times and so unbounded integration errors. 
	
	In the 1980s, Bunimovich and Sinai \cite{Bu79} indicated a program that could prove that the Lorenz equations have a transitive attractor. 
	Finally, Tucker proved the existence of a non-trivial attractor in~\cite{Tu00, Tucker} around the turn of the century using a computer-assisted proof.


	Arguably the statistical point of view on dynamical systems is one of the most useful tools available for the study of the asymptotic behavior of transformations or flows, and one of the most important concepts in this theory is the notion of the {\em equilibrium states}. 
	Here an equilibrium state is an invariant probability measure that achieves the supremum in the variational principle:
	$$
	P(\phi) = \sup\left\{h_\mu(X) + \int \phi \,d\mu: \mu\in \cM_X(\bM)\right\}.
	$$
	These terminologies will be explained in detail in Section~\ref{s.p2}.  Among equilibrium states are the physical measure (also known as the Sinai-Ruelle-Bowen measures~\cite{B75_2, BR, S72}, where $\phi=-\log \det\|Df|_{E^u}\|$) and the measure of maximal entropy (for $\phi\equiv 0$).

	Of particular interest is the measure of maximal entropy (abbrev. MME), a measure that captures the most complex behavior of the system. It has a deep connection with statistical physics and stochastic processes (such as Markov chains) and is closely related to various geometric properties such as the growth rate of closed orbit, as observed by Margulis in his pioneer work~\cite{Margulis}.
	Establishing the existence and, more importantly, the uniqueness of MME has been successful for hyperbolic diffeomorphisms and vector fields without any singularity. Among such works, let us mention: Anosov flows by Bowen~\cite{B75_2} and Margulis~\cite{Margulis}, Rank one geodesic flows by Knieper~\cite{Knieper}, and more recently, dispersing billiards by Baladi and Demers~\cite{BD}.

However, significant difficulty was encountered when people tried to work on the MME and, more generally, equilibrium states other than the SRB measure (see~\cite{LY} for the uniqueness of the SRB measure for the Lorenz attractor) for the classical Lorenz attractor. This is because the Lorenz attractor is not hyperbolic, and the geometric model was proven to be insufficient
when working with equilibrium states other than the SRB measure.\footnote{Take the measure of maximal entropy ($\phi\equiv 0$) for example; a measure of maximal entropy for a continuous flow $X$ is likely not obtained from a measure of maximal entropy for the Poincar\'e return map to a cross-section.}

Therefore a global theory for singular vector fields (i.e., vector fields that contain at least one singularity accumulated by regular orbits) parallel to Smale's hyperbolicity and independent of the geometric model is needed. 
The search for such a theory took another two decades and was finally successful around the year 1999, 
where {\em singular-hyperbolicity} was introduced in~\cite{MPP99}. 
As opposed to uniform hyperbolicity which requires the tangent bundle to split into $E^s\oplus \langle X\rangle \oplus  E^u$ (and inherently requires the vector fields to be non-degenerate), the authors of~\cite{MPP99} ask for a splitting of the form $E^{ss}\oplus F^{cu}$ where the second bundle contains the flow direction of non-singular points and is volume expanding. It is proven that robustly transitive attractors are singular-hyperbolic. 
Shortly after, Tucker proved in~\cite{Tucker} that the classical Lorenz attractor is indeed singular hyperbolic and supports a unique SRB measure. It was soon realized that singular-hyperbolicity in fact captures the typical behavior of singular vector fields: it is proven~\cite{CY1} that among three-dimensional singular vector fields that are $C^1$ away from homoclinic tangencies, singular-hyperbolicity is open and dense.\footnote{In comparison, for surface diffeomorphisms~\cite{PS} and three-dimensional non-singular vector fields~\cite{AR}, uniform hyperbolicity is $C^1$ open and dense when away from homoclinic tangencies.}

After this seminal work, significant advances in this theory were achieved through the work of many authors, who gradually turned to the statistical and stochastic point of view of the Lorenz attractor; see, for instance, \cite{AMV15, APPV09, LMP, AM16}.

However, the existence and uniqueness of general equilibrium states, including the MME, for the classical Lorenz attractor~\eqref{e.lorenz} remained open. 
In this paper we will finally give a positive answer to this question and its higher-dimensional counterparts known as the {\em sectional-hyperbolic attractors} (for such an example, see~\cite{BPV}). In particular, we will prove the following theorems.

\begin{main}\label{m.Lorenz}
The Lorenz attractor given by~\eqref{e.lorenz} with the classical parameter values supports a unique measure of maximal entropy. More generally, for every H\"older continuous function  $\phi: \RR^3\to \RR$ that satisfies 
$$
\phi(0)< P(\phi),
$$
there exists a unique equilibrium state.
\end{main}

\begin{main}\label{m.A}
Let  $\mathfrak{X}^1(\bM)$ be the space of $C^1$ vector fields on a compact Riemannian manifold $\bM$ with dimension at least three. There exists an open and dense subset $\cR\subset\mathfrak{X}^1(\bM)$, such that for every $X\in\cR$, every sectional-hyperbolic attractor $\Lambda$ of $X$, and every  H\"older continuous function $\phi:\bM\to\RR$ satisfying 
\begin{equation*}
	\phi(\sigma) < P(\phi, X|_\Lambda), \forall \sigma\in\Sing(X)\cap \Lambda,
\end{equation*}
$X|_\Lambda$ has a unique equilibrium state $\mu$ supported on $\Lambda$.
\end{main}
For the precise definition of sectional hyperbolicity, see the next subsection.

Both theorems are direct consequences of a more technical theorem (Theorem~\ref{m.B}) which provides the uniqueness of equilibrium states under three verifiable topological assumptions. The precise statement, as well as the definition of various terms, will be introduced in Section~\ref{ss.1.2} and Section~\ref{s.p1}.

\subsection{Statement of the main result}\label{ss.1.2}
Let  $\mathfrak{X}^1(\bM)$ be the space of $C^1$ vector fields on a compact Riemannian manifold $\bM$ with dimension at least three. For $X\in\mathfrak{X}^1(\bM)$, we denote by $f=(f_t)_{t\in\RR}$ the flow generated by $X$. A singularity is a point where $X(x)=0$, and the collection of singularities is denoted by $\Sing(X)$. Throughout this article we will use the notation that for $x\in\bM$ and $t\in \RR$, 
$$
x_t := f_t(x).
$$
In comparison, when referring to a sequence of points in $\bM$, we will always use superscripts such as $x^1,x^2,\ldots$.

For a compact $f$-invariant set $\Lambda$, we say that $\Lambda$ has a {\em  dominated splitting under the tangent flow $(Df_t)_{t\in\RR}$} if there exist a continuous invariant splitting $T_\Lambda\bM =  E\oplus F$, together with two constants $C_0>0$ and $\lambda_0>1$, such that for all $x\in\Lambda$ and $t>0$ one has 
\begin{equation}\label{e.dom}
\|Df_t\mid_{E(x)}\| \|Df_{-t}\mid_{F(x_t)}\| \le C\lambda_0^{-t}.
\end{equation}
In particular, every dominated splitting for the flow $f$ is a dominated splitting for its time-$t$ map $f_t$, $\forall t\in\RR{\setminus \{0\}}$.

In order to capture the hyperbolicity of the Lorenz attractor, {singular-hyperbolicity} was first introduced for three-dimensional flows in~\cite{MPP99}. This definition is later generalized to higher dimensional vector fields as {\em sectional-hyperbolicity} in~\cite{MM} and~\cite{LGW}. 

\begin{definition}\label{d.sectionalhyperbolic}
A compact invariant set $\Lambda$ of a flow $X$ is called {\em sectional-hyperbolic}, if it admits a dominated splitting $E^{ss}\oplus F^{cu}$, such that $Df_t\mid_{E^{ss}}$ is uniformly contracting, and $Df_t\mid_{F^{cu}}$ is {\em sectional-expanding:}  there are constants $C_1>0,\lambda>1$ such that for every $x\in\Lambda$ and any subspace $V_x\subset F^{cu}_x$ with $\dim V_x = 2$, we have 
\begin{equation}\label{e.sect.hyp}
	|\det Df_t(x)|_{V_x}| \ge C_1 \lambda ^t \mbox{ for all } t>0.
\end{equation}
\end{definition}

Now we are ready to state the main theorem of this paper.

\begin{main}\label{m.B}
Let $X$ be a $C^1$ vector field on a compact, boundaryless manifold $\bM$ with dimension at least three. Let $\Lambda$ be a sectional-hyperbolic attractor of $X$, and $\phi:\bM\to \RR$ be a H\"older continuous function. Assume that:
\begin{enumerate}
	\item[(A)] all singularities in $\Lambda$ are hyperbolic  and non-degenerate;
	\item[(B)] there exists a hyperbolic periodic orbit $\gamma\subset \Lambda$ such that 
	\begin{enumerate}
		\item[(B1)] the unstable manifold of $\gamma$ is a {quasi $u$-section}: for every regular point $x\in \Lambda$, one has 
		$$
		\cF^{s}(x)\pitchfork W^u(\gamma)\ne\emptyset,
		$$
		where $\cF^{s}(x)$ is the stable manifold of $x$ tangent to $E^{ss}$ bundle; and
		\item[(B2)] the stable manifold of $\gamma$ is dense in $\Lambda$;
	\end{enumerate}
	\item[(C)] the potential function $\phi$ satisfies 
	\begin{equation}\label{e.pgap}
		\phi(\sigma) < P(\phi, X|_\Lambda), \forall \sigma\in\Sing(X)\cap \Lambda.
	\end{equation}
\end{enumerate}
Then there exists a unique equilibrium state $\mu_\phi$ for $X|_\Lambda$ supported on $\Lambda$. 
\end{main}

Later we will see that Assumptions (A) and (B) are $C^1$ open dense (Theorem~\ref{m.A}), and are satisfied by the classical Lorenz equation. Now let us make a few observations concerning Assumption (C) of Theorem~\ref{m.B}.

\begin{remark}\label{r.mainthm}
	By the variational principle for flows, if $\phi$ is a continuous function on $\bM$, then for every $\sigma\in\Sing(X)\cap\Lambda$ we have 
	$$
	P(\phi, X|_\Lambda)\ge h_{\delta_\sigma}(f_1) + \int \phi(x)\,d\delta_\sigma(x) = \phi(\sigma),
	$$
	where $\delta_\sigma$ is the point mass at $\sigma$. Therefore Assumption~(C) holds if and only if none of the measures $\delta_\sigma$ is an equilibrium state of $\phi$. In particular, since $h_{top}(X|_\Lambda)>0$ (\cite[Theorem C]{PYY}), Assumption (C) is satisfied by the constant potential $\phi\equiv 0$ whose equilibrium states are the measures of maximal entropy. 
	
	Clearly Assumption (C) is optimal; otherwise, one may have as many as $\#(\Sing(X)\cap\Lambda)$ ergodic equilibrium states, each of which is a point mass on a singularity. In this case, it is unknown whether there exists a unique, non-atomic equilibrium state on $\Lambda$.
\end{remark}

This observation leads to the following conjecture concerning the number of equilibrium states when Assumption (C) fails.

\begin{conjecture}
	Suppose that Assumptions (A) and (B) hold. Then there exists at most one equilibrium state that is ergodic and non-atomic. In particular, there are at most $\#(\Sing(X)\cap\Lambda)+1$ many ergodic equilibrium states. 
\end{conjecture}

The following proposition states that Assumption (C) is satisfied by the ``majority'' of H\"older continuous functions.
\begin{proposition}\label{p.opendense}
	Assumption (C) is a $C^0$ open condition among continuous functions. Furthermore, there exists a residual set $\cR\subset \mathfrak{X}^1(\bM)$ such that for every $X\in\cR$ and every sectional-hyperbolic attractor of $X$, Assumption (C) holds $C^0$ densely in the space of H\"older continuous functions.
\end{proposition}

\begin{proof}
	The topological pressure $P(\phi, X|_\Lambda)$ is a continuous function of $\phi$ under $C^0$ topology (see~\cite[Theorem 9.7]{Wal}). On the other hand, $\phi(\sigma)$ also varies continuously in $C^0$ topology for every $\sigma\in\Sing(X)$. Therefore Assumption (C) holds on an $C^0$ open subset of $C^0(\Lambda)$. 
	
	Next we consider the $C^0$ denseness. It is proven in~\cite[Theorem A]{YZ} (first proven for three-dimensional flows in~\cite[Corollary 1]{MSV}) that there exists a residual set $\cR\subset \mathfrak{X}^1(\bM)$ such that for every $X\in\cR$ and every sectional-hyperbolic attractor of $X$, there exists a dense (in fact, residual) subset $\mathfrak D\subset C^0(\Lambda)$ such that for every $\phi\in \mathfrak D$, the maximal value of the function $g_\phi(\mu):= \mu(\phi)$ is achieved by a unique, ergodic measure $\bar\mu_\phi$ with $\supp\bar\mu_\phi=\Lambda$. The uniqueness of this measure shows that for every singularity $\sigma$, 
	$$
	\phi(\sigma) = \delta_\sigma(\phi) < \bar\mu_\phi(\phi)\le  h_{\bar\mu_\phi}(X|_\Lambda)+\bar\mu_\phi(\phi) \le P(\phi,X|_\Lambda).
	$$
	This shows that Assumption (C) holds $C^0$ open and densely in  $C^0(\Lambda)$, and therefore in the space of H\"older continuous functions. 
\end{proof}

The rest of this paper is devoted to the proof of Theorem~\ref{m.B}. The proof of Theorem~\ref{m.Lorenz} and Theorem~\ref{m.A} are contained in Section~\ref{ss.proofAB}.

\subsection{Structure of the paper and the sketch of the proof of Theorem~\ref{m.B}}
Before starting the proof, we would like to briefly explain the idea and highlight the structure of this article. The main tools used in this paper are:
\begin{itemize}
	\item an improved Climenhaga-Thompson criterion for the uniqueness of equilibrium states (Theorem~\ref{t.improvedCL}) developed by the authors of this paper in a previous work~\cite{PYY21};
	\item Liao's scaled tubular neighborhood theorem (see Section~\ref{ss.liao}); 
	\item the existence of a coordinate system on the normal bundle over regular points given by fake foliations (see~\cite{BW}), and the transition to another fake foliation coordinate system near each singularity; see Section~\ref{s.sing};
	\item a careful study on the Pliss times of the hyperbolicity of the $F^{cu}_N$ bundle and on the recurrence property to a small neighborhood of $\Sing(X)$ (Section~\ref{s.pliss}), which proves that if an orbit segment $(x,t)$ is ``good'', then the orbit of every point in the $(t,\vep)$-Bowen ball of $x$ is contained in Liao's scaled tubular neighborhood.
\end{itemize}

Let us be more precise. Climenhaga and Thompson provide in their remarkable work~\cite{CT16} an easy-to-verify topological criterion (abbr.\ CT criterion) on the uniqueness of equilibrium states for flows and homeomorphisms. The criterion consists of the following ingredients:
\begin{enumerate}[label={(\Roman*)}]
	\item[(0)] there exists a large collection of orbit segments $\cD$ with a `decomposition' $(\cP,\cG,\cS)$; thinking of $\cG$ as the `good core', for any positive number $M$ we define $\cG^M$ by adding a `prefix' from $\cP$ and a `suffix' from $\cS$  with length at most $M$ to orbit segments in $\cG$. Then it is easy to see that
	$$
	\cG\subset \cG^1\subset\cG^2\subset\cdots\subset \cD
	$$
	which eventually exhausts the set $\cD$;
	\item the system has the specification property on $\cG^M$ at scale $\delta>0$, for every $M\in\NN$;
	\item the potential function $\phi$ has bounded distortion (the Bowen property) on $\cG$ at a given scale {$\vep> 40\delta$};
	\item the `bad' parts of the system, consisting of $\cP,\cS,\cD^c$ and those points where the system is not expansive, must have smaller pressure compared to $\cG$.
\end{enumerate}
Under these assumptions, they prove that there exists a unique equilibrium state with the upper and lower Gibbs property.

It is worth noting that all the assumptions listed above are made at certain fixed scales, namely $\delta$ and $\vep$. This is particularly useful for certain applications such as the Bonatti-Viana diffeomorphism on $\TT^4$~\cite{CFT18}, Ma\~n\'e's derived from Anosov diffeomorphism on $\TT^3$~\cite{CFT19}, and geodesic flows on surfaces without conjugate points~\cite{CKW}.

Unfortunately, it turns out that their criterion is not applicable to our situations due to Assumption (I), the specification on every $\cG^M, M\ge 0$. To overcome this difficulty, we prove, in a previous work~\cite{PYY21}, that Assumption (I) can be replaced by the following, much weaker assumption
\begin{center}
	(I''). $\cG$ has tail (W)-specification at scale $\delta$.
\end{center}
We will provide more explanation later when we discuss the specification property of singular flows (see the beginning of Section~\ref{s.spec}). More details regarding the CT criterion and our improved version can be found in Section~\ref{s.p2}.

In Section~\ref{s.pliss} we explain one of the main ingredients of this work: the Pliss lemma (Theorem~\ref{t.pliss}; see also~\cite{Pliss}). We will use the Pliss Lemma in two ways:
\begin{enumerate}
	\item first we will apply the Pliss lemma on the scaled linear Poincar\'e flow and obtain $cu$-hyperbolic times which have a positive density (Lemma~\ref{l.hyptime});
	\item then we will use the Pliss lemma on the recurrence to a small neighborhood $W$ of singularities in $\Lambda$. Roughly speaking, we prove in Lemma~\ref{l.rec} that, for ``typical''\footnote{Here ``typical'' means that orbits which do not have this property must have smaller topological pressure.} orbit segments $(x,t)$, there exist times $t_i\in [0,t]$ such that every time the orbit sub-segment $(x,t_i)$ spends some time inside $W$, it then must spend a much longer amount of time outside $W$. This estimate applies to every time (not just the overall frequency) that the orbit visits $W$. More importantly, we will prove that the density of such $t_i$'s can be made arbitrarily close to one. 
\end{enumerate}
Combining these two cases together, we prove that for a ``typical'' orbit segment $(x,t)$, there exists a time $s\in[0,t]$ that is a $cu$-hyperbolic time {\bf and} has the good recurrence property to $W$ mentioned above. Such times will play an important role in the construction of the good core $\cG$ (Lemma~\ref{l.bihyptime}). 

Then we prove Theorem~\ref{m.B} in Section~\ref{s.mAproof}. In view of the (improved) CT criterion, the good core $\cG$ is defined as those orbit segments $(x,t)$ such that $x$ is away from a neighborhood of $\Sing(X)$ and away from a compact subset of $\cup_\sigma W^u(\sigma)$ (we will explain this later), and $x_t$ is both a $cu$-hyperbolic time and has the good recurrence property to $W$. Then we will state two theorems (Theorem~\ref{t.Bowen} and~\ref{t.spec}) concerning the Bowen property and specification on $\cG$.  The proof of these theorems occupies the rest of this paper.

In Section~\ref{s.sing} and~\ref{s.bowen} we prove Theorem~\ref{t.Bowen} on the Bowen property of $\cG$. Let us explain the primary difficulties in the proof and how to overcome them.
\begin{itemize}
	\item The dominated splitting $T_\Lambda\bM = E^{ss}\oplus F^{cu}$ induces a dominated splitting $E^s_N\oplus F^{cu}_N$ on the normal bundle of regular points that is invariant under the scaled linear Poincar\'e flow (Lemma~\ref{l.basic}).
	\item  It is easy to see that vectors in $E^s_N$ are uniformly contracted by the (scaled) linear Poincar\'e flow, so there exist certain stable manifolds on $\cN(x)$ (here $\cN(x)$ is the image of the normal bundle $N(x) = \langle X\rangle^\perp$ under the exponential map at $x$) for $x\in \Reg(X)\cap\Lambda$. However, the $F^{cu}_N$ bundle is non-uniformly hyperbolic. In particular, as orbit segments move away from a singularity, certain direction in $F^{cu}_N$ (e.g., the direction that is almost parallel to $E^c(\sigma)$) is contracted by the (scaled) linear Poincar\'e flow. As a result, we do not have any uniform estimates on the hyperbolicity in $F^{cu}_N$. We solve this issue by considering orbit segments $(x,t)$ such that $x_t$ is a  $cu$-hyperbolic times for $(\psi_t^*)$. See Section~\ref{s.pliss}.
	\item There is no local product structure on $\cN$. We have to use fake foliations $\cF^{s/cu}_{x,\cN}$ (originally developed in~\cite{BW} for partially hyperbolic diffeomorphisms), which exist only at the uniformly relative scale $\cN_{\rho_0|X(x)|}(x)$. 
	Since the scale of the fake foliation is not uniform, one loses control when an orbit segment enters a small neighborhood of some singularity. To solve this problem, we construct another family of fake foliations $\cF^*_\sigma, *=s,c,u,cs,cu$ near each singularity. These foliations are defined at a uniform scale. Since $\cF^*_{x,N}$ and $\cF_\sigma^*$ are not compatible (since the former is defined using the linear Poincar\'e flow and the latter using the tangent flow), every time a flow orbit enters and leaves a small neighborhood of a singularity, we must consider a change of coordinates.  During this process, only the larger coordinate is preserved. The estimate on the change of coordinates is carried out under a general setting in Section~\ref{ss.cone}. Then, in Section~\ref{ss.nearsing} we state two key lemmas {Lemma~\ref{l.Elarge} and~\ref{l.Flarge}} that deal with two cases (large $E$-length and large $F$-length) separately. These two lemmas show that when an orbit segment enters and leaves a small neighborhood of a singularity, the large coordinate gains contraction/expansion. 
	\item Since the local product structure formed by $\cF_{x,N}^{s/cu}$ only exists in Liao's scaled tubular neighborhood, to prove the Bowen property one must show that every $y$ in the $(t,\vep)$-Bowen ball of $x$ is in fact $\rho$-scaled shadowed by the orbit of $x$ up to time $t$. This is the Main Proposition (Proposition~\ref{p.key}). For this purpose, we introduce a new Pliss time which captures the recurrence property to a neighborhood $W$ of singularities. It requires that for a good orbit segment $(x,t)$ and every $s\in[0,t]$, the orbit segment $(x_s,t-s)$ can only spend a small portion of its time inside $W$. In particular, if during its last visit to $W$, the orbit segment $(x,t)$ spends time $T$ in $W$, it must then spend $cT,c\gg1$, amount of time outside $W$ (see Definition~\ref{d.RecPliss} for the precise formulation). This allows the hyperbolicity accumulated outside $W$ (under the scaled linear Poincar\'e flow) to overcome the lack of estimate inside it (recall that $\psi_t^*$ has a bounded norm  for every fixed $t$; see Proposition \ref{p.tubular3}). Furthermore, we prove the simultaneous existence of the recurrence Pliss time and the $cu$-hyperbolic time mentioned earlier. These simultaneous Pliss times play a central role in the construction of the good core $\cG$. 
	The proof of the Main Proposition, together with the Bowen property on $\cG$ occupies Section~\ref{s.bowen}.
\end{itemize}

The proof of the specification property on $\cG$ can be found in Section~\ref{s.spec}. In particular, we prove that the  specification property at scale $\delta$ holds on the collection of orbit segments with the following properties:
\begin{itemize}
	\item $x$ is away from some compact neighborhoods $\CP_\delta$ contained in the unstable manifold of every singularity $\sigma\in \Sing(X)\cap\Lambda$; 
	\item $x_t$ is away from singularities; and
	\item $x_t$ is a $cu$-hyperbolic time.
\end{itemize}  
In particular, the choice of $\cG$ depends on the scale $\delta$, and therefore we need an improved $CT$-criterion which replaces Assumption (I) with ``$\cG$ has specification at the fixed scale $\delta$''.

Finally, in the appendix we include the proof of several technical lemmas. In Appendix~\ref{s.parameter} we provide the readers with a short list of the parameters and notations that are frequently used. For convenience, we also provide the places where they first appear, as well as the dependence between certain parameters.

\section{Preliminary I: Singular flows}\label{s.p1}
In this section we collect some useful properties for singular vector fields. 

Throughout this article, $X$ is a $C^1$ vector field, and $\Lambda$ is a sectional-hyperbolic attractor.  Here an attractor is a compact invariant set $\Lambda$ that is topological transitive, such that for some neighborhood $U$ of $\Lambda$ it holds that $\Lambda = \bigcap_{t>0} f_t(U)$. We will write $\Sing(X)$ for the set of singularities of $X$, that is, points where $X(x) = 0$. Every point that is not a singularity is called a regular point, and the set of regular points is denoted by $\Reg(X)$.  We will also assume that all singularities in $\Lambda$ are hyperbolic and non-degenerate, which is a $C^1$ open and dense property.  This assumption guarantees that there are only finitely singularities.

From now on, we will consider every element of the set $\bM\times \RR^+$ as an orbit segment of the flow $(f_t)_t$ by identifying the pair $(x,t)$ with the orbit segment $\{f_s(x): s\in [0,t)\}$. This notation is consistent with~\cite{CT16} and~\cite{PYY21} on the (improved) CT criterion.

\subsection{The normal bundle and the scaled linear Poincar\'e flow}
Let $x$ be a regular point of $X$. We write 
$$
N(x)  = \langle X(x) \rangle^\perp\subset T_x\bM
$$
for the subspace of $T_x\bM$ consisting of vectors that are orthogonal to $X(x)$.  $N(x)$ is called the {\em normal plane} at the point $x$. Given $\rho>0$, we use $N_\rho(x)$ to denote the ball of radius $\rho$ about the origin, that is,
$$
N_\rho(x) = \{v\in N(x): |v|<\rho\}.
$$
The set $N = \bigsqcup_{x\in \Reg(X)} N(x)$ is called the {\em normal bundle}. It is worth noting that the normal bundle is only defined over the open set $\Reg(X)$.  Similarly, given a compact invariant set $\Lambda$, we defined the normal bundle over $\Lambda$ to be $N_\Lambda = \bigsqcup_{x\in \Lambda \cap \Reg(X)} N(x)$.

Writing $\pi_{N,x}: T_x\bM\to N(x)$ for the orthogonal projection to the normal plane $N(x)$, we define the {\em linear Poincar\'e flow} $\psi_t: N\to N$ to be the projection of the tangent flow to the normal bundle. To be more precise, given $v\in N(x)$, we let
\begin{equation}\label{e.Poin}
	\psi_t(v) = \pi_{N,x_t}\circ Df_t(v) = Df_t(v) -\frac{< Df_t(v), X(x_t) >}{\|X(x_t)\|^2}X(x_t),
\end{equation}
where $< .,. >$ is the inner product on $T_xM$ given by the Riemannian metric. The {\em scaled linear Poincar\'e flow} $\psi^*_t$ is the linear Poincar\'e flow scaled by the flow speed, that is,
\begin{equation}
	\psi_t^*(v) = \frac{\psi_t(v)}{\|Df_t\mid_{\langle X(x) \rangle}\|} = \frac{|X(x)|}{|X(x_t)|}\psi_t(v).
\end{equation}
It is easy to see that both $(\psi_t)_{t\in\RR}$ and $(\psi^*_t)_{t\in\RR}$ have the cocycle property that $\psi_{t+s} = \psi_t\,\circ\, \psi_s$. Despite the fact that $(\psi_t^*)_{t\in\RR}$ is only defined over the set of regular points which is not compact, it is still bounded in the following sense: 
\begin{lemma}\label{l.scaledflow}
	For any $\tau>0$, there exists $C_\tau>0$ such that for any $t\in[-\tau,\tau]$,
	$$
	\|\psi^*_t\|\le C_\tau.
	$$
\end{lemma}
\begin{proof}
	For each $\tau>0$ and $t\in [-\tau,\tau]$, $\|\psi_t\|$ is bounded from above and $\|Df_t\mid_{\langle X(x) \rangle}\|$ is bounded away from zero. So the upper bound of $\psi_t^*$ follows. 
\end{proof}

\subsection{Dominated splitting for $(Df_t)$, $(\psi_t)$, and $(\psi_t^*)$}\label{ss.DS}
Let $T_\Lambda\bM = E\oplus F$ be a dominated splitting on a compact invariant set $\Lambda$ under the tangent flow $(Df_t)_t$.
By replacing the vector field $X$ with $cX$ for some $c>1$ sufficiently large,\footnote{Note that speeding up or slowing down the flow (and changing the potential function $\phi$ to $|c\phi|$) does not affect the space of equilibrium states; for more details on this fact and its proof, see~\cite[Lemma 4.1]{PYY21}.} we may assume that 
\begin{equation}\label{e.onestepDS}
	\|Df_1\mid_{E(x)}\| \|Df_{-1}\mid_{F(x_t)}\| \le \frac12,
\end{equation}
that is, $E\oplus F$ is a {\em one-step} dominated splitting for the  time-one map $f_1$. Using the same argument, if $Df_t|_{F}$  is sectional-expanding in the sense of~\eqref{e.dom}, then one can again speed up the flow to achieve the following {\em one-step sectional-expansion}: for some $\overline\lambda>1$ one has 
\begin{equation}\label{e.onestepexp}
	|\det Df_1(x)|_{V_x}| \ge \overline\lambda\mbox{ for all two-dimensional subspace } V_x\subset F. 
\end{equation}
Also note that the angle between $E$ and $F$ is bounded away from zero. Therefore, using a continuous change of the Riemannian metric, we may assume that $E$ and $F$ are orthogonal at every $x\in\Lambda$.

Given  $\alpha > 0$ and $x \in \Lambda$, the {\em $(\alpha, F)$-cone} on the tangent space $T_xM$ is defined as
$$
C_\alpha(F(x)) = \{v\in T_x\bM:  v = v_E + v_F, v_E \in E, v_F\in F \mbox{ and } |v_E|\le \alpha|v_F|\}.
$$
For every $\alpha>0$, ~\eqref{e.onestepDS} implies that the cone field $C_\alpha(F(x))$, $x \in \Lambda$, is forward invariant by $Df_1$,
i.e., for any $x \in  \Lambda$, $Df_1(C_\alpha(F(x))) \subset  C_{\frac12 \alpha}(F(x_1))$. Similarly, we can define the $(\alpha, E)$-cone field $C_\alpha(E(x))$ for $x\in\Lambda$, which is invariant under $Df_{-1}$. 
By continuity, the cone fields $C_\alpha(F) $ and $C_\alpha(E) $ can be extended to invariant cone fields in a small neighborhood of $\Lambda$.  For more details and properties on dominated splittings, see for instance~\cite[Appendix B]{BDV} and~\cite[Section 2.6]{ArPa10}.

For an embedded disk $D$ centered at a point $x$ with $\dim D = \dim F$, we say that {\em $D$ is tangent to the $(\alpha,F)$-cone with diameter $\zeta>0$}, if  $\exp_{x}^{-1}(D)$ \footnote{In particular, we require that the diameter of $D$ is less than the injectivity radius of $\bM$.} is the graph of a $C^1$ function $g: F(x)\cap B_\zeta(0_x)\to E(x)$ with $\|Dg\|_{C^0}<\alpha$. Here $B_\zeta(0_x)$ is the ball in $T_x\bM$ with radius $\zeta$ centered at the origin.
Note that being tangent to the $(\alpha, F)$-cone implies that the tangent space at every point of $\exp_{x}^{-1}(D)$ is contained in the $(\alpha, F)$-cone at $x$. However, the converse is not true: there exists a disk $D$ with $\dim D = \dim F$ whose tangent subspace at every point is contained in the $(\alpha,F)$-cone, but $\exp_{x}^{-1}(D)$ is not the graph of any $C^1$ function (for example, take $D$ to be a spiral staircase).

Now we turn our attention to the (scaled) linear Poincar\'e flow. A {\em dominated splitting over an invariant set $\Lambda$ under the linear Poincar\'e flow} is a continuous splitting of the normal bundle $N_\Lambda = \Delta_1\oplus \Delta_2$ that is invariant under the linear Poincar\'e flow, together with two constants $C>0$ and $\lambda>1$, such that for every $x\in \Reg(X){\cap\Lambda}$ and $t>0$, we have 
$$
\|\psi_t\mid_{\Delta_1(x)}\| \|\psi_{-t}\mid_{\Delta_2(x_t)}\| < C\lambda^{-t}.
$$
{\em Dominated splittings under the scaled linear Poincar\'e flow} are defined similarly. Since for any $x\in\Reg(X)$ and $u,v\in N(x)$ one has 
$$
\frac{|\psi_t^*(u)|}{|\psi_t^*(v)|} = \frac{\frac{|\psi_t(u)|}{\|Df_t\mid_{\langle X(x)\rangle}\|}}{\frac{|\psi_t(v)|}{\|Df_t\mid_{\langle X(x)\rangle}\|}}=\frac{|\psi_t(u)|}{|\psi_t(v)|},
$$
we see that every dominated splitting for $(\psi_t)_t$ is a dominated splitting for $(\psi_t^*)_t$, and vice versa.

One way to obtain a dominated splitting for $\psi_t^*$  is by projecting a dominated splitting $E\oplus F$ for the tangent flow to the normal bundle. To be more precise, let $E\oplus F$ be a dominated splitting under the tangent flow, such that $X\subset F$ at every regular point (or $X\subset E$, in which case just consider $-X$). Let $E_N$, $F_N$ be the orthogonal projection of $E$ and $F$ to the normal bundle $N$, respectively. Then it is easy to see that the splitting $N = E_N\oplus F_N$ is invariant under the linear Poincar\'e flow and therefore invariant under its scaling $(\psi_t^*)_t$. 

The following lemma can be found in~\cite[Lemma 2.3]{BGY}; see also~\cite[Proposition 3.1]{CDYZ}.

\begin{lemma}\label{l.splitting}
	Let $T_\Lambda\bM = E\oplus F$ be a dominated splitting under the tangent flow with $X(x)\in F(x),\forall x\in\Lambda$. Then  the splitting $N_\Lambda = E_N\oplus F_N$ is a dominated splitting for the linear Poincar\'e flow, therefore a dominated splitting for the scaled Linear Poincar\'e flow.
\end{lemma}
\begin{proof}
	Despite our general assumption that $E$ and $F$ are orthogonal, here we shall provide a general proof without this assumption. 
	Take any $x\in\Lambda\cap \Reg(X)$ and unit vectors $u\in E_N(x)$, $v\in F_N(x)$. Since $X(y)\in F(y)$ for every $y$, we have $\psi_t(v)\in F_N(x_t)\subset  F(x_t)$. Let $u'=\pi_{N,x}^{-1}(u)\in E(x)$ and $v'=Df_{-t}\circ\psi_t(v)\in F(x)$. Since $E\oplus F$ is dominated under  $(Df_t)_t$, for $t>0$ we must have 
	\begin{equation}\label{e.scale0}
		\frac{|Df_t(u')|}{|Df_t(v')|} < C \lambda^{-t}\frac{|u'|}{|v'|},
	\end{equation}
	for some $C>0$ and $\lambda>1$. 
	
	Since $E\oplus F$ is dominated and  $X\subset F$, the angle between $E$ and $X$ must be away from zero. On the other hand, since  $N(x)$ is the orthogonal complement of $\langle X(x)\rangle$, the angle between  $E(x)$ and $N(x)$ must be away from $\pi/2$. This shows that there exists some constant $c_1\ge 1$ independent of $x$, such that for all unit vectors $u\in E_N(x)$, $|u'|=|\pi_{N,x}^{-1}(u)|<c_1$. Note also that $Df_t(v') = \psi_t(v)$. It then follows from~\eqref{e.scale0} that 
	\begin{equation}\label{e.scale1}
		\frac{|Df_t(u')|}{|\psi_t(v)|} < Cc_1 \lambda^{-t}\frac{1}{|v'|}.
	\end{equation}

	From the definition of $\psi_t$ and $\pi_{N,x}$, we have 
	$$
	\psi_t(u) = \pi_{N,x_t}\circ Df_t(u) = \pi_{N,x_t}\circ Df_t(u') \,\mbox{ and }\pi_{N,x}(v')=v,
	$$
	thus
	\begin{equation}\label{e.scale2}
		|\psi_t(u)|\le |Df_t(u')| \mbox{ and } |v'|\ge |v|=1,
	\end{equation}
	Combining (\ref{e.scale1}) and (\ref{e.scale2}) we get
	$$
	\frac{|\psi^*_t(u)|}{|\psi^*_t(v)|} = \frac{|\psi_t(u)|}{|\psi_t(v)|}< Cc_1 \lambda^{-t}.
	$$
	This shows that $E_N\oplus F_N$ is a dominated splitting under the linear Poincar\'e flow and under the scaled linear Poincar\'e flow. 
\end{proof}
\begin{remark}\label{r}
	It is worth noting that, if $E$ and $F$ are orthogonal with $X\subset F$, then $E_N=E$ and $F_N\subset F$. In this case one can take $u'=u$ in the previous proof, and consequently $c_1=1$. This shows that the domination constants $C$ and $\lambda$ under $(Df_t)_t$ and under $(\psi^*_t)_t$ are the same. In this case,~\eqref{e.onestepDS} implies the following one-step domination for the scaled linear Poincar\'e flow:
	\begin{equation}\label{e.onestepDS1}
		\|\psi^*_1\mid_{E_N(x)}\| \|\psi^*_{-1}\mid_{F_N(x_t)}\| \le \frac12.
	\end{equation}
\end{remark}

We conclude this section by pointing out that given $\alpha>0$, one can define the $(\alpha,F_N)$-cone field on the normal bundle by
$$
C^N_\alpha(F_N(x)) = \{v = v_{E_N} + v_{F_N}\in N(x): |v_{E_N}|\le\alpha|v_{F_N}| \}.
$$ 
Then~\eqref{e.onestepDS1} implies the forward invariance of the $(\alpha,F_N)$-cone field  under $\psi^*_1$. A similar statement holds for the $(\alpha,E_N)$-cone field. 

\subsection{Liao's scaled tubular neighborhood}\label{ss.liao}
In this section we introduce Liao's work~\cite{Liao96} on the scaled tubular neighborhood. More details can be found in~\cite{GY} and~\cite{WW}. {From now on, we always assume that all singularities are hyperbolic and non-degenerate.}

Let $X$ be a $C^1$ vector field on $\bM$, and $\mathfrak{d}_0$ be the injectivity radius of $\bM$. We define, for every $x\in\Reg(X)$,
$$
\cN(x) = \exp_x(N_{\mathfrak{d}_0}(x)).
$$
$\cN(x)$ will be referred to as the {\em normal plane at $x$ on the manifold}. It is clear that $\cN(x)$ is transverse to the flow direction at the point $x$. However, this is not true for other points on $\cN(x)$; in particular, if $x$ is sufficiently close to $\Sing(X)$, it is possible that  $\cN(x)\cap\Sing(X)\ne\emptyset$; as a result, $\cN(x)$ should not be used as a cross-section for the flow. To solve this issue, we define
$$
\cN_{\rho}(x) =\{y\in \cN(x): d_{\cN(x)}(x,y)<\rho\},
$$
where $d_{\cN(x)}(\cdot,\cdot)$ is the distance within the submanifold $\cN(x)$. Below we will prove that for $\rho_0>0$ sufficiently small, $\cN_{\rho_0|X(x)|}(x)$ is transversal to the flow direction at every point, and consequently can be used as a cross-section for the flow. The scale $\rho_0|X(x)|$ should be interpreted as a ``uniformly relative'' scale,  that is, a scale proportional to the flow speed at every regular point.

To this end, for every $x\in\Reg(X)$ and $\rho>0$, denote by 
$$
U_{\rho|X(x)|}(x)   = \left\{v+tX(x)\in T_x\bM: v\in N(x), |v|\le \rho |X(x)|, |t|<\rho\right\}
$$
the flow box in the tangent space of $x$ of relative size $\rho$. Consider the $C^1$ map:
\begin{equation}\label{e.Fx}
	F_x: U_{\rho|X(x)|}(x)\to\bM,\,\, F_x(v+tX(x)) = f_t(\exp_{x}(v)),
\end{equation} 
and note that $F_x(0)=x$. For a linear operator $A:\RR^n\to\RR^n$, define the {\em mininorm} of $A$ to be
$$
m(A) = \inf \{|A(v)|: v\in\RR^n, |v|=1\}.
$$
The next proposition states that $m(DF_x)$ and $\|DF_x\|$ are uniformly bounded.

\begin{proposition}\label{p.Fx}\cite[Proposition 2.2]{WW}
	For any $C^1$ vector field $X$ on $\bM$, there exists $\overline\rho_0>0$ such that for any regular point $x\in\Reg(X)$, the map $F_x: U_{\overline\rho_0|X(x)|}(x)\to \bM$ is an embedding whose image contains no singularity of $X$, and satisfies $m(D_pF_x)\ge\frac13$ and $\|D_pF_x\|\le 3$ for every $p\in  U_{\overline\rho_0|X(x)|}(x)$.
\end{proposition}

In particular, $F_x\left(U_{\overline\rho_0|X(x)|}(x)\right)$ contains a ball at $x$ of radius $\frac{\overline\rho_0}{3}|X(x)|$ for every $x\in\Reg(X)$.  Now let $\rho_0=\frac{\overline\rho_0}{3}$. For any $y\in B_{\rho_0|X(x)|}(x)$, there exists $v\in N_{\overline \rho_0|X(x)|}(x), t\in(-\overline\rho_0,\overline\rho_0)$ such that 
$$
y = F_x(v+tX(x)).
$$
Then we define the map 
\begin{equation}\label{e.Px1}
	\cP_x:  B_{\rho_0|X(x)|}(x)\to \cN_{\rho_0|X(x)|}(x),\,\, \cP_x(y) = F_x(v) = \exp_x(v). 
\end{equation}
In other words, $\cP_x$ projects every point $y\in B_{\rho_0|X(x)|}(x)$ to the normal plane $\cN_{\rho_0|X(x)|}(x)$ along the flow line through the point $y$. In particular, there exists a function 
$$\tau_x(y): B_{\rho_0|X(x)|}(x)\to [-\overline\rho_0,\overline\rho_0]$$
such that 
\begin{equation}\label{e.Px2}
	\cP_x(y) = f_{\tau_x(y)}(y). 
\end{equation}

Let $t\in\RR$ be fixed. Below we will consider the {\em sectional Poincar\'e map} $\cP_{t,x}(\cdot)$, which is the holonomy map from $\cN(x)$ to 
$\cN(x_t)$ (recall that $x_t=f_t(x)$) along flow lines. Note that the previous proposition implies that $\cN_{{\rho_0}|X(x)|}(x)$ is transverse to the flow direction at every point. By the continuous dependence of the flow on its initial condition, we see that for every $x\in\Reg(X)$ and $t\in \RR$ there exists $\rho_{x,t}>0$ such that the holonomy along flows lines is well-defined from $\cN_{\rho_{x,t}}(x)$ to $\cN(x_t)$. The following proposition, taken from~\cite{GY} and originally obtained by Liao~\cite{Liao96}, gives a quantitative estimate on $\rho_{x,t}$: it turns out that for fixed $t$, $\rho_{x,t}$ is also ``uniformly relative'' to the flow speed at $x$.

\begin{proposition}\label{p.tubular}\cite[Lemma 2.2]{GY}
	There exist $\rho_0>0$ and $K_0>1$, such that for every $\rho\le\rho_0$ and every regular point $x$ of $X$, the holonomy map 
	$$
	\cP_{1,x}: \cN_{\rho K_0^{-1} |X(x)|}(x) \to \cN_{\rho|X(x_1)|}(x_1)
	$$ 
	is well-defined, differentiable, and injective. 
\end{proposition}

Applying Proposition~\ref{p.tubular} recursively, we obtain
\begin{proposition}\label{p.tubular1}
	For every $\rho\le\rho_0$, every regular point $x$ of $X$ and $T\in\NN$, the holonomy map 
	$$
	\cP_{T,x}: \cN_{\rho K_0^{-T} |X(x)|}(x) \to \cN_{\rho|X(x_T)|}(x_T)
	$$ 
	is well-defined, differentiable, and injective.
\end{proposition}

Henceforth we will let $\rho_0$ be small enough such that both Proposition~\ref{p.Fx} and~\ref{p.tubular} are satisfied. 
Let us give an alternate definition for $\cP_{t,x}$ which is also useful. Fix $K>0$ sufficiently large (say, larger than $100\sup\{|X(x)|:x\in\bM\}$). It is easy to see (\cite[p.3191]{WW}) that
\begin{equation}\label{e.ballinclusion}
	f_t\left(B_{e^{-K|t|}\rho_0|X(x)|}(x)\right)\subset B_{\rho_0|X(x_t)|}(x_t). 
\end{equation}
Then one can define the sectional Poincar\'e map $\cP_{t,x}$ as
\begin{equation}\label{e.Ptx1}
	\cP_{t,x}: \cN_{e^{-K|t|}\rho_0  |X(x)|}(x) \to \cN_{\rho_0|X(x_t)|}(x_t),\,\, \cP_{t,x} = \cP_{x_t}\circ f_t,
\end{equation}
where $\cP_{x_t}$ is the projection to the normal plane $\cN_{\rho_0|X(x_t)|}(x_t)$ given by~\eqref{e.Px1} and~\eqref{e.Px2}. It is straightforward to check that this definition of $	\cP_{t,x}$ is consistent with the previous one, and $K_0$ in Proposition~\ref{p.tubular} is $e^{K}.$ 

We remark that $\cP_{t,x}(x) = x_t$; however, $\cP_{t,x}(y)$ is generally different from $y_t$ for $y\in  \cN_{\rho_0K_0^{-1}  |X(x)|}(x)$. To deal with this issue, for $x\in\Reg(X), y\in \cN_{\rho_0K_0^{-1}  |X(x)|}(x)$ and $t\in [0,1]$, we define
\begin{equation}\label{e.tau}
	\tau_{x,y}(t) = t+\tau_{x_t}(y_t)
\end{equation}
where $\tau_{x_t}(y_t)$ is given by~\eqref{e.Px2}. $\tau_{x,y}(t)$ is well-defined due to~\eqref{e.ballinclusion}. Then \eqref{e.Ptx1} implies that
\begin{equation}\label{e.tau1}
	f_{\tau_{x,y}(t)}(y)=\cP_{t,x}(y)\in \cN_{\rho|X(x_t)|}(x_t), \forall t\in [0,1].
\end{equation}
Moreover, using the implicit function theorem and Proposition~\ref{p.Fx}, we obtain the following upper bound for $|X|\cdot \left\|D_y \tau_{x,y}(1)\right\|$ uniformly in $y$ and $x$.
\begin{lemma}\label{l.tau}
	For every $x\in\Reg(X)$ and $t\in [0,1]$, $\tau_{x,y}(t)$ is differentiable as a function of $y\in \cN_{\rho_0K_0^{-1}|X(x)|}(x)$; furthermore, there exits $K_\tau>0$ such that for all $x\in \Reg(X)$ it holds
	\begin{equation}\label{e.Ktau}
		|X(x)|\cdot \sup\left\{\left\|D_y \tau_{x,y}(1)\right\| : y\in\cN_{\rho_0K_0^{-1}|X(x)|}(x) \right\}\le K_\tau.
	\end{equation}
\end{lemma}

\begin{proof} 
	The differentiability is due to the implicit function theorem. Below we shall prove~\eqref{e.Ktau}.
	
	By~\eqref{e.tau} we have 
	$$
	D_y \tau_{x,y}(1) = D_{y_1}\tau_{x_1} D_yf_1.
	$$
	By~\eqref{e.Px1} and~\eqref{e.Px2}, we can write 
	$$
	\tau_x(y) = \frac{1}{|X(x)|}|\pi_X\circ  F_x^{-1}(y)|
	$$
	where $F_x$ is the map defined by~\eqref{e.Fx}, and $\pi_X$ is the orthogonal projection to $\langle X\rangle $. Consequently,
	\begin{align*}
		|X(x)|\cdot\left\|D_y \tau_{x,y}(1)\right\| &\le  |X(x)| \frac{1}{|X(x_1)|} |D_{y_1}(\pi_{X}\circ F_{x_1}^{-1} (y_1))|\cdot \|D_yf_1\|\\
		&\le 3 \|Df_1\|_{C^0}\cdot \|Df_{-1}\|_{C^0} : = K_\tau,
	\end{align*}
	where the coefficient $3$ is due to Proposition~\ref{p.Fx}. Note that the bound is uniform in both $y\in \cN_{\rho_0K_0^{-1}|X(x)|}(x) $ and $x\in \Reg(X)$. The proof is complete. 
\end{proof}

We turn our attention to the differentiability of $\cP_{t,x}$. 
The map $\cP$ can be lifted to a map on the normal bundle using the exponential map in a natural way:\footnote{Regarding our notation: as a general rule, we use the normal math italic font $N, P$, etc. on the tangent bundle and the calligraphic font $\cN,\cP$, etc. for objects on the manifold.}
\begin{equation}\label{e.p}
	P_{t,x}: N_{\overline\rho_0K_0^{-1}|X(x)|}(x)\to N_{\overline\rho_0|X(x)|}(x),\,\,P_{t,x} = \exp_{x_t}^{-1}\circ\cP_{t,x}\circ\exp_x.
\end{equation}
We then defined the {\em scaled sectional Poincar\'e map on the normal bundle}:
\begin{equation}\label{e.p*}
	P_{t,x}^*: N_{\overline\rho_0K_0^{-1}}(x)\to N_{\overline\rho_0}(x),\,\,P_{t,x}^*(y) :=\frac{P_{t,x}(y|X(x)|)}{|X(x_t)|}.
\end{equation}

Then we have $D_x\cP_{t,x} = D_0P_{t,x} = \psi_t(x)$, and $D_0P_{t,x}^* = \psi^*_t(x)$. It should also be noted that the domain of  $P^*_{1,x}$ has a uniform size independent of $x$. The following propositions present some uniform continuity property for $DP_{t,x}$ and $DP_{t,x}^*$.
\begin{proposition}\cite[Lemma 2.3 and 2.4]{GY}\label{p.tubular3}
	The following results holds at every regular point $x\in \Reg(X)$ with all constants $a,\rho,\rho_1,K_1,K_1'$ uniformly in $x$:
	\begin{enumerate}
		\item 	$DP_{1,x}$ is uniformly continuous at a uniformly relative scale in the following sense: for every $a>0$ and $\rho\in [0,\rho_0]$, there exists $0<\rho_1<\rho$ such that if $y,z\in \cN_{\rho K_0^{-1}|X(x)|}(x)$ with $d(y,y')<\rho_1 K_0^{-1} |X(x)|$, then we have
		$$
		\|D_y\cP_{1,x} - D_{z}\cP_{1,x}\|<a.
		$$
		\item $DP^*_{1,x}$ is uniformly continuous at a uniform scale (not just relative!) in the following sense: for every $a>0$ there exists $\overline\rho_1>0$ such that for $y,z\in N(x)$, if $d(y,z)< \overline \rho_1$ then 
		$$
		\|D_yP^*_{1,x} - D_{z}P^*_{1,x}\|<a.
		$$
		\item  there exists $K_1>0$ such that 
		$$
		\|D\cP_{1,x}\|\le K_1,\,\, \mbox{ and }\,\,	\|D(\cP_{1,x})^{-1}\|\le K_1.
		$$
		\item there exists $K_1'\ge K_1$ such that  $\|DP^*_{1,x}\|\le K'_1$.
	\end{enumerate}
\end{proposition}
As a direct corollary of the previous proposition, we see that $\cP_{1,x}$  has a (uniformly relatively) large image.
\begin{lemma}\label{p.tubular4}
	For every $\rho\in(0,\rho_0 K_0^{-1}]$, there exists $\rho'\in(0,\rho]$ such that for every  regular point $x$, we have 
	$$
	\cP_{1,x}\left(\cN_{\rho |X(x)|}(x)\right)\supset \cN_{\rho'|X(x_1)|}(x_1).
	$$
\end{lemma}

Now we define Liao's scaled tubular neighborhood of an orbit segment. 
\begin{definition}\label{d.liao}
	For a regular point $x$, $t\in (0,\infty]$ and $\rho>0$, we define
	\begin{equation}\label{e.tubular1}
		B^*_{\rho} (x,t) = \bigcup_{s\in[0,t]} \cN_{\rho |X(x_s)|}(x_s)
	\end{equation}
	to be the {\em $\rho$-scaled tubular neighborhood} of the orbit segment $(x,t)$.\footnote{As usual, the star on the upper-right corner indicates that $	B^*_{\rho} (x,t)$ is defined with the scaled flow in mind.}
	We will refer to $t$ as the {\em length}  and $\rho$ the {\em size} of this scaled tubular neighborhood.
\end{definition}

By continuity, $B^*_{\rho} (x,t)$ contains an open neighborhood of the orbit segment $(x_{\varepsilon},t-2\varepsilon)$, for every $\vep>0$. $B^*_{\rho} (x,T)$ should be interpreted as a (local) tubular neighborhood of the one-dimension embedded submanifold $(x,t)$.


\begin{figure}[h!]
	\centering
	\def\svgwidth{\columnwidth}
	\includegraphics[scale=0.7]{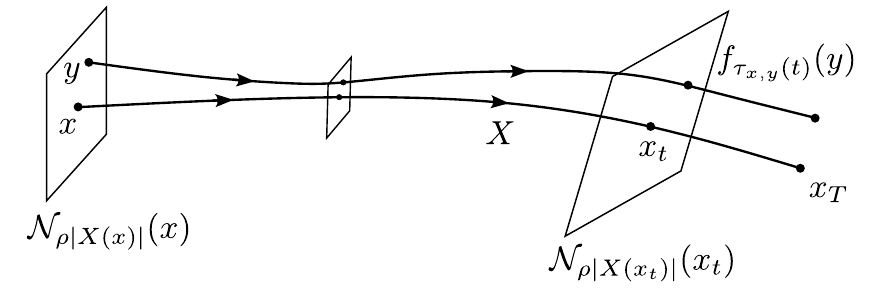}
	\caption{$\rho$-scaled shadowing. Different sizes of the normal planes $\cN_{\rho|X(x_t)|}(x_t)$ reflex the flow speed at each point $x_t$.}
	\label{f.scaled}
\end{figure}

Next, we introduce the concept of {\em $\rho$-scaled shadowing} which tracks a point $y$ as it moves inside $B_\rho^*(x,t)$; this definition is motivated by~\eqref{e.tau1}.
\begin{definition}\label{d.shadow1}
	For $0<\rho\le \rho_0$, we say that the orbit of $y$ is {\em $\rho$-scaled shadowed} by the orbit of $x$ up to  time $T\in(0,+\infty)$, if there exists a strictly increasing, continuous function
	$$
	\tau_{x,y}(t): [0,T]\to [0,\infty) 
	$$ 
	with $\tau_{x,y}(0) = 0$, such that for every $t\in[0,T]$, it holds
	$$
	f_{\tau_{x,y}(t)}(y)\in \cN_{\rho|X(x_t)|}(x_t). 
	$$
\end{definition}		
See Figure~\ref{f.scaled}. Note that the definition above requires $y\in \cN_{\rho|X(x)|}(x)$. Also, we have 
$$
f_{\tau_{x,y}(t)}(y) = \cP_{t,x}(y).
$$

\begin{remark}\label{r.shadowing}
	By Proposition~\ref{p.tubular} and Proposition~\ref{p.tubular1} and the discussion following it,  for $T>0$, if $y\in \cN_{\rho K_0^{-T}|X(x)|}(x)$, then the orbit of $y$ is $\rho$-scaled shadowed by the orbit of $x$ up to  time $T$.
\end{remark}

The following lemma collects a few useful properties for scaled shadowing. The proof directly follows from the definition and is therefore omitted.

\begin{lemma}\label{l.shadowing}
	Let $\rho\in (0,\rho_0]$, and assume that the orbit of $y$ is $\rho$-scaled shadowed by the orbit of $x$ up to  time $t$. Then the following statements hold:
	\begin{enumerate}
		
		\item For every $0\le s < s' \le t$, the orbit of $\cP_{s,x}(y)$ is $\rho$-scaled shadowed by the orbit of $x_{s}$ up to   time $s'-s$. 
		\item If, in addition, that the orbit of $\cP_{t,x}(y)$ is $\rho$-scaled shadowed by the orbit of $x_t$ up to  time $t'$, then the orbit of $y$ is $\rho$-scaled shadowed by the orbit of $x$ up to  time $t+t'$.
	\end{enumerate}
\end{lemma}

We also provide an easy way to verify the scaled shadowing property:

\begin{lemma}\label{l.shadowing2}
	Let $\rho\in (0,\rho_0]$ and $x\in\Reg(X)$. Assume that the point $y\in \cN_{\rho K_0^{-1}|X(x)| }(x)$ satisfies the following property:  for every $k\in [1,\floor{t}]\cap\NN$,\footnote{Here $\floor{t}$ denotes the integer part of $t$.} the point 
	$$
	y^{k,*} = \cP_{1,x_{k-1}}\circ\cdots\circ\cP_{1,x}(y)
	$$
	exists and is contained in $\cN_{\rho K_0^{-1}|X(x_k)|}(x_k)$.  Then the orbit of $y$ is $\rho$-scaled shadowed by the orbit of $x$ up to  time $\floor{t}+1$.
\end{lemma}
\begin{proof}
	This is a direct corollary of Remark~\ref{r.shadowing} and Lemma~\ref{l.shadowing} (2). 
\end{proof}

Finally, we remark that under the assumptions of the Lemma~\ref{l.shadowing2}, for every $s\in [0,\floor{t}+1]$, the function $\tau_{x,y}(s)$ is differential in $y$; the upper bound of  $\|D_y\tau_{x,y}(1)\|$ is of the order $\frac{1}{|X(x)|}$ due to Lemma~\ref{l.tau}.

\subsection{Sectional-hyperbolic attractors}
In this section we collect some preliminary results on sectional-hyperbolic attractors. Of particular interest are the almost expansivity and the existence of a hyperbolic periodic orbit. But before that, let us provide some characterization of singularities contained in such sets. Many of these results can be found (albeit in slightly different forms) in classical literature such as~\cite{MPP},~\cite{LGW},~\cite{GY},~\cite{GSW},~\cite{CY} and~\cite{CDYZ}. 

Let $X$ be a $C^1$ vector field and $\Lambda$ a sectional-hyperbolic attractor of $X$ with splitting $E^{ss}\oplus F^{cu}$. In particular, $\Lambda$ is transitive, and is the maximal invariant set of $X$ in a small neighborhood $U\supset \Lambda.$

We start with a basic observation regarding the flow direction at regular points, see~\cite{MPP} and~\cite{LGW}.
\begin{lemma}\label{l.basic}
	Let $\Lambda$ be a sectional-hyperbolic attractor. Then for every $x\in\Reg(X)\cap\Lambda$ one has $X(x)\in F^{cu}(x)$. Consequently, there exists a dominated splitting $E_N^s\oplus F^{cu}_N$ under the scaled linear Poincar\'e flow.  Furthermore, vectors in $E^s_N$ are exponentially contracted by the scaled linear Poincar\'e flow at a rate that is independent of $x\in\Reg(X)$.
\end{lemma}
\begin{proof}
	By~\cite[Lemma 2.6]{PYY} we have $X|_\Lambda\subset E^{ss}$ or $X|_\Lambda\subset  F^{cu}$. The first case is not possible since vectors in $E^{ss}$ are exponentially expanded (and therefore become unbounded) by the tangent flow $Df_{-t}$. The existence of the dominated splitting follows from Lemma~\ref{l.splitting}.
	
	We are left with the uniform contraction of $\psi^*_t|_{E^s_N}$. Recall our assumption that $E^{ss}$ and $F^{cu}$ are orthogonal, and consequently, $E^{ss}= E^s_N$. Take a vector $v\in E^s_N(x)$ at a regular point $x$. Then we write (and note that $\langle X(x)\rangle$ is one-dimensional, so for any linear operator $A$ preserving the bundle $\langle X\rangle$, one has $m(A|_{\langle X(x)\rangle}) = \left\|A|_{\langle X(x)\rangle}\right\|$)
	\begin{align*}
		|\psi_t^*(v)|&=\frac{|\psi_t(v)|}{\|Df_t|_{\langle X(x)\rangle }\|}\\
		&\le \frac{|Df_t(v)|}{\|Df_t|_{\langle X(x)\rangle }\|}\\
		&\le \frac{\|Df_t|_{E^{ss}}\|}{\|Df_t|_{\langle X(x)\rangle }\|} |v|.
	\end{align*}
	Since $X(x)\in F^{cu}$ and $E^{ss}\oplus F^{cu}$ is a dominated splitting for the tangent flow, we obtain from~\eqref{e.dom} that 
	$$
	|\psi_t^*(v)|\le C\lambda_0^{-t}|v|,
	$$
	as desired.
	
\end{proof}

Next, we provide some information on the hyperbolicity of singularities contained in $\Lambda$. Recall that we assume in Theorem~\ref{m.B} that all singularities are hyperbolic, which is a $C^1$ open and dense property in $\mathfrak X^1(\bM)$.

\begin{definition}\label{d.lorenzlike}
	Let $\sigma\in\Sing(X)$ with hyperbolic splitting $E_{\sigma}^s\oplus E^u_\sigma$. 
	We say that $\sigma$ is a {\em Lorenz-like singularity}, if $E^{s}$ can be further split into the direct sum of two subspaces  $E^{ss}_\sigma\oplus E^c_\sigma$ with $\dim E^c_\sigma=1$, such that the splitting $T_\sigma\bM = E_\sigma^{ss}\oplus \left(E^c_\sigma\oplus E^u_\sigma\right)$ is a dominated splitting under the tangent flow. 
\end{definition}
The following result is well-known for sectional-hyperbolic attractors. See for instance~\cite{MPP} for 3-flows, ~\cite[Lemma 4.3]{LGW} and~\cite[Proposition 2.4]{CY} for their higher-dimension counterparts.

\begin{proposition}\label{p.sh}
	Let $\Lambda$ be a sectional-hyperbolic attractor with splitting $E^{ss}\oplus F^{cu}$. Then every hyperbolic singularity in $\Lambda$ is Lorenz-like, with $E_\sigma^{ss} = E^{ss}(\sigma)$ and $E^c_\sigma\oplus E^u_\sigma = F^{cu}(\sigma)$. Furthermore, we have 
	$$
	\left(W^{ss}(\sigma)\cap\Lambda\right)\setminus\{\sigma\}=\emptyset,
	$$
	where $W^{ss}(\sigma)$ is the strong stable manifold of $\sigma$ tangent to $E_\sigma^{ss}$.
\end{proposition}

Below we will prove that when a regular orbit in $\Lambda$ approaches some $\sigma \in\Lambda$, it can only do so along the $E^c_\sigma$ direction. To rigorously formulate this phenomenon, we introduce the {\em extended linear Poincar\'e flow} which was first developed in~\cite{LGW}. Roughly speaking, it is the lift of the linear Poincar\'e flow to the Grassmannian manifold $G^1$, and allows one to capture the limiting direction of the flow when an orbit approaches a singularity. We remark that this concept will not be used beyond this subsection, so uninterested readers can safely skip it.

Denote by 
$$
G^1 = \{L_x: L_x \mbox{ is a 1-dimensional subspace of }T_x\bM, x\in \bM\}
$$
the Grassmannian manifold of $\bM$. Given a $C^1$ flow $(f_t)_t$, the tangent flow $(Df_t)_t$ acts naturally on $G^1$ by mapping each $L_x$ to $Df_t(L_x)$. 

Write $\beta:G^1\to \bM$ and $\xi: T\bM\to \bM$ the bundle projection. The pullback bundle of $T\bM$:
$$
\beta^*(T\bM) = \{(L_x,v)\in G^1\times T\bM: \beta(L)= \xi(v)=x\} 
$$
is a vector bundle over $G^1$ with dimension $\dim \bM$. The tangent flow $(Df_t)_t$ lifts naturally to $\beta^*(TM)$:
$$
Df_t(L_x,v) =(Df_t(L_x),Df_t(v)).
$$

Recall that the linear Poincar\'e flow $\psi_t$  is the projection of the tangent flow to the normal bundle $N$. The key observation is that this projection can be defined not only w.r.t the normal bundle but to the orthogonal complement of any section $\{L_x:x\in \bM\}\subset G^1$.

To be more precise, given $L = \{L_x:x\in \bM\}$ we write 
$$
N^L = \{(L_x,v)\in \beta^*(TM): v \perp L_x\}.
$$
Then $N^L$, consisting of vectors perpendicular to $L$, is a sub-bundle of $\beta^*(T\bM)$ over $G^1$ with dimension $\dim \bM-1$. The {\em extended linear Poincar\'e flow} is then defined as 
$$
\psi_t = \psi_{t}^L: N^L\to N^L,\,\,
\psi_t(L_x,v) = \pi(Df_t(L_x,v)), 
$$
where $\pi$ is the orthogonal projection from fibres of $\beta^*(TM)$ to the corresponding fibres of $N^L$ along $L$.

If we define the map 
$$
\zeta: \Reg(X)\to G^1,\,\, \zeta(x) = \langle X(x)\rangle,
$$
i.e., $\zeta$ maps every regular point $x$ to the unique $L_x\in G^1$ with  $\beta(L_x)=x$ such that $L_x$ is generated by the flow direction at $x$, then the extended linear Poincar\'e on $N^{\zeta(\Reg(X))}$ can be naturally identified with the linear Poincar\'e flow defined earlier. On the other hand, given any invariant set $\Lambda$ of the flow, consider the set: 
$$
\mathfrak B(\Lambda) = \overline{\zeta(\Lambda\cap \Reg(X))}.
$$ 
In other words, $\mathfrak B(\Lambda) $ consists of those directions in $G^1$ that can be approximated by the flow direction of regular points in $\Lambda$. 
If $\Lambda$ contains no singularity, then $\mathfrak B(\Lambda) $ can be seen as a natural copy of $\Lambda$ in $G^1$ equipped with the direction of the flow on $\Lambda$. If $\sigma\in\Lambda$ is a singularity, then $ \mathfrak B(\Lambda) $ contains all the directions in $\beta^{-1}(\sigma)$ that can be approximated by the flow direction as the orbit of regular points in $\Lambda$ approach $\sigma$. 
This motivates us to define, for $\sigma\in\Lambda\cap\Sing(X)$,
$$
\mathfrak B_\sigma(\Lambda) = \{L\in \mathfrak B(\Lambda): \beta(L)=\sigma\}.
$$

The following lemma is obtained from the proof of~\cite[Lemma 4.4]{LGW}. The proof only uses the fact that every $\sigma\in\Lambda\cap\Sing(X)$ is Lorenz-like.

\begin{lemma}\label{l.lgw}
	Let $\Lambda$ be a sectional-hyperbolic attractor for a $C^1$ vector field $X$. Then for every $\sigma\in\Lambda\cap\Sing(X)$ and every $L\in\mathfrak B_\sigma(\Lambda)$, one has
	$$
	L\subset E^c_\sigma\oplus E^u_\sigma=F^{cu}(\sigma).
	$$
\end{lemma}

\begin{proof}

	Due to the dominated splitting $E^{ss}_\sigma\oplus \left(E^c_\sigma\oplus E^u_\sigma\right)$, if there exists $L\in B_\sigma(\Lambda)$ that is outside  $E^c_\sigma\oplus E^u_\sigma$, then the backward iteration of $L$ under the extended linear Poincar\'e flow will accumulate on $E^{ss}_\sigma$. Since $\mathfrak B_\sigma(\Lambda)$ is closed, one can find some $\overline L\in \mathfrak B_\sigma(\Lambda)$ that is contained in $E^{ss}_\sigma$.
	Therefore we only need to show that there is no such line $L\in \mathfrak B_\sigma(\Lambda)$ that is in $E^{ss}_\sigma$. 
	
	Suppose that this is not the case, that is, one can find a sequence of points $x^n\in \Lambda$ with $x^n\to\sigma$, such that $\langle X(x^n)\rangle\to L$ with $L\subset E^{ss}_\sigma$.\footnote{Here we use superscripts to denote a sequence of points, and $x^n_t$ should be interpreted as $(x^n)_t = f_t(x^n)$. } 
	Recall that the cone $C_\alpha(F^{cu}(\sigma))$ is forward invariant by $Df_1$, that is,
	$$
	Df_1\left(C_\alpha(F^{cu}(\sigma))\right)\subset C_{\frac12\alpha}(F^{cu}(\sigma)).
	$$
	The cone can be extended continuously to a cone field in a small ball $U(\sigma)$ (which depends on $\alpha$) centered at $\sigma$, such that $$
	Df_1\left(C_\alpha(F^{cu}(x))\right)\subset C_{\frac23\alpha}(F^{cu}(x_1))
	$$
	for every $x$ such that the orbit segment $(x,1)$ is in $U(\sigma)$.
	
	Now we define 
	$$
	t_n = \sup\{t>0: (x^n_{-t}, 0)\subset U(\sigma)\}.
	$$
	Then $t_n\to\infty$, and $y^n:=x^n_{-t_n}\in\partial U(\sigma)\cap\Lambda$. Taking subsequence if necessary, we have $y^n\to y\in \partial U(\sigma)\cap\Lambda$ which is a regular point if we take $U(\sigma)$ to be small enough such that $\overline{U(\sigma)}$ contains no other singularity. By continuity, we have $y_t\in U(\sigma),\forall t>0$. This implies that $y\in W^s(\sigma)$. Then by Proposition~\ref{p.sh}, we have 
	$$
	y\in W^s(\sigma)\setminus W^{ss}(\sigma).
	$$
	
	By the dominated splitting $E^{ss}_\sigma\oplus \left(E^c_\sigma\oplus E^u_\sigma\right)$ we see that 
	$$
	\langle X(y_t)\rangle\to E^c_\sigma, \mbox{ as }t\to\infty.
	$$
	In particular, for every $\alpha>0$, one can find $T>0$ such that 
	$$
	X(y_T)\in C_\alpha(F^{cu}(y_T)).
	$$
	This shows that for all $n$ sufficiently large, one has 
	$$
	X(y^n_T)\in C_\alpha(F^{cu}(y^n_T)),
	$$
	and, by the invariance of $C_\alpha(F^{cu}(y^n_T))$ under $Df_1$,
	$$
	X(x^n) = X(y^n_{t_n})\in C_{\alpha\left(\frac23\right)^{t_n-T}}(F^{cu}(x^n)).
	$$
	This contradicts with the assumption that $\langle X(x^n)\rangle\to L\subset E^{ss}_\sigma$.
	
\end{proof}

We remark that a similar result hold on the forward orbit of $y\in\Lambda$ as opposed to the flow direction $X(y_t)$. We define 
$$
C_\alpha(E^c(\sigma)) = \{v\in T_\sigma\bM, v=v^{ss}+v^c+v^u, \max\left\{|v^{ss}, v^u|\right\}\le\alpha |v^c| \}
$$
(note that this cone is not invariant under the tangent flow unless one further  requires  that $v^u=0$)
and consider 
$$
\cC_{\alpha}(E^c(\sigma)) = \exp_{\sigma}\left(C_\alpha(E^c(\sigma))\right)
$$
its image under the exponential map. $\cC_{\alpha}(E^c(\sigma))$ can be considered as a {\em cone on the manifold} around the center direction of $\sigma$ despite that $\sigma$ may have no center manifold. The domination between $E^{ss}_\sigma$ and $E^c_\sigma$ implies that this cone has certain invariance property under $f_1$, in the sense that 
$$
f_1\left(\cC_{\alpha}(E^c(\sigma))\cap W^s(\sigma)\right)\subset \cC_{\theta\alpha}(E^c(\sigma)\cap W^s(\sigma)),\mbox{ for some }\theta\in(0,1). 
$$
Then a similar argument as in the previous lemma shows that for some $T>0$ and every $t>T$, $y_t\in	\cC_{\alpha}(E^c(\sigma))\cap W^s(\sigma)$. This observation leads to the following lemma. 

\begin{lemma}\label{l.nearEc}
	For every $\alpha>0$, there exist $r_0>0$ and $\overline r\in (0,r_0)$, such that for every $r\in(0,\overline r]$ and every $x\in B_{r}(\sigma)\cap\Lambda$, letting $t_x= \sup\{t>0: (x_{-t}, 0)\subset B_{r_0}(\sigma)\},$ then one has
	$$
	x_{-t_x} \in \partial B_{r_0}(x)\cap \cC_{\alpha}(E^c(\sigma)).
	$$

\end{lemma}

In other words, this lemma states that if $x$ is in $B_{r_0}(\sigma)\cap\Lambda$ such that its orbit enters the smaller ball $B_r(\sigma)$, then the orbit of $x$ can only cross the boundary of the larger ball $B_{r_0}(\sigma)$ through the center cone $\cC_{\alpha}(E^c(\sigma))$. 

\begin{proof}The proof is a modified version of the previous proof. Fix $\alpha>0$ and define $U(\sigma)$ as before. For every $r_0$ with $B_{r_0}(\sigma)\in U(\sigma)$ and every $x\in B_{r_0}(\sigma)\cap\Lambda$, define
	$\tilde t_x = \sup\{t>0: (x_{-t}, 0)\subset U(\sigma)\}.$ Then by  Proposition~\ref{p.sh} there exists an open (under the relative topology of $\partial U(\sigma)$) neighborhood $V$ with 
	$$
	W^{ss}(\sigma)\cap \partial U(\sigma)\subset V\subset \partial U(\sigma),
	$$
	such that $x_{-\tilde t_x}\in \partial U(\sigma)\setminus V$, for all $x\in B_{r_0}(\sigma)\cap\Lambda$.
	
	Having fixed $V$ and $\alpha$,  we take $T>0$ large enough with the following property: if $y\in \left(\partial U(\sigma)\cap W^s(\sigma) \cap\Lambda\right)\setminus V$, then for every $t\ge T$ it holds that 
	$$
	y_t\in W^s(\sigma)\cap \cC_{\frac12\alpha}(E^c(\sigma)).
	$$Such a $T$ exists due to the domination between $E^{ss}_\sigma$ and $E^c_\sigma$ and the compactness of $\partial U(\sigma)\setminus V$.
	Then we pick $r_0>0$ small enough such that every point in $\partial U(\sigma)$ must spend at least time $T$ before entering $B_{r_0}(\sigma)$.  
	
	Now, assume for contradiction's sake that for such a $r_0>0$ there exists a sequence of points $x^n\to\sigma$ such that for all $n$,
	\begin{equation}\label{e.t1}
		x^n_{-t_{x^n}} \in \partial B_{r_0}(x)\setminus \cC_{\alpha}(E^c(\sigma)).
	\end{equation}
	Note that $0<t_{x^n}<\tilde t_{x^n}$. Let $y$ be a limit point of the sequence $x^n_{-\tilde t_{x^n}}$, then $y\in \left(\partial U(\sigma)\cap W^s(\sigma) \cap\Lambda\right)\setminus V$. Let $s=\max \{t>0: y_t\in \partial B_{r_0}(\sigma)\}$ to be the last time that the orbit of $y$ crosses $\partial B_{r_0}(\sigma)$; such an $s$ exists because $y\in W^s(\sigma)$. By the choice of $r_0$, we have $s> T$.  This implies that 
	$$
	y_s\in \partial B_{r_0}(\sigma)\cap \cC_{\frac12\alpha}(E^c(\sigma));
	$$
	by continuity, we must have 
	$$
	x^n_{-t_{x^n}}\in \partial B_{r_0}(\sigma)\cap \cC_{\alpha}(E^c(\sigma))
	$$
	for all $n$ large enough. This contradicts the choice of $x^n$. 
\end{proof}

Next, we state some results on the topological structure of sectional-hyperbolic attractors. We start with the following definition.

\begin{definition}
	Given a vector field $X$, $x\in\bM$ and $\vep>0$, we define the {\em bi-infinite Bowen ball} to be 
	$$
	\Gamma_\vep{ (x) }= \{y\in\bM: d(x_t,y_t)\le \vep, \forall t\in\RR\}. 
	$$
	We say that $X$ is almost expansive at scale $\vep>0$, if the set 
	$$
	\Exp(\vep): = \{x\in\bM: \Gamma_\vep(x)\subset (x_{-s},2s) \mbox{ for some } s=s(x)>0\} 
	$$
	satisfies $\mu(\Exp(\vep))=1$ for every $X$-invariant ergodic  probability measure $\mu$.
	
\end{definition}

\begin{theorem}\label{t.PYY}\cite[Proposition 3.8 and 3.9]{PYY}
	Let $\Lambda$ be a sectional-hyperbolic chain recurrence class for a $C^1$ vector field $X$  with all singularities hyperbolic and non-degenerate. Then there exists $\vep>0$ such that $X$ is almost expansive at scale $\vep$ on the maximal invariant set in a small neighborhood $U$ of $\Lambda$.
\end{theorem}

By~\cite[Proposition 2.4]{LVY}, $X|_\Lambda$ is entropy expansive at the same scale. Due to the classical work of Bowen~\cite{B72}, the metric entropy, as a function of invariant probability measures, varies upper semi-continuously under the weak* topology. 
This further implies that every continuous function has at least one equilibrium state. Here we will not give the precise definition of entropy expansivity, since later we will state the CT criterion which provides both the existence and the uniqueness of such an equilibrium state for H\"older continuous functions.  

The following theorem is taken from~\cite[Theorem C]{PYY}.
\begin{theorem}\label{t.positiveentropy}
	Let $\Lambda$ be a sectional-hyperbolic attractor for a $C^1$ vector field $X$. Then $X|_\Lambda$ has positive topological entropy.
\end{theorem}
Combined with the previous theorem, we have the existence of a measure of maximal entropy for $X|_ \Lambda$. 

We conclude this section with the following theorem concerning the topological structure of $\Lambda$. This result is the reason why our main result is stated for $C^1$ open and dense vector fields. Note that since the bundle $E^{ss}$ is uniformly contracted by the tangent flow, there exists a stable foliation $\cF^s$. For $K>0$, write $\cF^{s}_K(x)$ for the disk with radius $K$ inside the stable leaf $\cF^{s}(x)$ centered at $x$.\footnote{Note that $\dim W^s(\sigma) = \dim \cF^{s}(x) + 1 = \dim E^{ss}+1$ for every regular point $x$ and singularity $\sigma\in\Lambda\cap\Sing(X)$. We also have $W^{ss}(\sigma) = \cF^s(\sigma)$. }

\begin{theorem}\cite[Theorem A and B]{CY}\label{t.top}
	There exists an open and dense set $\cU\subset \mathfrak{X}^1(\bM)$ such that for any $X\in\cU$, any sectional-hyperbolic Lyapunov stable chain-recurrence class $\Lambda$ of $X$ (not reduced to a singularity) satisfies:
	\begin{enumerate}
		\item $\Lambda$ is a robustly transitive attractor;
		\item $\Lambda$ is a homoclinic class of a hyperbolic periodic orbit $\gamma$; in particular, the stable manifold of $\gamma$ is dense in a small neighborhood of $\Lambda$;
		\item for every $x\in \Reg(X)\cap \Lambda$, there exists $K^s(x)>0$ such that 
		$$
		\cF^{s}_{K^s(x)}(x)\pitchfork W^u(\gamma)\ne\emptyset.
		$$
	\end{enumerate} 
\end{theorem}

\begin{proof}
	The existence of a hyperbolic periodic orbit is given by~\cite[Theorem B]{PYY}; this result holds for all sectional-hyperbolic invariant sets with positive topological entropy (not necessarily an attractor), and in particular, applies to sectional-hyperbolic Lyapunov stable chain recurrence classes due to~\cite[Theorem C]{PYY}. 
	
	Item (1) and (2) are the main results of~\cite{CY}. We only need to prove (3). For this purpose, let $x\in \Reg(X)\cap\Lambda$. Since having transverse intersection is an open property, we only need to prove that there exists a point $y\in \omega(x)\subset \Lambda$, where $\omega(x)$ is the $\omega$-limit set of $x$, such that $\cF^{s}(y)\pitchfork W^u(\gamma)\ne\emptyset$. There are two cases:
	
	\noindent {Case 1.} $\omega(x)\cap \Sing(X)=\emptyset$. Since $\omega(x)$ is compact and invariant, by~\cite[Lemma 3.4]{CY}, the stable manifold of every point in $\omega(x)$ has a transverse intersection point with $W^u(\gamma)$.
	
	\noindent {Case 2.} $\omega(x)\cap \Sing(X)\ne \emptyset$. Take a singularity $\sigma \in \omega(x)$, and a sequence of times $t_n\to\infty$ such that $f_{t_n}(x)\to \sigma$. Now fix $r>0$ small enough and take, for all $n$ large enough, 
	$$
	t_n' = \min \{0<s<t_n: (f_{s}(x), t_n-s) \in B_r(\sigma)\}. 
	$$ 
	Let $z$ be a limit point of the sequence $\{f_{t_n'}(x)\}$, then we have $z\in\omega(x)$ and $z\in  \Lambda\cap W^s(\sigma)\setminus\{\sigma\}$ since $t_n-t_n'\to\infty$. Now~\cite[Lemma 3.2]{CY} shows that $\cF^{s}(z)$ has a transverse intersection point with $W^u(\gamma)$, as desired. 
	
\end{proof}

We remark that the third item is sharp in the sense that if $\sigma$ is a singularity in $\Lambda$, then $\cF^s(\sigma) = W^{ss}(\sigma)$; also note that since $\Lambda$ is an attractor, we must have $W^u(\Orb(\gamma))\subset \Lambda$.  Then Proposition~\ref{p.sh} shows that 
$$
(\cF^s(\sigma)\setminus \{\sigma\})\cap W^u(\Orb(\gamma))=\emptyset.
$$

\section{Preliminary II: An improved Climenhaga-Thompson criterion}\label{s.p2}
In this section we provide the improved Climenhaga-Thompson Criterion. For the convenience of those who are familiar with~\cite{CT16} or~\cite{PYY21}, the notations here are the same. So it is safe to skip most of this Section and move on to Theorem~\ref{t.improvedCL}.

\subsection{Pressure on a collection of orbit segments}\label{s.2.1}
Throughout this article, we will assume that $\bM$ is a compact metric space and $f_t:\bM\to \bM$ a continuous flow generated. Denote by $\cM(\bM)$ the set of Borel probability measures on $\bM$ that are invariant.
For $t>0$, the {\em Bowen metric} is defined as 
$$
d_t(x,y) = \sup\{d(f_s(x),f_s(y)): 0\le s\le t\}.
$$
For $\delta>0$, the $(t,\delta)$-Bowen ball at $x$ is the $\delta$-ball under the Bowen metric $d_t$
$$
B_{t,\delta}(x) = \{y\in \bM: d_t(x,y)<\delta\}, 
$$
and its closure is given by
$$
\overline{B}_{t,\delta}(x)  = \{y\in \bM: d_t(x,y)\le\delta\}.
$$

Given $t>0$ and $\delta>0$, a set $E\subset \bM$ is call {\em $(t,\delta)$-separated}, if for every distinct $x,y\in E$, one has $d_t(x,y)>\delta$.

A core concept in the work of Climenhaga and Thompson is the pressure on a collection of orbit segments. Writing $\RR^+ = [0,\infty)$, we regard $\bM\times \RR^+$ as the collection of finite orbit segments. Given $\cC\subset \bM\times \RR^+$ and $t\ge 0$ we write $\cC_t = \{x\in \bM: (x,t)\in \cC\}$.

Given a continuous function $\phi:\bM\to\mathbb{R}$ which will be called a {\em potential function} and a scale $\vep>0$, we write 
\begin{equation}\label{e.Phi}
	\Phi_\vep(x,t) = \sup_{y\in B_{t,\vep}(x) }\int_0^t\phi(f_s(y))\,ds,
\end{equation}
with $\vep = 0$ being the standard Birkhoff integral
$$
\Phi_0(x,t) =\int_0^t\phi(f_s(y))\,ds.
$$
Putting $\Var(\phi,\vep) = \sup\{|\phi(x)-\phi(y)|: d(x,y)<\vep\}$, we obtain the trivial bound
$$
|\Phi_\vep(x,t) - \Phi_0(x,t)|\le t\Var(\phi,\vep).
$$

The {\em (two-scale) partition function} $\Lambda(\cC,\phi,\delta,\vep,t)$ for $\cC\in \bM\times \RR^+, \delta>0,\vep>0, t>0$ is defined as 
\begin{equation}\label{e.Lambda}
	\Lambda(\cC,\phi,\delta,\vep,t)=\sup\left\{\sum_{x\in E}e^{\Phi_\vep(x,t)}:E\subset \cC_t \mbox{ is $(t,\delta)$-separated}\right\}.
\end{equation}
Henceforth, we will suppress the potential function $\phi$ and write $\Lambda(\cC,\delta,\vep,t)$  when no confusion is caused.
When $\cC = \bM\times \RR^+$ we will also write $\Lambda(\bM,\delta,\vep,t)$. A $(t,\delta)$-separated set $E\in\cC_t$ achieving the supremum in~\eqref{e.Lambda} is called {\em maximizing} for $\Lambda(\cC,\delta,\vep,t)$. Note that the existence of such set is only guaranteed when $\cC_t$ is compact. Also note that $\Lambda$ is monotonic in both $\delta$ and $\vep$, albeit in different directions: if $\delta_1<\delta_2$, $\vep_1<\vep_2$ then
\begin{equation}\label{e.monotone}
	\begin{split}
		\Lambda(\cC,\delta_1,\vep,t)\ge\Lambda(\cC,\delta_2,\vep,t),\\
		\Lambda(\cC,\delta,\vep_1,t)\le\Lambda(\cC,\delta,\vep_2,t).
	\end{split}
\end{equation}

The pressure of $\phi$ on $\cC$ with scales $\delta,\vep$ is defined as
\begin{equation}\label{e.P1}
	P(\cC,\phi,\delta,\vep) = \limsup_{t\to+\infty}\frac1t \log\Lambda(\cC,\phi,\delta,\vep,t).
\end{equation}
The monotonicity of $\Lambda$ can be naturally translated to $P$. Note that when $\vep=0$, $\Lambda(\cC,\phi,\delta,0,t)$ and $P(\cC,\phi,\delta,0)$ agree with the classical definition. In this case we will often write $P(\cC,\phi,\delta)$, and let
\begin{equation}\label{e.P2}
	P(\cC,\phi) = \lim_{\delta\to0} P(\cC,\phi,\delta). 
\end{equation}
When $\cC = \bM\times \RR^+$, this coincides with the standard definition of the topological pressure $P(\phi)$.

The variational principle for flows~\cite{BR} states that 
$$
P(\phi) = \sup_{\mu\in\cM(\bM)}\left\{h_\mu(f_1)+\int\phi\,d\mu\right\},
$$
where $h_\mu(f_1)$ is the metric entropy of the time-one map $f_1$. A measure achieving the supremum, when it exists, is called an equilibrium state for the potential $\phi$. When $\phi\equiv0 $ we have $P(\phi)=h_{top}(f_1)$~\cite{B71, BR}, and the corresponding equilibrium states are called the measures of maximal entropy.

\subsection{Decomposition of orbit segments}\label{s.2.3}
The main observation in~\cite{CT16} is that the uniqueness of equilibrium states can be obtained if the collection of ``bad orbit segments'' has small topological pressure compared to the rest of the system.  For this purpose, they define:
\begin{definition}\cite[Definition 2.3]{CT16}
	A decomposition  $(\cP,\cG,\cS)$ for $\cD\subset \bM\times \RR^+$ consists of three collections $\cP,\cG,\cS\subset \bM\times \RR^+$ and three functions $p,g,s:\cD\to \RR^+$ such that for every $(x,t)\in \cD$, the values $p=p(x,t), g=g(x,t)$ and $s=s(x,t)$ satisfy $t=p+g+s$, and
	\begin{equation}\label{e.decomp}
		(x,p)\in\cP,\hspace{0.5cm} (f_px,g)\in\cG,\hspace{0.5cm} (f_{p+g}x,s)\in\cS.
	\end{equation}
	Given a decomposition $(\cP,\cG,\cS)$ and real number $M\ge0$, we write $\cG^M$ for the set of orbit segments $(x,t)\in \cD$ with $p< M$ and $s< M$.
\end{definition}
We assume that $\bM\times \{0\}$ (whose elements are identified with empty sets) belongs to $\cP\cap\cG\cap\cS$. This allows us to decompose orbit segments in trivial ways. Following~\cite[(2.9)]{CT16}, for $\cC\in\bM\times\RR^+$ we define the slightly larger collection $[\cC]\supset \cC$ to be
\begin{equation}\label{e.[C]}
	[\cC]:= \{(x,n)\in \bM\times \mathbb{N}:(f_{-s}x,n+s+t) \in\cC \mbox{ for some } s,t\in[0,1)\}.
\end{equation}
This allows us to pass from continuous time to discrete time.

\subsection{Bowen property}\label{s.2.4}
The Bowen property, also known as the bounded distortion property, was first introduced by Bowen in~\cite{B75} for maps and by Franco~\cite{Franco} for flows.
\begin{definition}\label{d.Bowen}
	Given $\cC\subset  \bM\times \RR^+$, a potential $\phi$ is said to have the Bowen property on $\cC$ at scale $\vep>0$, if there exists $K>0$ such that
	\begin{equation}\label{e.Bowen}
		\sup\left\{|\Phi_0(x,t) - \Phi_0(y,t)|:(x,t)\in\cC, y\in B_{t,\vep}(x) \right\}\le K.
	\end{equation}
\end{definition}
The constant $K$ is sometimes called the {\em distortion constant}. Note that H\"older potentials for uniformly hyperbolic systems (in the case of flows, this precludes the existence of singularities) have Bowen property on $\bM\times \RR^+$. Also note that if $\phi$ has the Bowen property on $\cG$ at scale $\vep$ with distortion constant $K$, then $\phi$ has the Bowen property on $\cG^M$ at the same scale for every $M>0$, with distortion constant $K(M) = K+2M\Var(\phi,\vep)$.

\subsection{Specification}\label{s.2.5}
The specification property plays a central role in the work of Bowen~\cite{B75} and Climenhaga-Thompson~\cite{CT16}. Roughly speaking, it states that `good' orbit segments (on which there is large pressure and the Bowen property) can be shadowed by regular orbits, with bounded transition time from one segment to the next.
\begin{figure}[h!]
	\centering
	\def\svgwidth{\columnwidth}
	\includegraphics[scale=0.7]{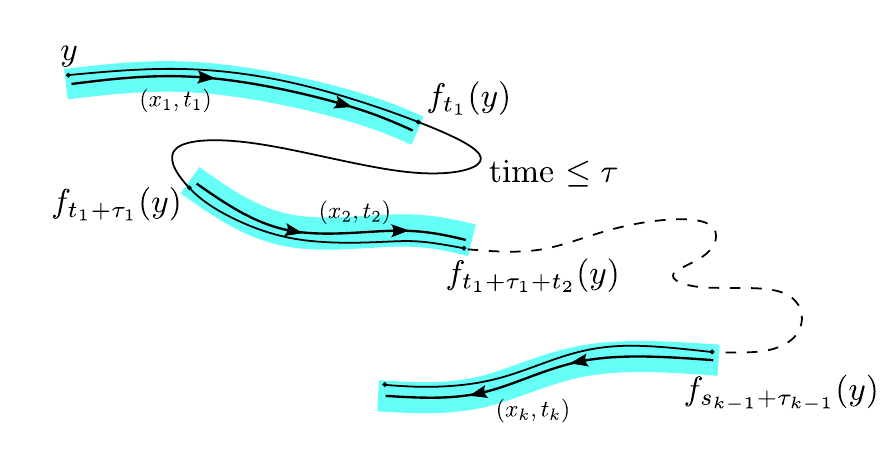}
	\caption{The specification property. The highlighted regions are the $\delta$-neighborhoods of orbit segments $(x_i,t_i)$. Note that they are not the {\bf scaled} tubular neighborhoods.}
	\label{f.specification}
\end{figure}
\begin{definition}
	We say that $\cG\subset\bM\times\RR^+$ has weak specification at scale $\delta$ if there exists $\tau>0$ such that for every finite orbit collection $\{(x_i,t_i)\}_{i=1}^k\subset \cG$, there exist a point $y$ and a sequence of ``gluing times'' $\tau_1,\ldots,\tau_{k-1}$ with $\tau_i\le \tau$ such that for $s_j = \sum_{i=1}^{j}t_i+\sum_{i=1}^{j-1}\tau_i$ and $s_0=\tau_0=0$, we have 
	\begin{equation}\label{e.spec}
		d_{t_j}(f_{s_{j-1}+\tau_{j-1}}(y), x_j)<\delta \mbox{ for every } 1\le j\le k.
	\end{equation}
	The constant $\tau = \tau(\delta)$ is referred to as the {\em maximum gap size}.  See Figure~\ref{f.specification}.
\end{definition}
As in~\cite{CT16} we will sometimes say that $\cG$ has (W)-specification, or simply specification. This version of the specification is weak in the sense that the ``transition times'' $\{\tau_i\}$ are only assumed to be bounded by, rather than equal to, $\tau$. 

It is clear that the construction above can be applied to a compact, invariant subset $\Lambda$ of $\bM$. Since we are exclusively working on the sectional-hyperbolic attractor $\Lambda$ which is isolated, we shall use the same notation as before without highlighting the dependence on $\Lambda$. In particular, we shall write $P(\phi)$ for $P(\phi,f_1|_\Lambda)$,

\begin{definition}\label{d.tailspec}
	We say that $\cG$ has {\em tail (W)-specification}  at scale $\delta$ if there exists $T_0>0$ such that $\cG\cap(\bM\times [T_0,\infty))$ has (W)-specification at scale $\delta$. We may also say that $\cG$ has (W)-specification at scale $\delta$ for $t>T_0$ if we need to declare the choice of $T_0$.
\end{definition}

\subsection{The main result of~\cite{CT16}}
The main theorem of~\cite{CT16} gives the existence and uniqueness of equilibrium states for systems with Bowen property, specification and a ``pressure gap'' on certain ``good orbits''.
\begin{theorem}\cite[Theorem 2.9]{CT16}\label{t.CT}
	Let $(f_t)_t$ be a continuous flow on a compact metric space $\bM$, and $\phi:\bM\to\RR$ a continuous potential function. Suppose that there are $\vep>0,\delta>0$ with $\vep>40\delta$ such that $(f_t)_t$ is almost expansive at scale $\vep$,\footnote{It is worth noting that the version we cited here is slightly weak than\cite[Theorem 2.9]{CT16} since we replaced the condition  $P^\perp_{\exp}(\phi,\vep)<P(\phi)$ by the assumption that $(f_t)_t$ is almost expansive at scale $\vep$.} and there exists $\cD\subset\bM\times\RR^+$ which admits a decomposition $(\cP,\cG,\cS)$ with the following properties:
	\begin{enumerate}[label={(\Roman*)}]
		\item[(I)] For every $M>0$, $\cG^M$ has tail (W)-specification at scale $\delta$;
		\item[(II)]  $\phi$ has the Bowen property at scale $\vep$ on $\cG$;
		\item[(III)] $P(\cD^c\cup [\cP]\cup[\cS], \phi,\delta,\vep)<P(\phi)$.
	\end{enumerate}
	Then there exists a unique equilibrium state for the potential $\phi$.
\end{theorem}
Note that Assumption (I) can be replaced by the following assumption:
\begin{lemma}\label{l.tailspec}\cite[Lemma 2.10]{CT16}
	Assume that \begin{center}
		(I').\,\,\,\, $\cG$ has specification at all scales.
	\end{center}		Then Assumption (I) holds. 
	
\end{lemma}

\subsection{An improved CT criterion}
As we explain in the introduction, it turns out that Theorem~\ref{t.CT} is not applicable in our case due to its strong hypothesis on the specification (I) (or (I')). To solve this issue, we developed an improved version of the CT criterion in a previous article~\cite{PYY21}. 

Let 
\begin{equation}\label{e.lip}
	Lip = \inf\{L>0: d(f_s(x),f_s(y))\le Ld(x,y), \forall s\in [-1,1]\}.
\end{equation}

\begin{theorem}\label{t.improvedCL}
	Let $(f_t)_{t\in\RR}$ be a Lipschitz continuous flow on a compact metric space $\bM$, and $\phi:\bM\to\RR$ a continuous potential function. Suppose that there exist $\vep>0,\delta>0$ with $\vep\ge1000 Lip\cdot\delta$ such that $(f_t)_t$ is almost expansive at scale $\vep$, and $\cD\subset \bM\times\RR^+$ which admits a decomposition $(\cP,\cG,\cS)$ with the following properties:
	\begin{enumerate}[label={(\Roman*)}]
		\item[(I'')] $\cG$ has tail (W)-specification at scale $\delta$;
		\item[(II)]  $\phi$ has the Bowen property at scale $\vep$ on $\cG$;
		\item[(III)] $P(\cD^c\cup [\cP]\cup[\cS], \phi,\delta,\vep)<P(\phi)$.
	\end{enumerate}
	Then there exists a unique equilibrium state for the potential $\phi$.
\end{theorem}

\begin{remark}The version cited here is slight different from~\cite[Theorem A]{PYY21}. In particular, in~\cite{PYY21} it was assumed that $(f_t)$  is continuous (as opposed to Lipschitz continuous) and $\vep \ge 2000 \delta$; in other words, $Lip = 2$. However, this was achieved by slowing down the flow, see~\cite[Section 4]{PYY21}. In this article, we cannot afford such a treatment, as we need to speed up the flow to obtain a domination constant of $\frac12$ (see~\eqref{e.onestepDS}) and a one-step sectional expansion (see~\eqref{e.onestepexp}). Note that the slowdown was only needed to obtain~\cite[Lemma 4.2]{PYY21}, namely the specification on $\cG^1$ at scale $2\delta$. However, it is straightforward to verify that this lemma holds if $(f_t)$ is Lipschitz; in this case, the scale for the specification property on $\cG^1$ is $Lip\cdot\delta$ provided that $\cG$ has the specification at scale $\delta$.    
\end{remark}

In comparison to Theorem~\ref{t.CT}, we only require the specification to hold on the ``good core'' $\cG$ at the fixed scale $\delta$. Later we will prove (Section~\ref{s.spec}) that such an assumption is verifiable if one chooses $\cG$ carefully. Indeed, we shall prove in Theorem~\ref{t.spec} that for every $\delta>0$ there exists an orbit segments collection $\cG_{spec}(\delta)$ with the tail specification property at scale $\delta$. Here the fact that $\cG_{spec}(\delta)$ depends on $\delta$ explains the reason why we cannot apply the original CT criterion.

\section{The Pliss lemma and simultaneous Pliss times}\label{s.pliss}
In this section we introduce the Pliss lemma~\cite{Pliss}. It has been proven to be a useful tool when studying systems with non-uniform behavior. Later we shall provide two applications of the Pliss lemma: on the existence of $cu$-hyperbolic times (Section~\ref{ss.cuhyptime}) and on the existence of times with ``good recurrence property'' to a small neighborhood of singularities in $\Lambda$ (Section~\ref{ss.recpliss}). The goal of this section is to prove that, under proper choices of certain parameters, these two types of Pliss times will co-exist.

\begin{theorem}[The Pliss Lemma,~\cite{Pliss}]\label{t.pliss}
	Given $A\ge b_2 > b_1 >0$, let 
	$$
	\theta_0 = \frac{b_2-b_1}{A-b_1}.
	$$
	Then, given any real numbers $a_1, a_2, \ldots, a_N$ such that 
	$$
	\sum_{j=1}^N a_j\ge b_2N, \mbox{ and } a_j\le A \mbox{ for every } 1\le j \le N,
	$$
	there exist $\ell >\theta_0 N$ and $1\le n_1 <\cdots n_\ell\le N$ so that 
	$$
	\sum_{j=n+1}^{n_i}a_j\ge b_1(n_i-n) \mbox{ for every } 0\le n < n_i \mbox{ and } i=1,\ldots,\ell.
	$$
\end{theorem}

Given a sequence $\{a_i\}_{i=1}^n$. We say that $n$ is a $b_1$-Pliss time for the sequence $\{a_i\}_{i=1}^n$, if for every $0\le j<n$ it holds 
\begin{equation}\label{e.pliss}
	\sum_{i=j+1}^{n}a_i\ge b_1(n-j). 
\end{equation}

Then the previous lemma states that the density of Pliss times is at least $\theta_0$.

The next lemma is straightforward.
\begin{lemma}\label{l.gluepliss} Assume that 
	$n$ is a $b_1$-Pliss time for the sequence $\{a_i\}_{i=1}^n$, and $m$ is a $b_1$-Pliss time for the sequence $\{a'_i\}_{i=1}^m$. Define the sequence $\{c_j\}_{j=1}^{n+m}$ as
	$$
	c_j = \begin{cases}
		a_j,& j\le n\\
		a'_{j-n},& j>n.
	\end{cases}
	$$
	Then $n+m$ is a $b_1$-Pliss time for the sequence $\{c_j\}_{j=1}^{n+m}$. 
\end{lemma} 
\begin{proof}
	It is clear that~\eqref{e.pliss} holds for the sequence $\{c_i\}$ if  $j\ge n$. If $j< n$, we write
	$$
	\sum_{i=j+1}^{n+m}c_i = 	\sum_{i=j+1}^{n}a_i + 	\sum_{i=1}^{m}a_i' \ge b_1(n-j) + b_1m = b_1(m+n-j),   
	$$ 
	and the lemma follows.
\end{proof}

\subsection{$cu$-hyperbolic times for the scaled linear Poincar\'e flow}\label{ss.cuhyptime}
First we applied the Pliss lemma on the scaled linear Poincar\'e flow. For this purpose, we use the notation $\psi^*_{t,x}$ to highlight the base point of the scaled linear Poincar\'e flow. In other words, $\psi^*_{t,x}$ is the restriction of $\psi_t^*$ to the normal space $N(x)$.
\begin{definition}
	Let $\lambda>1$. For $x\in\Reg(X){\cap\Lambda}$ and $t>1$, we say that $x_t$ is a $(\lambda,cu)$-hyperbolic time for the scaled linear Poincar\'e flow, if for every $j=1,\ldots,\floor{t}$ we have 
	\begin{equation}\label{e.cupliss}
		\prod_{i=0}^{j-1}	 \|\psi^*_{-1,x_{t-i}}\mid_{F_N^{cu}}\| \le \lambda^{-j}.
	\end{equation}
\end{definition}

The following two lemmas provide the existence of $(\lambda,cu)$-hyperbolic times for orbits that start and end outside a small neighborhood of singularities. 

\begin{lemma}\label{l.cuhyp}
	There exists $\tilde \lambda>1$ such that for any open neighborhood $W$ of $\Sing(X)\cap\Lambda$, there exists $N_0>0$ such that for every real number $t>N_0$ and any orbit segment $(x,t)\in{\Lambda}\times\RR^+$, if $x\notin W$ and $x_t\notin W$, then we have 
	$$
	\prod_{i=0}^{\floor{t}-1}	 \|\psi^*_{-1,x_{t-i}}|_{F^{cu}_N(x_{t-i})}\| \le \tilde \lambda^{-\floor{t}}.
	$$
\end{lemma}

\begin{proof}
	Let $\overline \lambda>1$ be given by~\eqref{e.onestepexp}. Take any $\tilde\lambda\in(1,\overline\lambda)$, we will prove that $\tilde\lambda$ has the desired property. 
	
	Let $W$ be any open neighborhood of $\Sing(X)\cap\Lambda$. Then we take $c>0$ large enough such that for every $x,y\in W^c\cap\Lambda$ it holds 
	\begin{equation}\label{e.bddflowspeed}
		c^{-1}<\frac{|X(x)|}{|X(y)|}<c.
	\end{equation}

	Let $x\in \Reg(X)\cap\Lambda$. Take any $v\in F^{cu}_N(x)$ and consider the rectangle $R$ in the tangent space $T_x\bM$ formed by the vectors $v$ and $X(x)$. Note that the image of $R$ under the tangent flow $Df_1$ is a parallelogram in $T_{x_t}\bM$ formed by the vectors $Df_1(v)$ and $Df_1(X(x)) = X(x_1)$. The sectional-hyperbolicity of $X|_ \Lambda$ (recalled that we speed up the flow to achieve the one-step sectional expansion on $F^{cu}$, see~\eqref{e.onestepexp}) requires that the area of $Df_1(R)$ to be larger than the area of $R$. Let $A(\cdot)$ denote the area, then we have ($\pi_{N,x_1}$ is the orthogonal projection from $T_{x_1}\bM$ to $N(x_1)$)
	\begin{align*}
		|v|\cdot |X(x)|&=A(R)\\
		&\le \overline\lambda^{-1} A(Df_1(R)) \\
		&=\overline\lambda^{-1}|\pi_{N,x_1}\left(Df_1(v)\right)|\cdot |X(x_1)|\\
		&=\overline\lambda^{-1}|\psi_{1,x}(v)|\cdot |X(x_1)|\\
		&=\overline\lambda^{-1}\frac{|\psi_{1,x}(v)| |X(x)|}{|X(x_1)|}\cdot \frac{|X(x_1)|^2}{|X(x)|}\\
		&=\overline\lambda^{-1}|\psi^*_{1,x}(v)|\frac{|X(x_1)|^2}{|X(x)|}
	\end{align*}
	Let $y=x_1$, then we have the following one-step backward contraction:
	\begin{equation}\label{e.5.3}
		\left\|\psi^*_{-1,y}|_{F^{cu}(y)}\right\|\le \overline\lambda^{-1}\frac{|X(y)|^2}{|X(y_{-1})|^2}.
	\end{equation}	
	Now let $(x,t)$ be an orbit segment starting at a regular point $x$, such that $x\notin W$ and $x_t\notin W$. Then we apply~\eqref{e.5.3} at every $x_{t-i}$, $i=0,\ldots, \floor{t}-1$ to obtain
	\begin{align*}
		&\prod_{i=0}^{\floor{t}-1}	 \|\psi^*_{-1,x_{t-i}}|_{F^{cu}_N(x_{t-i})}\| \\
		&\le\overline \lambda^{-\floor{t}}\prod_{i=0}^{\floor{t}-1} \frac{|X(x_{t-i})|^2}{|X(x_{t-i-1})|^2}\\
		&= \overline\lambda^{-\floor{t}}\frac{|X(x_{t})|^2}{|X(x_{t-\floor{t}})|^2}\\
		&\le\overline\lambda^{-\floor{t}} c^2\sup_{s\in[0,1]} \|Df_s|_{\langle X\rangle}\|,
	\end{align*} 
	where the last line is due to~\eqref{e.bddflowspeed}. 
	
	Now we see that for fixed $\tilde\lambda\in(1,\overline\lambda)$, there exists $N_0>0$ such that whenever $t>N_0$, we have 
	$$
	c^2\sup_{s\in[0,1]} \|Df_s|_{\langle X\rangle}\|\cdot \overline\lambda^{-\floor{t}}<\tilde \lambda^{-\floor{t}},
	$$
	and the lemma follows.
\end{proof}

\begin{lemma}\label{l.hyptime}
	There exist $\lambda_0>1$ and $\theta_0\in(0,1)$ such that for any open neighborhood $W$ of $\Sing(X)\cap\Lambda$, there exists $N_0>0$ such that for every real number $t>N_0$ and any orbit segment $(x,t)\in\bM\times\RR^+$, if $x\notin W$ and $x_t\notin W$ then there exist $\theta_0 \floor{t}$ many natural numbers $n_i$  such that each $x_{t-n_i}$ is a $(\lambda_0,cu)$-hyperbolic time for the orbit segment $(x,t-n_i)$ for the scaled linear Poincar\'e flow $(\psi_t^*)$. 
\end{lemma}

\begin{proof}
	Let $\lambda_{0}\in(0,\tilde \lambda)$. Let $(x,t)$ be an orbit segment satisfying the assumptions of Lemma~\ref{l.hyptime}. Then, by Lemma~\ref{l.cuhyp}, we have
	\begin{equation}\label{e.5.5}
		\prod_{i=0}^{\floor{t}-1}	 \|\psi^*_{-1,x_{t-i}}|_{F^{cu}_N(x_{t-i})}\| \le \tilde \lambda^{-\floor{t}}.
	\end{equation}

	Consider the sequence
	$$
	a_n = -\log\left\|\psi^*_{-1,x_{t-\floor{t}+n}}|_{F^{cu}_N(x_{t-\floor{t}+n})}\right\|, n=1,\ldots,\floor{t}.
	$$
	Setting $A=\log K_1$ with $K_1$ given by~Proposition~\ref{p.tubular3}, $b_2 =\log\tilde\lambda$ and $b_1=\log\lambda_0$, then we have $a_j\le A, j=1,\ldots, \floor{t}$; furthermore, \eqref{e.5.5} implies that
	$$
	\sum_{i=1}^{\floor{t}}a_i\ge \floor{t}b_2.
	$$ 
	Then, Theorem~\ref{t.pliss} gives $\theta_0>0$, and times $m_1,\ldots,m_\ell$ with $\ell>\theta_0\floor{t}$ such that for each $m_j$, one has
	$$
	\sum_{i=n+1}^{m_j}a_i\ge \floor{t}b_1, \forall n=0,\ldots, m_j-1.
	$$
	This shows that, with $n_j = \floor{t}-m_j$,
	$$
	\prod_{i=0}^{k-1}	 \|\psi^*_{-1,x_{t-n_j-i}}\| \le  \lambda_0^{-k},\,\, \forall j=1,\ldots,\floor{t_j},
	$$
	that is, each $x_{t-n_j}$ is a $(\lambda_0,cu)$-hyperbolic time for the orbit segment $(x,t-n_j)$. 
\end{proof}

\subsection{Recurrence Pliss times}\label{ss.recpliss}
Next, we shall explicitly define points with ``good recurrence property'' to a small neighborhood of $\Sing(X)\cap\Lambda$.

\begin{definition}\label{d.RecPliss}
	Given $\beta\in(0,1)$, a set $W\subset \bM$ and an orbit segment $(x,t)$ with $t\ge 1$, we say that $x_t$ is a {\em $(\beta, W)$-recurrence Pliss time for the orbit segment $(x,t)$}, if for every $s=0,1,\ldots, \floor{t}-1$, it holds 
	\begin{equation}\label{e.RecPliss}
		\frac{\# \{\tau\in[0,s]\cap\NN: x_{t-\tau}\in W\}}{s+1}\le \beta. 
	\end{equation}
\end{definition}
In other words, $x_t$ is a $(\beta, W)$-recurrence Pliss time for the orbit segment $(x,t)$, if for every $s\in [0,\floor{t}-1]\cap\NN$, the number of points $\{x_{t-s},\ldots, x_t\}$ that are contained  in $W$ has a portion that is less than $\beta$.

\begin{remark}\label{r.RecPliss}
	It should be noted that the definition above implies that $x_t\notin W$; in fact, let $s\in[0,\floor{t}-1]\cap \NN$ be the smallest integer such that $x_{t-s}\in W$, then it follows that $\frac{1}{s+1}\le \beta$, i.e., $s\ge \beta^{-1}-1$. In other words, the orbit segment $\{ x_t,x_{t-1}, \ldots,x_0\}$ (as an orbit segment for the time-$(-1)$ map $f_{-1}$) must spend at least $\beta^{-1}$ many iterates outside $W$ before entering it for the first time. 
\end{remark}

The next lemma gives the existence of $(\beta, W)$-recurrence times with density arbitrarily close to one, provided that the overall time { ratio} that the orbit segment spends in $W$ is small enough.

\begin{lemma}\label{l.rec}
	For any $\beta_0\in(0,1)$ and $\kappa\in(0,1)$, there exists $\beta_1 \in(0,\beta_0)$ such that for any set  $W\subset\bM$ and any orbit segment $(x,t)\in\bM\times\RR^+$, if 
	\begin{equation}\label{e.pliss1}
		\frac{\#\{s=0,\ldots, \floor{t}-1: x_{t-s}\in W\}}{\floor{t}} \le \beta_1,
	\end{equation}
	then there exist at least $\kappa (\floor{t}+1)$ many $n_i$'s, such that each $x_{t-n_i}$ is a $(\beta_0,W)$-recurrence Pliss time for the orbit segment $(x,t-n_i)$ and the set $W$.
\end{lemma}

\begin{proof}{
		
		For simplicity we shall assume that $t\in \NN$; otherwise consider the orbit segment $(x_{t-\floor{t}},\floor{t})$. 
		
		We take $\beta_1\in(0,1)$ small enough such that 
		\begin{equation}\label{e.beta1}
			\beta_1<\beta_0(1-\kappa)<\beta_0.
		\end{equation}
		The reason behind this choice of $\beta_1$ will be apparent soon.
		
		We then consider the $\{0,1\}$-value sequence $a_j = \id_{W^c}\circ f_j(x), k=0,\ldots, t,$ where $\id_{W^c}$ is the indicator function of $W^c$. In other words, $\sum_{j=0}^{t}a_j$ counts the number of entries of the orbit segment $(x,t)$  to $W^c$ under the time-one map $f_1$.
		From~\eqref{e.pliss1} we have 
		$$
		\sum_{j=1}^{t} a_j > (1-\beta_1 )t. 
		$$	
		Now we apply Theorem~\ref{t.pliss} with $A=1, b_2 = 1- \beta_1$, and $b_1 = 1-\beta_0$. Note that~\eqref{e.beta1} shows that $0<b_1<b_2$. Furthermore, by~\eqref{e.beta1} we have 
		$$
		\frac{b_2-b_1}{A-b_1} = \frac{(1-\beta_1) - (1-\beta_0)}{1-(1-\beta_0)} = \frac{\beta_0-\beta_1}{\beta_0} > \frac{\beta_0-\beta_0(1-\kappa)}{\beta_0} = \kappa.
		$$	
		By Theorem~\ref{t.pliss}, there exist at least $\kappa (t+1)$ many $n_i$'s, each of which satisfies 
		$$
		\sum_{j=n}^{n_i}a_j\ge (1-\beta_0)(n_i-n+1), \forall 0\le n < n_i.
		$$
		This shows that for each $n_i$ and $0\le n \le  n_i$, one has 
		$$
		\frac{\#\{\tau\in [n,n_i]: x_\tau\in W\}}{n_i-n+1} \le \beta_0,
		$$
		i.e., each $n_i$ is a $(\beta_0, W)$-recurrence Pliss time for the orbit segment $(x,n_i)$.
		
	}
\end{proof}

\subsection{Existence of simultaneous Pliss times}
The following definition will play a central role in the next section when we construct the good orbit segments collection $\cG$.
\begin{definition}\label{l.bi-Pliss}
	Given $\lambda_0>1$, $\beta_0\in(0,1)$ and a set $W\subset \bM$, we say that $x_t$ is a $(\lambda_0,cu, \beta_0,W)$-simultaneous Pliss time for the orbit segment $(x,t)$, if $x_t$ is a $(\lambda_0,cu)$-hyperbolic time for the scaled linear Poincar\'e flow, as well as a $(\beta_0, W)$-recurrence Pliss time for the orbit segment $(x,t)$.
\end{definition}

The existence of such times is given by the following lemma. For this purpose, let $\lambda_0>1$ be given by Lemma~\ref{l.hyptime}.

\begin{lemma}\label{l.bihyptime}
	For every $\beta_0\in(0,1)$ there exists $\beta_1\in (0,\beta_0)$ such that for every open neighborhood $W$ of $\Sing(X)\cap\Lambda$, there exists $N_1>0$ such that for every real number $t>N_1$ and every orbit segment $(x,t)\in\Lambda\times\RR^+$, suppose that 
	\begin{itemize}
		\item $x\notin W$, and $x_t\notin W$;
		\item it holds $$\frac{\#\{s=0,\ldots, \floor{t}: x_{t-s}\in W\}}{\floor{t}+1} \le \beta_1.$$
	\end{itemize} 
	Then there exists $s\in[0,t]\cap\NN$ such that $x_s$ is a $(\lambda_0,cu, \beta_0,W)$-simultaneous Pliss time for the orbit segment $(x,s)$ and the set $W$.
\end{lemma}

\begin{proof}
	Let $\theta_0\in(0,1)$ be given by Lemma~\ref{l.hyptime}. Then, we take $\kappa<1$ large enough such that 
	$$
	\kappa+\theta_0>1,
	$$
	and apply Lemma~\ref{l.rec} to obtain $\beta_1$. For every orbit segment satisfying both assumptions of Lemma~\ref{l.bihyptime}, Lemma~\ref{l.rec} provides $\kappa\floor{t}$ many $n_i\in \NN$ such that each $x_{t-n_i}$ is a $(\beta_0, W)$-recurrence Pliss time for the orbit segment $(x,t-n_i)$. Moreover, by Lemma~\ref{l.hyptime} we have $\theta_0(\floor{t}+1)$  many natural numbers  $k_i$ such that each $x_{t-k_j}$ is a $(\lambda_0,cu)$-hyperbolic time for the orbit segment $(x,t-k_j)$. Since $\kappa+\theta_0>1$, one of the $n_i$ must coincide with one of the $k_j$. Such a point is a $(\lambda_0,cu, \beta_0,W)$-simultaneous Pliss time.

\end{proof}

We remark that the $(\lambda_0,cu, \beta_0,W)$-simultaneous Pliss times given by the previous lemma indeed have a positive density $\kappa+\beta_0-1$, which can be made close to $\beta_0$ by taking $\kappa$ close to one (and hence decreasing $\beta_1$); however we will not use this fact in this paper.

We conclude this section with the following lemma, which is an immediate consequence of Lemma~\ref{l.gluepliss}, and therefore the proof is omitted.

\begin{lemma}\label{l.gluePliss}
\end{lemma}

\section{Proof of the main theorems}\label{s.mAproof}
We will prove the technical theorem, Theorem~\ref{m.B}, in Section~\ref{ss.6.1} to~\ref{ss.6.3} { by applying the improved Climenhaga-Thompson Criterion (Theorem \ref{t.improvedCL} to $X|_\Lambda$). } The Proof of Theorem~\ref{m.Lorenz} and Theorem~\ref{m.A} occupies Section~\ref{ss.proofAB}.

\subsection{Orbit segments with the Bowen Property and specification}\label{ss.6.1}

First we state two theorems that deal with the Bowen property and the specification property, respectively. Their proofs are postponed to Section~\ref{s.sing} and onward.
\begin{figure}[h!]
	\centering
	\def\svgwidth{\columnwidth}
	\includegraphics[scale=0.9]{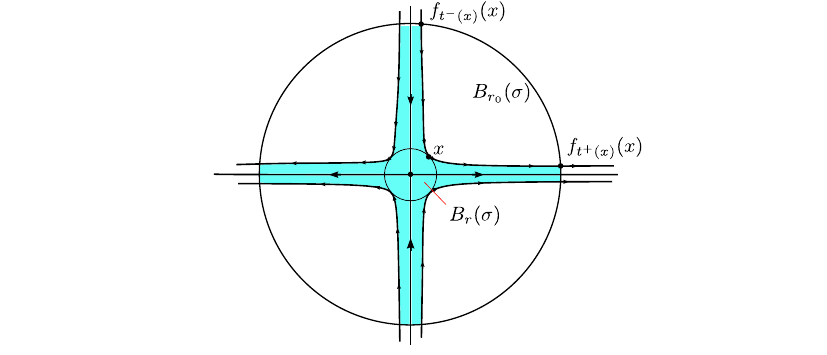}
	\caption{The neighborhood $W_r(\sigma)$ (the highlighted region).}
	\label{f.Wr}
\end{figure}

Let us start with the choice of a small neighborhood for each singularity. Fix $r_0>0$ (for the moment just assume that $r_0$ satisfies Lemma~\ref{l.nearEc} for some $\alpha>0$ small enough; the precise choice will be made clear in Section~\ref{s.sing}) and take $r\in (0,r_0)$. For each $\sigma\in\Sing(X)\cap \Lambda$ and $x\in {B_{r}(\sigma)}$, we let 
$$
t^-=t^-(x)=\sup\{s>0: (x_{-s},s)\subset \overline{B_{r_0}(\sigma)}\},\,\mbox{ and }
$$ 
$$
t^+=t^+(x)=\sup\{s>0: (x,s)\subset \overline{B_{r_0}(\sigma)}\}.
$$
Then the orbit segment $(x_{-t^-},t^-+t^+)$ is contained in the closed ball $\overline{B_{r_0}(\sigma)}$, with $x_{t^\pm}\in \partial B_{r_0}(\sigma)$. Now, we define
\begin{equation}\label{e.W.def}
	W_{r }(\sigma) = \bigcup_{x\in B_{r }(\sigma)} \{x_s:s\in(-t^-,t^+)\}.
\end{equation}
See Figure~\ref{f.Wr}. In other words, $W_{r }(\sigma)$ is a neighborhood of $\sigma$ {which is a union of finite orbit segments} and satisfies
$$
B_r(\sigma)\subset W_r\subset B_{r_0}(x).
$$
Note that $\partial B_{r_0}(\sigma)\cap \overline W_{r }(\sigma)$ consists of two pieces $\partial W_{r }^-(\sigma)$ and $\partial W_{r }^+(\sigma)$, such that if $x\in \partial W_{r }^-(\sigma)$ then there exists $x^+\in \partial W_{r }^+(\sigma)$ on the forward orbit of $x$ such that the orbit segment from  $x$ to $x^+$ is contained in $W_{r }(\sigma)$. In other words, points in $\partial W_{r }^-(\sigma)$ are moving towards $\sigma$ and points in $\partial W_{r }^+(\sigma)$ are moving away from $\sigma$. 

Clearly $W_{r }(\sigma)$ is an open neighborhood of $\sigma$. Meanwhile, define the open set 
$$
W_{r }:=\bigcup_{\sigma\in\Sing(X)\cap\Lambda}W_{r }(\sigma),
$$ then points in $(W_{r })^c$ 
are bounded away from all singularities. In particular, the flow speed on $(W_{r })^c$ is bounded away from zero, and the ``uniformly relative'' scale as in Section~\ref{ss.liao} is indeed uniform.

Given $\lambda_0>1$, $\beta\in (0,1)$ and $0<r<r_0$, we let $\cG_B=\cG_B(\lambda_0,\beta, r_0, W_{r})$ be the collection of orbit segments $(x,t)$ with $t\in\NN$ such that:
\begin{enumerate}
	\item $x_t$ is a $(\lambda_0,cu, \beta,W_r)$-simultaneous Pliss times for the orbit segment $(x,t)$ and the set $W_r$;
	\item $x_0=x\notin W_{r }$.
\end{enumerate}

Recall that $x_t$ being a $(\beta_0, W_r)$-recurrence Pliss time for the set $W_r$ already implies that $x_t\notin W_r$, due to Remark~\ref{r.RecPliss}. The following theorem provides the Bowen property on $\cG_B$; the proof can be found in Section~\ref{s.sing} and~\ref{s.bowen}.
\begin{theorem}\label{t.Bowen} { Let $\Lambda$ be a sectional-hyperbolic attractor of a $C^1$ vector field $X$ with all singularities hyperbolic. Then,} there exists $\beta_0\in(0,1)$, such that for every $\beta\in (0,\beta_0]$ and $r_0>0$ small enough, there exists $\overline r\in (0,r_0)$, such that for every $r\in (0,\overline r)$, there exists $\vep>0$ such that $\cG_B(\lambda_0,  \beta,r_0, W_r)$ has the Bowen property for every H\"older continuous function at the scale $\vep$.
\end{theorem}
Note that Bowen property at scale $\vep$ implies the Bowen property at every smaller scale.

Next, let $\delta>0$ and $r_0>0$ be fixed. For each singularity $\sigma\in\Sing(X)$, let $\CP_\sigma = \CP_\sigma(\delta)$ be a compact neighborhood of $\sigma$ inside $W^u(\sigma)$ under the relative topology, and let $W(\sigma)\subset B_{r_0}(\sigma)$ be any open neighborhood of $\sigma$.\footnote{Later we will take $W(\sigma) = W_r(\sigma)$.} Writing 
$$
\CP_\Sing(\delta) = \cup_{\sigma\in\Sing(X)} \CP_\sigma, W_\Sing = \cup_{\sigma\in\Sing(X)} W(\sigma),
$$
and letting $U_\Sing$ be an arbitrary open neighborhood of $\CP_\Sing$, we define
\begin{equation}\label{e.D1}
	\cgs(\lambda_0,r_0, W_\Sing ,U_\Sing)\subset \Lambda\times\RR^+
\end{equation}
to be the collection of orbit segments $(x,t)$ with the following properties:
\begin{enumerate}
	\item $x_t$ is a $(\lambda_0,cu)$-hyperbolic time for the orbit segment $(x,t)$ for the scaled linear Poincar\'e flow $\psi^*$;
	\item $x\notin W_\Sing \cup U_\Sing$, and $x_t\notin W_\Sing$.
\end{enumerate}

For any orbit segments collection $\cD\subset \Lambda\times\RR^+$, we denote by 
$$
\cD^{>L}\subset \cD
$$
to be the sub-collection of  those orbit segments $(x,t)\in \cD$ with $t>L$.

Then we have the following theorem on the tail specification property of $\cgs$:

\begin{theorem}\label{t.spec} {Let $\Lambda$ be a sectional-hyperbolic attractor of a $C^1$ vector field $X$ with all singularities hyperbolic. Also assume that Assumption (B) of Theorem \ref{m.B} holds. Then,} for any $\delta>0$ and $r_0>0$ small enough, and any neighborhoods $W(\sigma)\subset B_{r_0}(\sigma)$ of $\sigma$,  there exists a compact neighborhood $\CP_\sigma(\delta)\subset W^u(\sigma)$  for each $\sigma\in\Sing(X)$, such that for any neighborhood $U_\Sing$ of $\cup_\sigma \CP_\sigma(\delta)$, there exist $L>0,\tau>0$ such that the corresponding orbits collection $\cgs^{>L}(\lambda_0,r_0, W_\Sing,U_\Sing)$ as defined by~\eqref{e.D1} has the specification property at scale $\delta$ with maximum gap size $\tau$. { Furthermore, it is possible to choose the shadowing orbit to be contained in $\Lambda$.}
\end{theorem}

The proof of Theorem~\ref{t.spec} can be found in Section~\ref{s.spec}.

\begin{remark}
	Since we need to apply Theorem \ref{t.improvedCL} to $X|_\Lambda$, one must also verify that the shadowing orbit given by the specification property is also contained in $\Lambda$. In Section \ref{s.spec} we will see that the shadowing orbit is indeed taken from the unstable manifold of a periodic orbit $\Orb(p) \subset \Lambda.$ This is why we require $\Lambda$ to be an attractor.
\end{remark}

\subsection{The choice of parameters and the $(\cP,\cG,\cS)$-decomposition}
In this section we describe the choice of various parameters, most importantly the choice of $\vep$ and $\delta$. We shall also provide the construction of the $(\cP,\cG,\cS)$-decomposition.

First, we state the following lemma, which will later give us the ``pressure gap'' property.

Given any invariant measure $\mu$, we denote by 
$$
P_\mu(\phi) = h_\mu(f_1) + \int \phi\,d\mu
$$
the measure-theoretic pressure of $\mu$. 

\begin{lemma}\label{l.gap1}
	Let $\phi:\bM\to\RR$ be a H\"older continuous function satisfying Assumption (C). Then for any $b>0$, there exist $r_0>0$ and $a_0>0$ such that for any $r\le r_0$ and any invariant probability measure $\mu$ satisfying $\supp\mu\subset \Lambda$ and $\mu\left(\Cl\left(\cup_{\sigma\in\Sing(X)\cap\Lambda} B_{r}(\sigma)\right)\right)\ge b$, we have $P_\mu(\phi)<P(\phi)-a_0$. 
\end{lemma}
\begin{proof}
	Assume that this is not the case, that is, there exist $b>0$ and a sequence of measures $\{\mu_n\}$ supported on $\Lambda$ such that 
	\begin{equation}\label{e.largemeasure}
		\mu_n\left(\Cl\left(\cup_{\sigma\in\Sing(X)}B_{\frac1n}(\sigma)\right)\right)\ge b;
	\end{equation}
	furthermore, 
	\begin{equation}\label{e.largepressure}
		P_{\mu_n}(\phi)\ge P(\phi)-\frac1n.	
	\end{equation}
	Taking a subsequence if necessary, we assume that $\mu_n\to\mu$ where $\mu$ is an invariant probability measure supported on $\Lambda$. Then by~\eqref{e.largemeasure} we must have 
	$$
	\mu(\Sing(X)\cap\Lambda)\ge b.
	$$
	Moreover, since $h_\mu(X|_\Lambda)$ is upper semi-continuous with respect to invariant measures (see Theorem~\ref{t.PYY} and~\cite{PYY}), it follows from~\eqref{e.largepressure} that 
	\begin{equation}\label{e.llll}
		P_\mu(\phi)\ge \limsup_{n\to\infty} P_{\mu_n}(\phi)\ge P(\phi).
	\end{equation}
	On the other hand, writing 
	$$
	\mu = \sum_{\sigma\in\Sing(X|_\Lambda)} c_\sigma \delta_\sigma + c_0\tilde\mu
	$$
	where the coefficients satisfy $\sum_\sigma c_\sigma\ge b>0$ and $\sum_\sigma c_\sigma + c_0=1$, by~\eqref{e.pgap} we obtain 
	\begin{align*}
		P_\mu(\phi) &= \sum_{\sigma\in\Sing(X|_\Lambda)} c_\sigma P_{\delta_\sigma}(\phi) + c_0P_{\tilde \mu}(\phi)\\
		& = \sum_{\sigma\in\Sing(X|_\Lambda)} c_\sigma \phi(\sigma)+ c_0P(\tilde\mu)\\
		&< P(\phi)\left( \sum_{\sigma\in\Sing(X|_\Lambda)} c_\sigma + c_0\right)\\
		& = P(\phi),
	\end{align*}
	a contradiction with~\eqref{e.llll}.
	
\end{proof}

Let $\phi$ be a H\"older continuous function that satisfies Assumption (C). Next we describe the choice of various parameters, which consists of the following steps.
\begin{enumerate}
	\item We first pick $\beta_0\in(0,1)$ given by Theorem~\ref{t.Bowen} (the Bowen property); this is the frequency for the recurrence Pliss times;
	\item using $\beta_0$, we obtain $\beta_1\in(0,\beta_0)$ from Lemma~\ref{l.bihyptime}; $\beta_1$ dictates the maximum overall frequency that an orbit segment can spend in $W$ to guarantee the existence of simultaneous Pliss times;
	\item we then apply Lemma~\ref{l.gap1} with $b=\beta_1/2$ to obtain $r_0$, $a_0>0$ such that for any invariant probability measure $\mu$,
	\begin{equation}\label{e.pgap1}
		\mu\left(\Cl\left(\cup_{\sigma\in\Sing(X)\cap\Lambda} B_{r_0}(\sigma)\right)\right)\ge \beta_1/2\, \implies \, P(\mu) < P(\phi)-a_0;
	\end{equation}
	\item decreasing $r_0$ if necessary, we apply Theorem~\ref{t.Bowen} with $\beta = \beta_0$ to obtain $0< r<\overline r$, $\vep >0$ and an open set  $W_r = \cup_{\sigma\in\Sing(X)\cap\Lambda} W_r(\sigma)$  with $\Sing(X)\cap\Lambda\subset W_r\subset \cup_{\sigma\in\Sing(X)\cap\Lambda} B_{r_0}(\sigma)$  such that the corresponding orbits collection $\cG_B(\lambda_0, \beta_0,r_0,W_r)$ has the Bowen property at scale $\vep;$
	\item  we shall further decrease $\vep$ if necessary, such that $X|_ \Lambda$ is almost expansive at scale $\vep$ (see Theorem~\ref{t.PYY}); furthermore,
	\begin{equation}\label{e.vep}
		\sup_{x,y\in\bM, d(x,y)\le\vep} |\phi(x)-\phi(y)|< \frac{a_0}{2};
	\end{equation}
	the purpose of this step is to get rid of the second scale in the pressure $P([\cC], \phi,\delta,\vep)$;
	\item then, in view of the improved CT criterion (Theorem~\ref{t.improvedCL}), we fix (recall $Lip$ from~\eqref{e.lip})
	$$
	\delta =\frac{\vep}{1000Lip};
	$$
	$\delta$ will be the scale for the specification; we now apply Theorem~\ref{t.spec} with $W(\sigma) = W_r(\sigma)$ to get the compact neighborhoods (under the relative topology of $W^u(\sigma)$) $\sigma \in \CP_\sigma(\delta)\subset W^u(\sigma)$ for each singularity $\sigma\in\Lambda$; 
	
\end{enumerate}

Before proceeding, we state the following lemmas concerning the measure of $\CP_\sigma(\delta)$.

\begin{lemma}\label{l.D}
	For any invariant probability measure $\mu$, we have 
	$$
	\mu(\CP_\sigma(\delta)\setminus \Sing(X)) = 0.
	$$
\end{lemma}
\begin{proof}
	Note that for every $y\in \CP_\sigma(\delta)\setminus \Sing(X)$, it holds that $y_{-n}\to \sigma$ as $n\to\infty$.  Then this lemma is a simple corollary of the Poincar\'e recurrence theorem applied to the map $f_{-1}$.
\end{proof}

\begin{lemma}\label{l.U}
	For every $\upsilon>0$, there exists a open neighborhood $U_\Sing$ of $\cup_{\sigma\in\Sing(X)\cap\Lambda} \CP_\sigma(\delta)$ such that for any invariant probability measure $\mu$, we have 
	\begin{equation*}
		\mu\left(\Cl\left(U_\Sing\setminus \left( \cup_{\sigma\in\Sing(X)\cap\Lambda} B_{r_0}(\sigma)\right)\right)\right) < \upsilon.
	\end{equation*}
\end{lemma}

\begin{proof}
	Assume that this is not the case, that is, there exist $\upsilon>0$, a sequence of open sets $\{U_n\}_{n=1}^\infty$ with $U_{n+1}\subset U_n$ and $\cap_{n>0} \Cl(U_n) = \cup_\sigma \CP_\sigma(\delta)$, and a sequence of invariant probability measures $\mu_n$ such that 
	\begin{equation}\label{e.U1}
		\mu_n\left(\Cl\left(U_n\setminus \left( \cup_{\sigma\in\Sing(X)\cap\Lambda} B_{r_0}(\sigma)\right)\right)\right) \ge  \upsilon>0.
	\end{equation}
	
	Taking a subsequence if necessary, we assume that $\mu_n\to \mu$ where $\mu$ is an invariant probability measure. By~\eqref{e.U1} and the fact that $U_{n+1}\subset U_n$, we see that 
	$$
	\mu\left(\Cl\left(U_n\setminus \left( \cup_{\sigma\in\Sing(X)\cap\Lambda} B_{r_0}(\sigma)\right)\right)\right) \ge  \upsilon>0, \forall n>0,
	$$
	and consequently, $\mu(\cup_\sigma \CP_\sigma(\delta)\setminus \Sing(X))>0$. However, this is impossible due to Lemma~\ref{l.D}.
\end{proof}

We continue with the choice of parameters.  By Lemma~\ref{l.U}, we may take $U_\Sing$ an open neighborhood of $\cup_\sigma \CP_\sigma(\delta)$ such that for any invariant probability  measure $\mu$, it holds that 
\begin{equation}\label{e.U}
	\mu\left(\Cl\left(U_\Sing\setminus \left( \cup_{\sigma\in\Sing(X)\cap\Lambda} B_{r_0}(\sigma)\right)\right)\right) < \frac{\beta_1}{1000},
\end{equation} 
where we recall that $\beta_1$ is given by Lemma~\ref{l.bihyptime} applied to $\beta_0$; 
then by Theorem~\ref{t.spec}, the corresponding orbits collection $\cG_{spec}(\lambda_0,r_0,W_r, U_\Sing)$ has the tail specification property at scale $\delta.$

Finally we are ready to define $\cG$, the collection of ``good'' orbit segments. Keeping Theorem~\ref{t.Bowen} and Theorem~\ref{t.spec} in mind, we defined $\cG$ to be the collection of orbit segments $(x,t)$ such that:
\begin{enumerate}
	\item $x_0\notin W_r\cup U_\Sing$, and $x_t\notin  W_r$;
	\item $x_t$ is a $(\lambda_0,cu, \beta_0,W_r)$-simultaneous Pliss times.
\end{enumerate}
In other words, $\cG = \cG_{B}\cap \cG_{spec}$. As an immediate corollary of the aforementioned theorems, we obtain:
\begin{proposition}\label{p.cG}
	The orbit segments collection $\cG$ defined above has the Bowen property at scale $\vep$, and the tail specification property at scale $\delta$. 
\end{proposition}

Below we shall describe the orbit segments collection $\cD$, on which there exists a $(\cP,\cG,\cS)$-decomposition. 

Given an orbit segment $(x,t)\in\Lambda\times \RR^+$, we define the following two functions: 
\begin{equation}\label{e.ps}
	\begin{split}&p(x,t) = \min\{s\in[0,t]: x_s\notin W_r\cup U_\Sing\}, \mbox{ and }\\
		&\tilde s(x,t) = \max\{s\in[0,t]: x_{s}\notin W_r\}. 
	\end{split}
\end{equation}
We let 
\begin{equation}\label{e.B1}
	\cB_1=\{(x,t)\in\Lambda\times \RR^+: \mbox{ either $p$ or $\tilde s$ does not exist, or $\tilde s-p\le N_1$}\},
\end{equation}
where $N_1$ is given by Lemma~\ref{l.bihyptime} applied to $W_r.$

\begin{figure}[h!]
	\centering
	\def\svgwidth{\columnwidth}
	\includegraphics[scale=1.1]{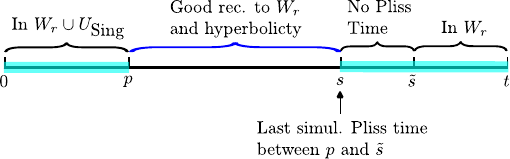}
	\caption{The $(\cP,\cG,\cS)$ decomposition.}
	\label{f.decomp}
\end{figure}

For those orbit segments $(x,t)$ such that $\tilde s-p> N_1$, we take $s=s(x,t)\in (p,\tilde s]$ to be the largest time $\tau$ such that $\tau$ is a $(\lambda_0,cu, \beta_0,W_r)$-simultaneous Pliss times time for the orbit segment $(x_{p}, \tau-p)$. See Figure \ref{f.decomp}.   If there is no such time between $(p,\tilde s]$, then we put $(x,t)$ in the set $\cB_2$; that is,
\begin{equation*}
	\begin{split}
		\cB_2 = \big\{(x,t)\in (\cB^1)^c:&\,\, x_s \mbox{ is not a $(\lambda_0,cu, \beta_0,W_r)$-simultaneous Pliss time for }\\& (x_p,s-p), \forall s\in (p,\tilde s]{\cap \NN}\big\}.
	\end{split}
\end{equation*}
Note that the definition of $p$ and $\tilde s$ implies that $x_{p}$ and $x_{\tilde s}$ are not in $W_r$. As an immediate corollary of Lemma~\ref{l.bihyptime}, we obtain:

\begin{lemma}\label{l.B2}
	If $(x,t)\in\cB_2$, then 
	$$
	\frac{\#\{j=0,\ldots, \floor{\tilde s-p}-1: x_{\tilde s -j}\in W_r\}}{\floor{\tilde s-p}} >\beta_1.
	$$
\end{lemma}
It is straightforward from the construction that the orbit segment $(x_p,s-p)$ is contained in $\cG$. Therefore the functions $p(x,t)$ and $s(x,t)$ gives a $(\cP,\cG,\cS)$-decomposition on the orbit collection $\cD = (\cB_1\cup \cB_2)^c$.

\subsection{Proof of Theorem~\ref{m.B}}\label{ss.6.3}
Finally we are ready to prove Theorem~\ref{m.B}.

First, recall that $X|_ \Lambda$ is almost expansive at scale $\vep$.
Meanwhile, the tail specification at scale $\delta$ and the Bowen property at scale $\vep$ on $\cG$ has been proven by Proposition~\ref{p.cG}. Therefore we are only left to prove that the ``pressure gap'' property, that is, 
\begin{equation}\label{e.gap11}
	P(\cB_1\cup\cB_2 \cup [\cP]\cup[\cS],\phi,\delta,\vep) < P(\phi).
\end{equation}

We start with the following observations, whose proofs are straightforward.

\medskip 
\noindent {\em Observation 1}:
\begin{equation*}
	\begin{split}
		&P(\cB_1\cup\cB_2 \cup [\cP]\cup[\cS],\phi,\delta,\vep)\\ &\le \max\left\{P(\cB_1,\phi,\delta,\vep), P(\cB_2,\phi,\delta,\vep), P([\cP],\phi,\delta,\vep), P([\cS],\phi,\delta,\vep)
		\right\}.
	\end{split}
\end{equation*}

\noindent {\em Observation 2}:
For any $\cC\subset \Lambda\times\RR^+$, by~\eqref{e.vep} we have $P(\cC,\phi,\delta,\vep)< P(\cC,\phi,\delta) + \frac{a_0}{2}$

Therefore, in order to prove~\eqref{e.gap11} we must prove that $P(\cC, \phi,\delta) < P(\phi)-\frac{a_0}{2}$ for $\cC = \cB_1,\cB_2,[\cP]$ and $[\cS]$. The proof relies on the following general results concerning the pressure of an orbit segments collection.

Let $\cC\subset \Lambda\times \RR^+$ be a collection of orbit segments. Recall that the pressure of the potential $\phi$ on $\cC$ is defined as (here the second scale is zero, in view of Observation 2 above)
$$
\Lambda(\cC,\phi, \delta,t) = \sup\left\{\sum_{x\in E_t}\exp(\Phi_0(x,t)): E_t \mbox{ is a $(t,\delta)$-separated set of } \cC_t \right\}, 
$$
and
$$
P(\cC,\phi, \delta)  =\limsup_{t\to\infty} \frac1t \log \Lambda(\cC,\phi, \delta,t),
$$
where $\cC_t = \{x:(x,t)\in\cC\}$.

For each $t>0$ we choose $E_t$ a $(t,\delta)$-separated set of $\cC_t$ with 
$$
\log\sum_{x\in E_t}\exp(\Phi_0(x,t)) \ge \log \Lambda(\cC,\phi, \delta,t)-1.
$$
Then we consider 
\begin{equation}\label{e.mu}
	\begin{split}
		\nu_t :&= \frac{\sum_{x\in E_t}\exp(\Phi_0(x,t))\cdot \delta_x}{\sum_{x\in E_t}\exp(\Phi_0(x,t))}, \mbox{ and }\\
		\mu_t :&= \frac1t \int_0^t(f_s)_*\nu_t\, ds\\
		& = \frac{1}{\sum_{x\in E_t}\exp(\Phi_0(x,t))} \sum_{x\in E_t}\exp(\Phi_0(x,t))\cdot \left(\frac1t\int_0^t\delta_{x_s} \,ds\right).
	\end{split}
\end{equation}
Let $\mu$ be any limit point of the integer-indexed sequence $(\mu_t)_{t\in \NN}$, under the weak-* topology.

\begin{lemma}\cite[Proposition 5.1]{BCTT}\label{l.var.princ}
	It holds that 
	$$
	P(\cC, \phi,\delta)\le P(\mu).
	$$
\end{lemma}
The proof follows the proof of the variational principle~\cite[Theorem 8.6]{Wal} closely. We include it in the appendix for completeness; see Appendix~\ref{ap.var}.

For simplicity, below we will denote by 
$$
\delta_{(x,t)} = \frac1t \int_0^t \delta_{x_s} \,ds
$$
the empirical measure on the orbit segment $(x,t)$. It should not be confused with $\delta_{x_t}$ which is the point mass at the point $x_t = f_t(x)$. The next four lemmas deal with the pressure of $\cB_1, \cB_2, [\cP]$ and $[\cS]$, respectively.

\begin{lemma}\label{l.B1}
	We have $P(\cB_1, \phi,\delta) < P(\phi)- a_0$.
\end{lemma}
\begin{proof}
	Recall that $\cB_1$ consists of those orbit segments $(x,t)$ for which either $p$ or $\tilde s$ defined by~\eqref{e.ps} does not exist, or $p> \tilde s - N_1$ where $N_1$ is the constant given by Lemma~\ref{l.bihyptime}. We consider the following cases:
	
	\medskip
	
	\noindent Case 1: $p(x,t)$ does not exist. In this case, the entire orbit segment is contained in $W_r\cup U_\Sing$. As a result, the empirical measure $\delta_{(x,t)}$ satisfies $\dxt(W_r\cup U_\Sing) =1$.
	
	\noindent Case 2: $\tilde s(x,t)$ does not exist. Similar to the previous case, we must have $\dxt(W_r) =1$.
	
	\noindent Case 3: $\tilde s-p\le N_1$. Since $N_1$ is a constant, for $t$ sufficiently large we have 
	$$
	\dxt(W_r\cup U_\Sing)  = 0.99 > 0.6\cdot \beta_1. 
	$$
	
	In all cases, we see that for any $(t,\delta)$-separated set $E_t$ of $(\cB_1)_t$, the measure $\mu_t$ defined as in~\eqref{e.mu} is a convex combination of empirical measures $\dxt$ for $x\in E_t$, and consequently, satisfies
	$$
	\mu_t(W_r\cup U_\Sing) > 0.6\beta_1
	$$
	if $t$ is sufficiently large.
	This shows that for any limit point $\mu$ of $\mu_t$ (as $t\to\infty$), it holds that
	$$
	\mu\left(\Cl\left(W_r\cup U_\Sing\right)\right)\ge 0.6\beta_1.
	$$ 
	Keeping in mind that $W_r\subset \cup_{\sigma\in\Sing(X)\cap\Lambda} B_{r_0}(\sigma)$ and~\eqref{e.U} which holds for all invariant probability measures, we have 
	$$
	\mu\left(\Cl\left(\cup_{\sigma\in\Sing(X)\cap\Lambda} B_{r_0}(\sigma)\right)\right)\ge \beta_1/2,
	$$ 
	which, according to~\eqref{e.pgap1}, implies that $P(\mu) < P(\phi) - a_0$. Then we obtain from Lemma~\ref{l.var.princ} 
	$$
	P(\cB_1, \phi,\delta) \le P(\mu) < P(\phi) - a_0,
	$$
	as required.
\end{proof}

\begin{lemma}\label{l.B2'}
	$P(\cB_2, \phi,\delta) < P(\phi)- a_0$.
\end{lemma}
\begin{proof}
	By Lemma~\ref{l.B2}, we have 
	$$
	\frac{\#\{j=0,\ldots, \floor{\tilde s-p}-1: x_{\tilde s -j}\in W_r\}}{\floor{\tilde s-p}} >\beta_1.
	$$
	for any $(x,t)\in \cB_2$. The definition of $W_r$ (see~\eqref{e.W.def}) guarantees that if $x_{i}$ and $x_{i+1} = f_1(x_i)$ are both in $W_r^c$, then the orbit segment $(x_i,1)$ does not intersect with $W_r$.	Also note that the orbit segment before $p$ or after $\tilde s$ are contained in $W_r\cup U_\Sing$. 
	Then, for any $(t,\delta)$-separated set $E_t$ of $(\cB_2)_t$ and every $x\in E_t$, it holds 
	$$
	\dxt(W_r) > \beta_1.
	$$
	Similar to the previous lemma, we then have 
	$$
	\mu\left(\Cl\left( \cup_{\sigma\in\Sing(X)\cap\Lambda} B_{r_0}(\sigma)\right)\right)  \ge \frac{\beta_1}{2},
	$$
	where $\mu$ is any limit point of the sequence of measures $(\mu_t)_t$ defined by~\eqref{e.mu}. Now Lemma~\ref{l.var.princ} and Equation~\eqref{e.pgap1} implies that 
	$$
	P(\cB_2, \phi,\delta) < P(\phi)- a_0,
	$$
	as required.
\end{proof}

\begin{lemma}\label{l.P}
	$P([\cP], \phi,\delta) < P(\phi)- a_0$.
\end{lemma}
\begin{proof}
	Recall the definition of  $[\cP]$ from~\eqref{e.[C]}.  Also recall that $\cP$ consists of orbit segments of the form $(x,p(x,t))$ where $p(x,t)$ is defined by~\eqref{e.ps}. By the definition of $p$, we see that if $(x,p)\in [\cP]$, then  
	$$
	\delta_{(x,p)} (W_r\cup U_\Sing) > 0.9.
	$$
	Then the rest of the proof follows from the same argument in the proof of the previous two lemmas.
\end{proof}

\begin{lemma}\label{l.S}
	$P([\cS], \phi,\delta) < P(\phi)- a_0$.
\end{lemma}

\begin{proof} As in the previous lemma, we shall consider  $\cS$  instead, as the effect of the prefix and suffix of length at most one in the empirical measure $\mu_t$ is inconsequential after sending $t$ to infinity.
	
	Let $(x,t)\in \cS$, and assume that it is obtained from some orbit $(y,\tau)\in \cD$. By this, we mean
	$$
	x = y_{s(y,\tau)}, \mbox{ and }t = \tau - s(y,\tau),
	$$
	where $s(y,\tau)$ is the largest $(\lambda_0,cu, \beta_0,W_r)$-simultaneous Pliss times  between $(p,\tilde s]$ where $\tilde s(y,\tau)$ is defined by~\eqref{e.ps}.

	From the definition of $\tilde s(y,\tau)$ we see that the orbit segment $(y_{\tilde s}, \tau - \tilde s) = (x_{\tilde s - s}, \tau -\tilde s)$ is contained in $\Cl(W_r)$. Then we have 
	\begin{equation}\label{e.tildes}
		\delta_{(x_{\tilde s - s},\tau-\tilde s)} (W_r) = 1.
	\end{equation}

	Furthermore, we have  $x_{\tilde s - s}=y_{\tilde s}\notin W_r$ and $x\notin W_r$ due to the definition of $s$ and $\tilde s$.  Now we consider the orbit segment $(x,\tilde s-s)$; there are two cases: 
	
	\medskip
	\noindent Case 1. ${\tilde s - s}\le N_1$; in this case, for all $t$ large enough (say, larger than $20N_1$) one has $	\delta_{(x,t)} (W_r) > 0.9$.
	
	\noindent Case 2. ${\tilde s - s}>N_1$; in this case we enlist the help of Lemma~\ref{l.gluePliss}, which shows that one cannot find $\tilde t\in [0,\tilde s-s]$  such that $x_{\tilde t}$ is a $(\lambda_0,cu, \beta_0,W_r)$-simultaneous Pliss times  for the orbit segment $(x,{\tilde t})$; otherwise, $y_{s+\tilde t} = x_{\tilde t}$ would become a $(\lambda_0,cu, \beta_0,W_r)$-simultaneous Pliss times for the orbit segment  $(y_p,s-p+\tilde t)$, and contradicts with the maximality of $s(y,\tau)$. Then by Lemma~\ref{l.bihyptime} (and the observation that $x = y_{s} ,x_{\tilde s -s } = y_{\tilde s}$ are both in $W_r^c$) we have 
	$$
	\frac{\#\{i=0,\ldots, \floor{\tilde s - s}-1: x_i\in W_r\}}{\floor{\tilde s - s }} > \beta_1.
	$$
	This, combined with~\eqref{e.tildes}, shows that 
	$$
	\dxt > \beta_1.
	$$
	Then we have $\mu(\Cl(W_r)) \ge \beta_1$ where $\mu$ is any limit point of $\mu_t$ defined using~\eqref{e.mu}. By Lemma~\ref{l.var.princ} and Equation~\eqref{e.pgap1} we have
	$$
	P(\cS, \phi,\delta) < P(\phi)- a_0,
	$$
	as required.
\end{proof}

Now by Observation 1 and 2, we have 
\begin{align*}
	&P(\cD^c\cup[\cP]\cup[\cS],\phi,\delta,\vep)\\ 
	&\le \max\left\{P(\cB_1,\phi,\delta,\vep),  P(\cB_2,\phi,\delta,\vep),  P([\cP],\phi,\delta,\vep),  P([\cS],\phi,\delta,\vep)\right\}\\
	&\le \max\left\{P(\cB_1,\phi,\delta),  P(\cB_2,\phi,\delta),  P([\cP],\phi,\delta),  P([\cS],\phi,\delta)\right\} + \frac{a_0}{2}\\
	&\le P(\phi)-a_0+ \frac{a_0}{2}\\
	&= P(\phi)-\frac{a_0}{2}.
\end{align*}
This verifies Assumption (III) of the improved CT criterion (Theorem~\ref{t.improvedCL}) and concludes the proof of Theorem~\ref{m.B}, {leaving only the proof of Theorem \ref{t.Bowen} and \ref{t.spec}.}

\subsection{Proof of Theorem~\ref{m.Lorenz} and~\ref{m.A}}\label{ss.proofAB}
Now we are ready to prove  Theorem~\ref{m.Lorenz} and Theorem~\ref{m.A} by verifying the assumptions in Theorem~\ref{m.B}. 

\begin{proof}[Proof of Theorem~\ref{m.Lorenz}]	
	The existence and the singular-hyperbolicity of the classical Lorenz attractor were proven in~\cite{Tucker} (see also~\cite{AM} on the smoothness of the stable foliation, a fact that Tucker used in his paper without giving the proof). 
	
	Note that the classical Lorenz equation is a compact attractor. Let $U$ be an attracting neighborhood of $\Lambda$ whose closure is compact, then it is straightforward to verify that all the discussion in Section~\ref{s.p1} and \ref{s.p2} applies to $X|_U.$ In particular, Theorem~\ref{m.B} applies despite it being stated for compact manifolds without boundary.\footnote{Alternatively one can smoothly identify $X|_U$ with a $C^1$ flow $\tilde X$ on a small open region of $S^3$. $\tilde X$ can then be extended to the entire $S^3$ such that the only invariant set outside $U$ is a hyperbolic source. Then one can apply Theorem~\ref{m.B} to $\tilde X$.} Below we shall verify the three assumptions of Theorem~\ref{m.B}.

	Recall that the classical Lorenz attractor has a hyperbolic singularity $\sigma$ at the origin $(0,0,0)$; it is the only singularity contained in the attractor $\Lambda$.  Therefore Assumption (A) of Theorem~\ref{m.B} is satisfied.
	
	Next we verify Assumption (B). 
	The existence of a hyperbolic periodic orbit $\gamma$ is proven in~\cite[Corollary D]{PYY}. Below we will show that $\gamma$ has the desired properties. For this purpose, we briefly recall Tucker's work concerning the topological structure of the classical Lorenz attractor.
	
	We shall write $(x_1,x_2,x_3)$ for the coordinates in $\RR^3$. The following facts are summarized in~\cite[Section 2.1 and Proposition 5.1]{Tucker}: 
	\begin{enumerate}
		\item there exists a return plane $\Sigma$ at $\{x_3=27=\rho-1\}$ (taken to be $[-6,6]\times [-3,3]\times \{27\}$) such that $\Sigma\setminus W^s(\sigma)$ consists of two disjoint rectangular pieces $\Sigma^+$ and $\Sigma^-$, and a return map $R: (\Sigma^+\cup \Sigma^-)\to\Sigma$; the set  $\Lambda_\Sigma := \bigcap_{n=0}^\infty R^n(\Sigma^+\cup \Sigma^-)$ is a transitive attractor for the return map;
		\item there exists a one-dimensional $C^{1+\alpha}$ stable foliation on $\Sigma$, invariant under the return map $R$ (see also~\cite{AM} for the smoothness); the quotient of $R$ over the stable foliation gives a $C^{1+\alpha}$ one-dimensional map $f_\Sigma$;
		\item on $\Sigma$ there exists a cone field $\mathfrak C$ (the unstable cone field) which is mapped strictly into itself by $DR$;
		\item let $\ell$ be a leaf of the stable foliation on $\Sigma$; we call the region between $\ell$ and $R(\ell)$ a {\em fundamental domain} of $R$; note that the quotient of a fundamental domain over the stable foliation is an interval of the form $(x,f_\Sigma (x))$;
		\item (\cite[Proposition 5.1]{Tucker}) there exists a set $F\subset \Sigma$ containing a fundamental domain 
		such that 
		\begin{itemize}
			\item any orbit with $x_0\in F$ that eventually leaves $F$ satisfies that for every return $= R^n(x)\in F$,
			$$
			\min \{|DR^n(x_0)v|: v\in\mathfrak C\} >2|v|;			
			$$
			\item  any orbit completely contained in $F$ satisfies 
			$$
			\min \{|DR^n(x_0)v|: v\in\mathfrak C\}>2^{n/2}|v|;
			$$
		\end{itemize}
		in particular, any small segment tangent to the cone field $\mathfrak C$ must more than double its length between two consecutive slicings over $\Gamma$;
		\item the one-dimensional map $f_\Sigma$ is uniformly expanding; in particular, if $x\notin W^s(\sigma)$ and $F$ is a fundamental domain of $R$, then the orbit of $x$ will eventually enter $F$ (this fact is hidden in the proof of~\cite[Main Theorem]{Tucker}; see the end of~\cite[Section 2.4]{Tucker}). 
	\end{enumerate}
	
	With this, we are ready to prove Assumption (B). 
	
	\smallskip  
	\noindent (B1) {\em   $W^u(\gamma)$ is a quasi $u$-section.} We take $p\in\gamma\cap \Sigma$ a hyperbolic periodic point of $R$, and denote by $S_0:=W^u_\Sigma(p)$ the connected component of the intersection between $\Sigma$ and the local unstable manifold of $\gamma$  that contains $p.$ Then $S_0$ is a one-dimensional segment tangent to $\mathfrak C$. By (5) above, the first return of $S_0$ to $F$, denoted by $\tilde S_1$, must have more than double the length of $S_0$. It may be sliced by $\Gamma$ into two pieces; if this happens, we take the longer piece and denote it by $S_1$; otherwise let $S_1=\tilde S_1$. Then the length of $S_1$ is greater than $S_0$ with a ratio bounded away from one and is still tangent to $\mathfrak C$. Repeating this process, we obtain a sequence of longer and longer segments $\{S_n\}$ tangent to $\mathfrak C$. They will eventually totally cross a fundamental domain $F$. Now take any regular point  $x\in \Lambda$; there are two possibilities:
	
	\noindent Case 1. $x\in W^s(\sigma)$. Then $\Gamma \subset W^s(\Orb(x))$  has a non-trivial transverse intersection with $S_n$ for $n$ large enough.
	
	\noindent Case 2. $x\ne W^s(\sigma)$. In this case, we let $\tilde x$ be the first entry of the orbit of $x$ into $F$. Let $W^s_\Sigma(\tilde x)\subset\Sigma$ be the (local) intersection between the stable manifold of $\Orb(\tilde x)$ and $\Sigma$. Then we have 
	$$
	W^s_\Sigma(\tilde x)\pitchfork S_n\ne\emptyset \mbox{ for some } n>0.
	$$ 
	In both cases, by invariance of invariant manifolds under the flow, we have $W^s(x)\pitchfork W^u(\gamma)\ne\emptyset$, that is, $W^u(\gamma)$ is a quasi $u$-section.
	
	\smallskip
	\noindent  (B2) {\em $W^s(\gamma)$ is dense in $\Lambda$.} The proof is similar. For any open set $U\cap\Lambda\ne\emptyset$, we take $z\in (U\cap\Lambda)\setminus W^s(\sigma)$ and $S\subset U$ a two-dimensional disk centered at $z$ and tangent to a small $(\alpha,F^{cu})$-cone. Let $z^0 = f_\tau(z)$ be the first entry of the orbit of $z$ into a fixed fundamental domain $F\subset \Sigma$, and $S_0$ be a segment contained in $f_t(S)\cap \Sigma$ that contains $z^0$. Then $S_0$ is tangent to $\mathfrak C$ if $\alpha$ is small enough. Using the previous argument, we obtain a sequence of longer and longer segments $\{S_n\}$ that eventually fully cross $F$. On the other hand, let $p\in\gamma\cap F$;  then, for $n$ large enough, $S_n$ must have a non-trivial transverse intersection with $W^s_\Sigma(p)$. By invariance, we have 
	$$
	S\pitchfork W^s(\gamma )\ne\emptyset.
	$$
	This shows that $W^s(\gamma)\cap U\ne\emptyset$. Since $U$ is taken arbitrarily, we see that $W^s(\gamma)$ is dense in $\Lambda$. This finishes the proof of Assumption (B).
	
	Finally we consider Assumption (C). Recall that the topological entropy of the Lorenz attractor is positive (since it supports an SRB measure (\cite[Main Theorem]{Tucker}) which satisfies Pesin's entropy formula; see also \cite[Theorem C]{PYY}). By Remark~\ref{r.mainthm}, for the constant potential $\phi\equiv 0$ we have $\phi(0) = 0< h_{top}(X)$. This verifies Assumption (C) for $\phi\equiv 0$; as a result, there exists a unique measure of maximal entropy. 
	
	The proof of Theorem~\ref{m.Lorenz} is now complete. 
\end{proof}

\begin{proof}[Proof of Theorem~\ref{m.A}] 
	Again, we verify the assumptions of Theorem~\ref{m.B}. 
	
	Clearly Assumption (A) is an open and dense condition among $C^1$ vector fields. Assumption (B) is also satisfied by an open and dense subset of $\mathfrak{X}^1(\bM)$ thanks to Theorem~\ref{t.top}.  Therefore, for every H\"older potential $\phi$ that satisfies Assumption (C), there exists a unique equilibrium state on $\Lambda$. 
\end{proof}

\section{Estimates near singularities}\label{s.sing}

	In the theory of uniform hyperbolicity, the Bowen property (also known as the bounded distortion property) requires the existence of a local product structure formed by local invariant manifolds at a uniform scale. For singular flows, however, there are two essential problems:
	\begin{itemize}
		\item there does not exist an unstable foliation, either on the manifold $\bM$ or on the normal planes $\cN(x), x\in\Reg(X)$;\footnote{The Pesin unstable manifolds do not have uniform size, and therefore do not provide a local product structure with uniform size. Furthermore, they are defined on the manifold, and we need a local product structure on the normal planes.}
		\item for the map $\cP_{t,x}$, its derivative is only uniformly continuous at the scale $\rho_1|X(x)|$; see Proposition~\ref{p.tubular3}.  Here we need the uniform continuity to obtain the exponential contraction
			for the distance between $x$ and $[x, y]$ using the hyperbolicity at $x$, where $[x,y]$ is the point given by the local product structure on the normal plane of $x$.
	\end{itemize}
	
	To solve these issues, in this section we will construct two families of fake foliations:
	\begin{itemize}
		\item $\cF^{s/cu}_{x,\cN}$ are fake foliations defined on $\cN_{\rho_0|X(x)|}(x)$ at every regular point $x\in\Lambda$; they form a local product structure on the normal planes at the uniformly relative scale $\rho_1|X(x)|$; 
		\item $\cF_\sigma^{*},*=s,c,u,cs,cu$ are fake foliations defined near each singularity in $\Lambda$; they have uniform scales and are used when an orbit passes through a neighborhood of a singularity. 
	\end{itemize} 
	Then, in Section~\ref{ss.nearsing} we will state and prove two key lemmas that provide the expansion/contraction estimates on respective leaves of $\cF_{x,\cN}^{s/cu}$ when an orbit segment passes through the set $W_r(\sigma)$ defined by~\eqref{e.W.def}. The key in the proof is a change of coordinate between $(\cF_{x,\cN}^{s},\cF_{x,\cN}^{cu})$ and $(\cF_\sigma^{s},\cF_\sigma^{c},\cF_\sigma^{u})$. When such a change of coordinate is performed, one will lose information on the smaller coordinate (Section~\ref{ss.cone}). For this reason, we need to consider two cases that are determined by whether $s$ or $cu$ has the larger coordinate separately.

\subsection{Cone Estimates}\label{ss.cone}
The foliations $(\cF^s_{x,\cN},\cF^{cu}_{x,\cN})$ and $(\cF^s_\sigma, \cF^c_\sigma,\cF^u_\sigma)$ that will be constructed later are tangent to some small cones around their respective bundles,
but in general are not compatible,   not to mention that they are invariant under different dynamics ($\cP_{1,x}$ and $f_1$, respectively). For example, for $y\in B_{r_0}(\sigma)\cap \cN_{\rho_0|X(x)|}(x)\cap\Lambda$ where $x$ is a regular point in $\Lambda$, it is possible that $\cF^s_{x,\cN}(y)$ and $\cF^s_\sigma(y)$ intersect transversely at $y$,  despite them both  being tangent to a small cone of $E^{ss}= E^s_N$.  
We start this section with a general setup regarding the change of coordinates between (fake) foliation coordinate systems.  For simplicity, in this subsection we shall only consider foliations defined in $\RR^n$, but the same results hold on $\bM$ or on $\cN(x),x\in\Reg(X)$ if one only considers small scales.

\subsubsection{An elementary lemma for vectors contained in a cone}

Fix $0<k<n$. We write $\RR^n=E\oplus F$ with $E$ being the subspace generated by the first $k$ coordinates and $F$ the subspace generated by the last $n-k$ coordinates.

For a vector $ v\in \RR^n$ we will say that $ v$ is contained in the $(\alpha,F)$-cone if $ v =  v^F +  v^{E}$ with $ v^F\in F$ and $ v^{E}\in E$ such that $|v^{E}|\le\alpha | v^F|$. Vectors contained in the $(\alpha,E)$-cone are defined analogously.

\begin{lemma}\label{l.comparecoord}
	There exists a constant $L_0>1$, such that for any $\alpha>0$ sufficiently small, the following holds:
	
	Let $y$ and $y^F$ be points in $\RR^n$,  such that  $y^F$ (considered as a vector) is contained in the $(\alpha,F)$-cone, and $y-y^F$ (considered as a vector) is contained in the $(\alpha,E)$-cone. 
	Write $y=y^1+y^2$ with $y^1\in E$ and $y^2\in F$. 
	Then:
	\begin{align*}
		|y^2|-L_0\alpha |y^1|\le \,\hspace{4.1mm}&|y^F|\hspace{3.7mm}\le|y^2|+L_0\alpha (|y^1|+|y^2|),\mbox{ and }\\
		|y^1|-L_0\alpha |y^2|\le \,|y&-y^F|\le |y^1|+L_0\alpha (|y^1|+|y^2|).
	\end{align*}
	Furthermore, by increasing  $L_0$ if necessary, we have
	\begin{enumerate}
		\item if $|y^F|\ge C_0|y-y^F|$ for some constant $C_0>0$, then there exists $C_0'>0$ such that for $\alpha>0$ sufficiently small,
		$$
		y^F\in \left[(1-C_0^{-1}L_0\alpha )|y^2|, (1+(C_0^{-1}+1)L_0\alpha )|y^2|	\right];
		$$
		moreover,  $|y^2|\ge C_0(1-C_0'\alpha)|y^1|$.
		\item if $|y-y^F|\ge C_0|y^F|$ for some constant $C_0>0$, then there exists $C_0'>0$ such that for $\alpha>0$ sufficiently small,
		$$|y-y^F|\in \left[(1-C_0^{-1}L_0\alpha )|y^1|, (1+(C_0^{-1}+1)L_0\alpha )|y^1|	\right];
		$$
		moreover,  $|y^1|\ge C_0(1-C_0'\alpha)|y^2|$.
	\end{enumerate}	
\end{lemma}

\begin{remark}
	Note that the lower bounds $	|y^2|-L_0\alpha |y^1|$ or $|y^1|-L_0\alpha |y^2|$ could be negative (e.g., fix $\alpha$ and take either $|y^2|/|y^1|$ or $|y^1|/|y^2|$  sufficiently small). However, with assumptions in either Case (1) or (2), we can control one of the lower bounds and obtain that $|y-y^F|$ or $|y^F|$, whichever is larger, is ``comparable'' to the corresponding Euclidean coordinate $|y^i|, i=1,2$. 
\end{remark}

The proof of this lemma can be found in Appendix~\ref{s.A1}.

\subsubsection{{Change of fake foliation coordinates}}\label{ss.foliationcoord}
Given $\alpha>0$ and an embedded disk $D\subset \RR^n$ with dimension $k$, we say that $D$ is {\em tangent to the $(\alpha,E)$-cone}, if $D$ is the graph of a $C^1$ function $g: \{v\in\RR^k: |v|<R\}\to\RR^{n-k}$ for some $R>0$, with $\|Dg\|_{C^0}\le \alpha$. Note that  disks tangent to the $(\alpha,F)$-cone can be defined similarly, and the definition here is consistent with Section~\ref{ss.DS}.  It is clear that if $D$ is tangent to the $(\alpha,E)$-cone, then for any points $x,y\in D$, the vector $y-x$ is contained in the $(\alpha,E)$-cone.  The same statement holds for disks tangent to the $(\alpha, F)$-cone. 

Given two foliations $\cF^E$ and $\cF^{F}$ of $\RR^n$ such that each of their leaves is tangent to the corresponding cones and form a local product structure, for each point $y\in \RR^n$ we write
$$
y=[y^{\cF^E},y^{\cF^F}]
$$
if $y^{\cF^{E}}\in \cF^{E}(0),y^{\cF^{F}}\in \cF^{F}(0)$ such that 
$$
\{y\}=\cF^{F}(y^{\cF^{E}})\pitchfork \cF^{E}(y^{\cF^{F}}).
$$ 
We remark that the vector $y^{\cF^E}$ is contained in the $(\alpha, E)$-cone and the vector $y-y^{\cF^E}$ is contained in the $(\alpha, F)$-cone.

\begin{figure}[h!]
	\centering
	\def\svgwidth{\columnwidth}
	\includegraphics[scale=0.5]{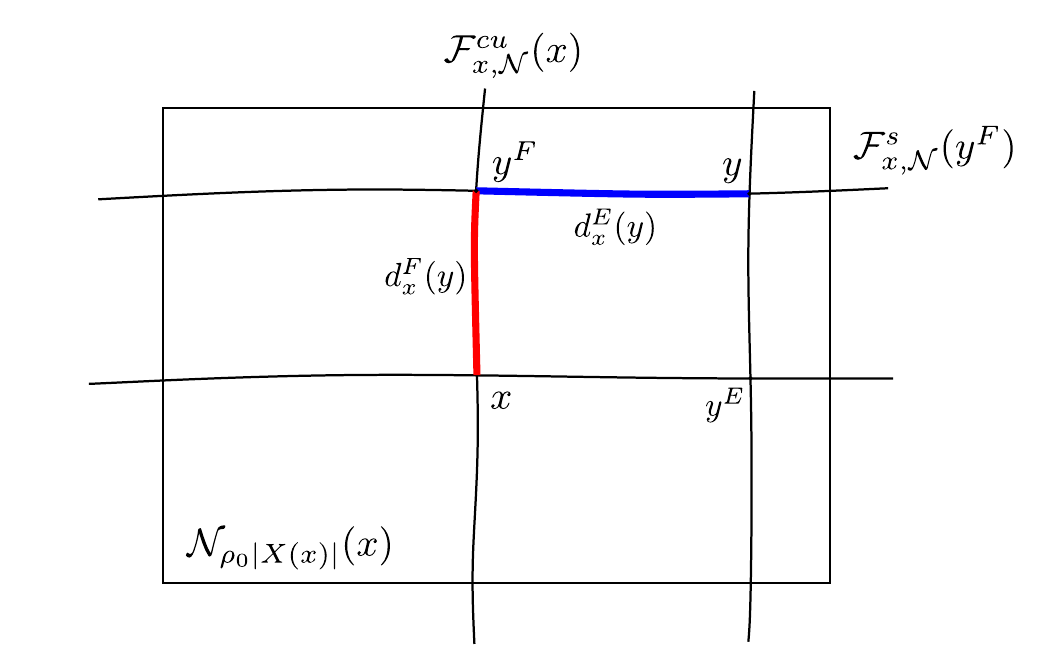}
	\caption{The local product structure and the $E$-length and $F$-length of $y$.}
	\label{f.EF}
\end{figure}

We will call the pair $(\cF^E,\cF^{F})$ a {\em foliation coordinate system of $\RR^n$}.
Clearly, for a given $y$, the points $y^{\cF^{*}}, *=E,F$ depend on the choice of the foliations. As explained before, it is crucial for us to make the transition from one foliation coordinate system $(\cF^E,\cF^{F})$ to another foliation coordinate system $((\cF^{E})',(\cF^{F})')$. This is made possible through the following lemmas.

Let $D$ be an embedded disk in $\RR^n$. For $x,y\in D$ we denote by $d_D(x,y)$ the shortest distance within $D$. The next lemma allows us to interchange $d_D(x,y)$ with $|y-x|$ as long as the disk in question is contained in a proper cone.

\begin{lemma}\label{l.disktovec}
	There exists $L_1$ such that for every $\alpha>0$ small enough, every $k$-dimensional embedded disk  $D$ tangent to the $(\alpha, E)$-cone and every pair of points $x,y\in D$, it holds 
	$$
	|y-x|\le d_{D}(x,y)\le (1+L_1\alpha) |y-x|.
	$$
	A similar statement holds for disks tangent to the $(\alpha,E)$-cone.
\end{lemma}
The proof is straightforward and thus omitted. 

As an immediate corollary of Lemma~\ref{l.comparecoord} and~\ref{l.disktovec}, we have the following lemma.
\begin{lemma}\label{l.changecoord}
	There exists $L_2>0$ such that for all $\alpha>0$ small enough, the following holds:
	
	Let $y,y^F_1, y^F_2$ be points in $\RR^n$ such that
	\begin{itemize}
		\item there exist $k$-dimensional disks $D^E_1,D^E_2$ tangent to the $(\alpha, E)$-cone and contain $y$;
		\item there exist $(n-k)$-dimensional disks $D^F_1,D^F_2$ tangent to the $(\alpha, F)$-cone and contain $0$;
		\item $D^E_i$ and $D^F_i$ intersect transversely at $y^F_i$, for $i=1,2$.
	\end{itemize} 
	Then: 
	\begin{enumerate}
		\item if $d_{D^{F}_1}(0,y^{F}_1)\ge 0.9d_{D^{E}_1}(y^F_1,y)$, then 
		$$
		d_{D^{F}_2}(0,y^F_2)\in (1-L_2\alpha, 1+L_2 \alpha)d_{D^{F}_1}(0,y^F_1), 
		\mbox{ and }	d_{D^{F}_2}(0,y^F_2)\ge\frac12 d_{D^{E}_2}(y^F_2,y);
		$$
		\item  if $d_{D^{E}_1}(y^F_1,y)\ge 0.9d_{D^{F}_1}(0,y^{F}_1)$, then 
		$$
		d_{D^{E}_2}(y^F_2,y)\in (1-L_2\alpha, 1+L_2 \alpha)d_{D^{E}_1}(0,y^F_1), \mbox{ and }	d_{D^{E}_2}(y^F_2,y)\ge\frac12 d_{D^{F}_2}(0,y^F_2).
		$$
	\end{enumerate}
\end{lemma}

\begin{proof}
	We will only prove Case (1), as the proof of Case (2) is analogous.
	
	By Lemma~\ref{l.disktovec} we see that $|y^F_1|\ge 0.8|y-y^F_1|$ for $\alpha$ small enough. Applying Lemma~\ref{l.comparecoord}, case (1) with $C_0=0.8$, we obtain:
	\begin{equation}\label{e.2large}
		|y^F_1|\in \left[(1-2L_0\alpha )|y^2|, (1+3L_0\alpha )|y^2|	\right],\mbox{ and } |y^2|\ge 0.8(1-C_0'\alpha)|y^1|.
	\end{equation}
	\noindent {\em Claim:} $d_{D^{F}_2}(0,y^F_2)\ge\frac12 d_{D^{E}_2}(y^F_2,y)$. 
	
	\noindent {\em Proof of the claim}: assume that this is not the case, i.e., $d_{D^{E}_2}(y^F_2,y)\ge 2d_{D^{F}_2}(0,y^F_2)$. Then we can apply Lemma~\ref{l.disktovec} and Lemma~\ref{l.comparecoord}, case (2) with $C_0=1.8$ to obtain
	\begin{equation*}
		|y^1|\ge 1.8(1-C_0''\alpha)|y^2|.
	\end{equation*}
	However, this contradicts with~\eqref{e.2large} for $\alpha>0$ small enough. 
	
	Now we continue with the proof of the lemma. In view of the previous claim, we can apply  Lemma~\ref{l.comparecoord}, case (1) with $C_0=\frac13$ and Lemma~\ref{l.disktovec} to obtain
	\begin{equation}\label{e.2large'}
		d_{D^{F}_2}(0,y^F_2)\in \left[(1-3L_0\alpha )|y^2|, (1+4L_0\alpha )|y^2|	\right].
	\end{equation}
	Combining~\eqref{e.2large'} with~\eqref{e.2large} then apply Lemma~\ref{l.disktovec}, we get the desired result. 
\end{proof}

\begin{remark}\label{r.1.5}
	Clearly the constant $0.9$ is not essential for the proof. One could easily replace it with any constant $C_2$, and replace $\frac12$ with any constant $C_3<C_2$, and the lemma still holds (for all $\alpha<\overline\alpha$, where $\overline\alpha$ is determined by $C_2$ and $C_3$). 
\end{remark}

\subsection{Construction of fake foliation coordinate systems}
In this section we construct the previously mentioned families of fake foliations on the normal bundle and near each singularity, respectively.
\subsubsection{The Hadamard-Perron theorem}We start with the Hadamard-Perron theorem which is used in~\cite{BW} to construct fake foliations. See~\cite[Theorem 6.2.8]{Katok} and~\cite[Section 3]{BW}. 
\begin{theorem}[Hadamard-Perron Theorem]\label{t.fakeleaves} Fix $\zeta<\eta$
	and let  $\{A_x\}_{x\in I}$ be a $GL(n,\RR)$ cocycle over an invertible dynamical system $T:I\to I$ (here we do not assume $I$ to be compact) such that 
	\begin{itemize}
		\item  $\forall x\in I, u\in\RR^{n-k}$ and $v\in \RR^{k}$, $ A_x(u,v) = (A_x^1(u), A_x^2(v))$;
		\item $\|(A^1_x)^{-1}\|\le \eta^{-1}$, and $\|A_x^2\|\le \zeta$.
	\end{itemize}
	Then for $0<\alpha<\min\{1, \sqrt{\eta/\zeta}-1\}$ and $\iota>0$ sufficiently small, 
	for every $\{h_x\}_{x\in I}$ a family of $C^1$ diffeomorphisms of $\RR^n$ such that 
	$$
	h_x(u,v) = (A_x^1(u)+h_x^1(u,v), A_x^2(v) + h_x^2(u,v))
	$$
	with $\|h_x^1\|_{C^1}<\iota$ and $\|h_x^2\|_{C^1}<\iota$ for all $x\in I$, there exist two foliations $\cF^E_{x}$ and $\cF^{F}_{x}$, with the following properties:
	\begin{enumerate}
		\item (tangent to respective cones) each leaf of $\cF^E_{x}$ is the graph of a $C^1$ function $h_x^E:\RR^{k}\to \RR^{n-k}$ with $\|Dh_x^E\|_{C^0}<\alpha$ for all $x\in I$; similarly, each leaf of $\cF^{F}_{x}$ is the graph of a $C^1$ function $h_x^{F}:\RR^{n-k}\to \RR^{k}$ with $\|Dh_x^{F}\|_{C^0}<\alpha$;
		\item (invariance) for every $x\in I$ and $y\in\RR^n$, it holds for $*=E,F$:
		$$
		h_x\left(\cF^{*}_{x}(y)\right) = \cF^{*}_{Tx}(h_x(y));
		$$
		\item (product structure) for every $y\in \RR^n$ and every $x\in I$, there exist two unique points $y^E\in \cF^{E}_{x}(0),y^{F}\in \cF^{F}_{x}(0)$ such that 
		$$
		\{y\}=\cF^{F}_{x}(y^E)\pitchfork \cF^{E}_{x}(y^{F}).
		$$ 
	\end{enumerate}
\end{theorem}


The proof is omitted. We invite interested readers to~\cite[Section 3]{BW}. It is worth noting that the index set $I$ need not be compact (in fact, it needs not to be a topological space). The uniform estimate of the foliations is due to, as noted by the authors of~\cite{BW}, the uniformity of the parameters concerning $f^{*}_x$ and $A^{*}_{x}, *=1,2.$

Below we will apply this theorem in two different ways: first we will apply it to the family of (lifted) Poincar\'e maps $\{P_{1,x}: x\in \bM\setminus \Sing(X)\}$ extended to the normal bundle via some smooth bump function, where each $\RR^n$ should be considered as $N(x)$ (hence $n=\dim \bM-1$); this gives the foliations $\cF^{s/cu}_{x,\cN}$ defined on the normal plane of regular points at a uniformly relative scale;
next, we will also apply it to each singularity, resulting in a family of foliations $\cF_\sigma^*, *=s,c,u,cs,cu$ for each $\sigma\in\Sing(X)$ at a uniform scale.

\subsubsection{Fake foliations on the scaled normal bundle}\label{ss.fake.normal}
Fix $\alpha>0$ small enough so that all the results in Section~\ref{ss.cone} hold, with $1000\alpha\cdot \max\{L_0, L_1,L_2\}<0.01$.

We will apply the previous theorem for this $\alpha$ (decrease it if necessary) to the family of scaled sectional Poincar\'e maps  $\{P^*_{1,x}\}$ defined by~\eqref{e.p*} over the time-one map $T=f_1$ on the index set $I=\Reg(X)$. In particular, we have  $n=\dim \bM-1$, $\RR^k=E^s_N$ and $\RR^{n-k}=F^{cu}_N$. Here each normal plane $N(x)$ is identified with $\RR^n$, and the $C^1$ diffeomorphisms $f_x$ are defined as
$$
\tilde P^*_{1,x} := \begin{cases}
	P_{1,x}^*, &|y|<\upsilon\\
	g_x, &|y|\in[\upsilon,K\upsilon]\\
	\psi^*_{1,x}, &|y|>K\upsilon
\end{cases}
$$
where $\upsilon$ is taken small enough, independent of $x$, such that the maps $P^*_{1,x}$ are $\iota$-close to $\psi^*_{1,x}$  under $C^1$ topology where $\iota=a>0$ is given by the previous theorem (see Proposition~\ref{p.tubular3} (2)); also recall that 
$$
D_0\tilde P^*_{1,x} =D_0P_{1,x}^*=\psi^*_{1,x};
$$
$K>0$ is a constant large enough such that the smooth bump functions $g_x$ satisfy $\|g_x-\psi^*_{1,x}\|_{C^1}<\iota$.  The construction of $\tilde P^*_{1,x}$ guarantees that 
$$
\|\tilde P^*_{1,x}-\psi^*_{1,x}\|_{C^1}<\iota,
$$
provided that $\upsilon$ is taken small enough and $K$ sufficiently large.

Then Theorem~\ref*{t.fakeleaves} gives us two foliations $\cF^{s,*}_{x,N}$ and $\cF^{cu,*}_{x,N}$. Here the sub-index $N$ highlights the fact that these foliations are defined on the (scaled) normal plane $N(x)\subset T_x\bM$ for every $x\in \Reg(X)$. The $*$ in the superscript is due to the foliations being defined using the scaled maps $P^*_{1,x}$. Note that these foliations are tangent to the $(\alpha, E_N^s)$-cone and $(\alpha, F_N^{cu})$-cone, respectively. They are invariant in the sense that 
$$
\tilde P^*_{1,x}(\cF^{s,*}_{x,N}) = \cF^{s,*}_{x_1,N},\,\,\, \tilde P^*_{1,x}(\cF^{cu,*}_{x,N}) = \cF^{cu,*}_{x_1,N}.
$$ 
We remark that for each $x\in \Reg(X)$,  $\cF^s_{x,N}$ and $\cF^{cu}_{x,N}$ form a foliation coordinate system of $N(x)$, as described in Section~\ref{ss.foliationcoord}.

Finally, we rescale these foliations by the flow speed $|X(x)|$ and push them to the manifold $\bM$ via the exponential map $\exp_x$. To be more precise, we define the maps
$$
S_x: N(x)\to N(x), S_x(u) = |X(x)|u
$$
and the foliations
$$
\cF^i_{x,\cN} = \exp_{x}\left(S_x\left(\cF^{i,*}_{x,N}\right)\right) , i=s,cu,
$$
Note that $\cF^i_{x,\cN}$ are well-defined on $\cN_{{\rho_0}|X(x)|}(x)$ \footnote{Indeed, these foliations are well-defined at a uniform scale $\cN(x) = \exp_x (N_{\mathfrak{d}_0}(x))$ where $\mathfrak{d}_0$ is the injectivity radius; however, since $\cP_{1,x}$ is only defined at a uniformly relative scale $\rho_0|X(x)|$, the invariance of those foliations only holds at a uniformly relative scale.} and are invariant under the maps $\cP_{1,x}$ in the following sense:
\begin{equation}\label{e.inv}
	\cP_{1,x}\left(\cF^{cu}_{x,\cN} (y)\right)\supset \cF^{cu}_{x_1,\cN} (\cP_{1,x}(y)),\,\, \mbox{ for } y\in \cN_{\min\{\rho_0K_0^{-1},|v|\}|X(x)|}(x), x\in\Reg(X);
\end{equation}
and a similar statement holds for $\cF^{s}_{x,\cN}$ under the maps $\cP_{-1,x}$. Below we will always assume that $\rho_0K_0^{-1}< |v|$ by decreasing $\rho_0$ when necessary.

Consequently, we have a local product structure on $\cN_{{\rho_0}|X(x)|}(x)$  that is invariant under the maps $\cP_{1,x}$, and $[\cdot,\cdot]$ is well-defined. It is worth pointing out that these foliations are likely not invariant under $\cP_{t,x}$ for $t\notin \ZZ$.

\subsubsection{A coordinate system near singularities}

The main issue with the foliations defined in the previous subsection is that they are only well-defined and invariant within the scale ${\rho_0}|X(x)|$ which becomes arbitrarily small when $x$ is close to $\Sing(X)$. For $y\in B_{t,\vep}(x)$, as the trajectories of  $y$ and $x$ approach singularities and consequently $|X(x)|\ll \vep$, we do not know {\em a priori} whether $[\cdot,\cdot]$ is always well-defined at every point in the orbit of $y$. To solve this issue, we introduce a coordinate system consisting of three foliations, well-defined in an open neighborhood of each singularity. The tools in Section~\ref{ss.cone} will allow us to transition from one coordinate system to the other.

Let $\sigma\in\Lambda$ be a singularity.  Recall that all singularities in $\Lambda$ are  Lorenz-like (Proposition~\ref{p.sh}): there is a dominated splitting $T_\sigma \bM =  E_\sigma^{ss}\oplus E_\sigma^c\oplus E_\sigma^u$ under the tangent flow with $\dim E_\sigma^c=1$, such that the eigenvalue $\lambda^c$ on $E^c_\sigma$ satisfies $|\lambda^c|<1$. Furthermore, $E^{ss}_\sigma = E^{ss}(\sigma)$, $E_\sigma^c\oplus E_\sigma^u=F^{cu}(\sigma)$ where $E^{ss}(\sigma)$ and $F^{cu}(\sigma)$ are given by the sectional-hyperbolic splitting on $\Lambda$. Without loss of generality, we will assume that $E_\sigma^{ss},  E_\sigma^c$ and $E_\sigma^u$ are pairwise orthogonal. This is consistent with our assumption that $E^{ss}$ and $F^{cu}$ are orthogonal at every point.

On $T_\sigma\bM$, we apply Theorem~\ref{t.fakeleaves} (or, more precisely, \cite[Proposition 3.1]{BW}) with the constant $\alpha>0$, $r_0$ given by Lemma~\ref{l.nearEc}  (we assume that $r_0$ is less than the injectivity radius), $n=\dim \bM$,  the index set $I=\{\sigma\},$ $A_\sigma = Df_1|_{T_\sigma\bM}$ and the map 
$$
g_\sigma:= \begin{cases}
	\exp_\sigma^{-1}\circ f_1\circ\exp_\sigma, &|u|<\upsilon\\
	h, &|u|\in[\upsilon,K\upsilon]\\
	Df_{1}(\sigma), &|u|>K\upsilon
\end{cases}
$$
where $\upsilon$ is small enough, $K$ large enough and $h$ is a smooth bump function such that $\|g_\sigma -A_\sigma\|_{C^1}<\iota$ where $\iota$ is the constant given by Theorem~\ref{t.fakeleaves}. This gives us foliations $\cF_\sigma^*$, $*=s,c,u,cs,cu$ tangent to the corresponding $\alpha$-cones, with $\cF^{s}_\sigma$ and $\cF^c_\sigma$ sub-foliating $\cF^{cs}_\sigma$, and $\cF^{c}_\sigma$, $\cF^u_\sigma$ sub-foliating $\cF^{cu}_\sigma$; we then push these foliations to the ball $B_{r_0}(\sigma)$ by the exponential map $\exp_\sigma$.  Here we decrease $r_0$ if necessary such that the resulting foliations, still denoted by $\cF_\sigma^*$, are well defined in $B_{r_0}(\sigma)$ and are invariant under the time-one map $f_1$ in the following sense: 
$$
f_1\left(\cF_\sigma^{s}(y)\right)\subset \cF_\sigma^s(y_1) ,\,\,\mbox{ for }y\in B_{r_0}(\sigma) \mbox{ such that }y_1\in B_{r_0}(\sigma). 
$$
Similar statements hold for $\cF^*_\sigma$, $*=c,u,cs,cu$.  It is worth noting that these foliations may not be invariant under $f_t$ for $t\notin \ZZ$. 

Since there are only finitely many singularities in $\Lambda$, we will assume that $r_0$ is taken so that the aforementioned foliations are well-defined in $B_{r_0}(\sigma)$ for every $\sigma\in\Sing(X)\cap\Lambda$.

\subsection{Two key lemmas}\label{ss.nearsing}

Let $\cF^s_{x,\cN},\cF^{cu}_{x,\cN}$ be the fake foliations on the normal plane $\cN_{\rho_0|X(x)|}(x)$, and recall that $[\cdot,\cdot]$ is the local product structure defined on $\cN_{\rho_0|X(x)|}(x)$. To simplify notation, for $y\in \cN_{\rho_0|X(x)|}(x)$ we denote by 
$$
y^E = y^{\cF^s_{x,\cN}}, \mbox{ and }y^F = y^{\cF^{cu}_{x,\cN}};
$$
so we have $\{y\} = [y^E,y^F]$. 
For $0<r<r_0$, recall that $W_r(\sigma)$ is the neighborhood of $\sigma\in \Sing(X)\cap\Lambda$ defined by~\eqref{e.W.def}, and $W_r=\cup_{\sigma\in\Sing(X)\cap\Lambda} W_r(\sigma)$. See Figure~\ref{f.EF}. 

From now on, to simplify notation and highlight the structure of the proof, we make the following definition:
\begin{definition}\label{d.EFcomparison}
	Let $x\in\Reg(X)$ and $y\in \cN_{\rho_0|X(x)|}(x)$. Let $y^E, y^F\in \cN(x)$ be such that $y=[y^E,y^F]$. We define:
	$$
	d^E_x(y) = d_{\cF_{x,\cN}^s}( y^{F}, y),\,\, \mbox{ and }\,\, d^F_x(y) = d_{\cF_{x,\cN}^{cu}}( x,y^{F}). 
	$$ 
	We will sometimes refer to $d^E_x(y)$ and $d^F_x(y)$ as the $E$-length and $F$-length of $y$. See Figure~\ref{f.EF}. 
\end{definition}

The following two lemmas play central roles in the proof of the Bowen property.  For $x\in\partial W_r^-\cap\Lambda$ we write $t(x) = \sup\{s>0: (x,s)\in W_r\}$ and $x^+= x_{t(x)}\in \partial W_r^+$; take $y\in B_{t(x),\vep}(x)\cap \Lambda$ and let $y^+ = \cP_{x^+}(y_{t(x)})\in \cN_{\rho_0|X(x^+)|}(x^+)$. Then, the following two lemmas  state that
\begin{enumerate}
	\item if upon {\em entering} $W_{r }$ we have the $F$-length of $y$ larger than the $E$-length, then the same is true when the orbit of $y$ leaves $W_{r }$;
	\item on the other hand, if upon {\em leaving} $W_{r }$ we have the $E$-length of $y^+$ larger than the $F$-length, then the same is true when $y$ enters $W_{r }$;
	\item in the second case, the orbit of $y$ is scaled shadowed by the orbit of $x$ before leaving $W_r$;
	\item in both cases, we see forward/backward expansion on the larger coordinate at an exponential speed.
\end{enumerate}


%

Recall the definition of $\rho$-scaled shadowing from Definition~\ref{d.shadow1} and its properties from Lemma \ref{l.shadowing2}.
\begin{lemma}\label{l.Elarge}
	For all $r_0>0$ sufficiently small and every $\rho\in (0,\rho_0]$, there exist $\overline r\in(0,{r_0}),\vep_1>0$ such that for $0<r <\overline r$ and $0<\vep<\vep_1$, the following holds:\\
	Let  $x\in \partial W_{r }^-\cap\Lambda$ and $ x^+ =f_{t(x)}(x)$. For $y\in B_{t(x),\vep}(x)\cap \cN(x)$, assume that $y=[y^E,y^F]$ and that $y^+ = [y^{+,E}, y^{+,F}]$  where $y^+=\cP_{x^+}(y_{t(x)})\in \cN_{\rho_0|X(x^+)|}(x^+)$. Finally, assume that 
	\begin{equation}\label{e.Elarge}
		d^E_{x^+}(y^+)\ge d^F_{x^+}(y^+) , \hspace{0.5cm} \mbox{($y^+$ has large $E$-length at $x^+$)}
	\end{equation}
	then we have the following statements:
	\begin{enumerate}
		\item ($y$ has large $E$-length at $x$) we have $d^E_{x}(y)\ge d^F_{x}(y) $;
		\item (backward expansion on $E$-length) there exists $\lambda_\sigma>1$ such that 
		$$d^E_{x}(y)\ge\lambda_\sigma^{t(x)}d^E_{x^+}(y^+).$$
		\item  (scaled shadowing until leaving $W_{r }$) $y$ is $\rho$-scaled shadowed by the orbit of $x$ up to  time $t(x)$.
	\end{enumerate}
\end{lemma}
See Figure \ref{f.6.7}.
\begin{figure}[h!]
	\centering
	\includegraphics[scale=0.95]{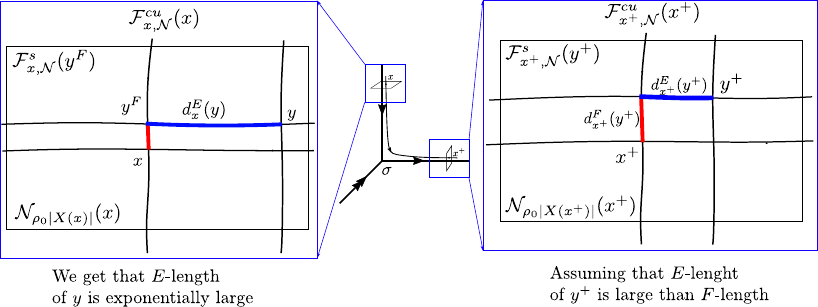}
	\caption{Backward exponential expansion of $E$-length.}
	\label{f.6.7}
\end{figure}

Similarly, we have 	

\begin{lemma}\label{l.Flarge}
	For all $r_0>0$ sufficiently small, there exist $\overline r\in(0,{r_0}),\vep_1>0$ such that for $0<r <\overline r$ and $0<\vep<\vep_1$, the following holds:\\
	Let  $x\in \partial W_{r }^-\cap\Lambda$ and $ x^+ = f_{t(x)}(x)$. For $y\in B_{t(x),\vep}(x)\cap \cN(x)$, assume that $y=[y^E,y^F]$ and that $y^+ = [y^{+,E}, y^{+,F}]$  where $y^+=f_{\overline t_x(y)}(y)\in \cN(x^+)$. Finally assume that 
	\begin{equation}\label{e.Flarge}
		d^F_x(y)\ge d^E_x(y), \hspace{1cm} \mbox{($y$ has large $F$-length at $x$)}
	\end{equation}
	then we have the following statements:
	\begin{enumerate}
		\item ($y^+$ has large $F$-length at $x^+$) we have $$ d^F_{x^+}(y^+)\ge d^E_{x^+}(y^+);$$
		\item (forward expansion on $F$-length) there exists $\lambda_\sigma>1$ such that 
		$$ d^F_{x^+}(y^+)\ge\lambda_\sigma^{t(x)}d^F_{x}(y).$$
	\end{enumerate}
	
\end{lemma}

It is worth noting that the second lemma does not follow from the first lemma applied to $-X$, since $\sigma$, as a singularity of $-X$, is not Lorenz-like. 
We also remark that the second lemma does not allow one to show the same scaled shadowing property as item (3) of the first lemma. This is due to the fact that the $F^{cu}$ bundle is not uniformly expanded by the flow.

The proofs of these lemmas are technical and therefore are left in Appendix~\ref{s.A2}.  We remark that the proof of these lemmas in fact do not rely on the fact that $\Lambda$ is sectional-hyperbolic. Indeed they only require $\sigma$ to be a Lorenz-like singularity, and that the orbit of $x$ enters $B_{r_0}(\sigma)$ through the center cone.

We end this section by remarking that these two lemmas are not needed (or at least the proof can be simplified) if one assumes that the flow can be smoothly linearized near each singularity. However, this would require the flow to be at least $C^2$ (see~\cite{Newhouse, Sternberg}), which we do not assume in this paper.

\section{The Bowen property}\label{s.bowen}
In this section we finish the proof of Theorem~\ref{t.Bowen}. 

\subsection{The Main Proposition}

In this subsection we state and prove the Main Proposition (Proposition~\ref{p.key}), which is essential for the proof of the Bowen property. Recall that given $0<r <{r_0}$ and a singularity $\sigma$, $W_{r }(\sigma)$ is the union of local orbits that are contained in $B_{r_0}(\sigma)$ and intersects with $B_{r }(\sigma)$; see~\eqref{e.W.def}. Let $\rho_0$ and $K_0$ be given by Section~\ref{ss.liao}, and $\cP_x$ be defined as~\eqref{e.Px1} as before.

Now we recall the definition of $\cG_{B}$ from Section~\ref{ss.6.1}. 
Let $\lambda_0$ be given by Lemma~\ref{l.hyptime}. Given $\beta\in(0,1)$ and $0<r <{r_0}$, $\cG_B=\cG_B(\lambda_0,\beta, r_0, W_{r })$ is the collection of orbit segments $(x,t)$ such that:
\begin{enumerate}
	\item $x_t$ is a $(\lambda_0, cu,\beta,W_r)$-simultaneous Pliss time for the orbit segment $(x,t)$;
	\item $x_0=x\notin W_{r }$.
\end{enumerate}
In particular, Remark~\ref{r.RecPliss} shows that $x_t\notin W_r$. 

Below we will assume that $a>0$ is taken small enough such that 
\begin{equation}\label{e.lambda1}
	\lambda_1:=\frac{\lambda_0}{a+1}>1. 
\end{equation}
The next lemma is a useful property for $(\lambda_0,cu)$-hyperbolic times.
\begin{lemma}\label{l.cuhyptime1}
	There exist $\rho_1\in (0,\rho_0)$ and $\alpha_0>0$ such that every $\rho\in(0,\rho_1]$ and $\alpha\in (0,\alpha_0]$ have the following property: 
	
	Let $x$ be a regular point (not necessarily in $W_r^c$) and $x_t$ be a $(\lambda_0,cu)$-hyperbolic time for orbit segment $(x,t)$ under the scaled linear Poincar\'e flow. Let $y\in \cN_{\rho|X(x)|}(x)$ be $\rho$-scaled shadowed by the orbit of $x$ up to time $t$. Then, for every $j=0,\ldots,\floor{t}$ and every vector $v\in C_\alpha(F^{cu}_N(y))$ one has 
	\begin{equation}\label{e.cucont}
		|D_{\cP_{x_t}(y_t)}\cP_{-j, x_t}(u)|\le 4 \frac{|X(x_{t-j})|}{|X(x_t)|}\lambda_1^{-j} |u|.
	\end{equation}
	where $u= D_y\cP_{t,x}(v)$. 
\end{lemma}
We remark that the cone field $C_\alpha(F^{cu}_N)$ is only forward invariant. This is why we do not directly take $u\in C_\alpha(F^{cu}(\cP_{x_t}(y_t))). $
\begin{proof}
	Without loss of generality, we assume that for all $x\in\bM$ it holds 
	$$
	\frac12 \le \inf_ym(D_y\exp_x) <\sup_y\|D_y\exp_x\|\le 2.
	$$
	where the supremum and infimum are taken over $y\in B_{\mathfrak{d}_0}(x)$ with $\mathfrak{d}_0$ being the injectivity radius.  
	Recalling the definition of $P_{1,x}$ on the normal bundle by~\eqref{e.p}, we only need to prove~\eqref{e.cucont} for the map $P_{1,x}$ without the coefficient $4$. 
	
	Let $y\in\cN_{\rho|X(x)|}(x)$ be $\rho$-scaled shadowed by $x$ up to time $t$. To simplify notation, we write 
	$$
	\tilde y^i = \cP_{x_i}(y_i)
	$$
	and $\overline y^i = \exp_{x_i}^{-1}(y_i)$ for its image on the normal bundle. Then from Proposition \ref{p.Fx} we have $\overline y^i\in N_{3 \rho|X(x_i)|}(x_i)$, and
	$$
	P_{1,x_i}(\overline y^i) = \overline y^{i+1}. 
	$$
	If we further define $\overline y^{i,*} =  \frac{\overline y^i}{|X(x_i)|}$, then
	$$
	P^*_{1,x_i}(\overline y^{i,*}) = \overline y^{i+1,*},
	$$
	where $P^*_{t,x}$ is defined by \eqref{e.p*}.
	Furthermore, we have  $|y^{i+1,*}|<3\rho$ by hypothesis, for all $i=0,\ldots,t$ (for simplicity we assume that $t\in\NN$).
	
	Recall that $D_z P^*_{1,x}$ is uniformly continuous at a uniform scale (Proposition~\ref{p.tubular3} (2)). Then, for every $a$ sufficiently small,  there exist $\alpha_0>0$ and $\rho_1>0$ such that for all $\rho \le \rho_1$, $\alpha\le\alpha_0$, $v\in C_\alpha(F^{cu}_N(\overline y^{i,*}))$ and $u= DP^*_{1, x_{i}}(v)\in C_\alpha(F^{cu}_N(\overline y^{i+1,*}))$, we have 
	$$
	|D_{\overline y^{i,*}} P^*_{-1, x_{i}}(u)|\le  (1+a) \left\|D_{\overline x_i} P^*_{-1, x_{i}}|_{F^{cu}_N(x_i)}\right\|  |u| =  (1+a)\left\|\psi^*_{-1,x_i}|_{F^{cu}_N(x_i)}\right\|  |u|.
	$$
	This implies that 		
	\begin{align*}
		|D_{\overline y^{t,*}} P^*_{-(t-j), x_t}(u)|&= \left|\left( \prod_{i=0}^{t-j-1}D_{\overline y^{t-i,*}} P^*_{-1, x_{t-i}}\right)(u)\right| \\
		&\le \prod_{i=0}^{t-j-1} \left\|D_{\overline y^{t-i,*}} P^*_{-1, x_{t-i}}|_{F^{cu}_N(x_{t-i})}\right\||u|\\
		&\le (1+a)^{t-j}\prod_{i=0}^{t-j-1}\left\|\psi^*_{-1,x_{t-i}}|_{F^{cu}_N(x_{t-i})}\right\||u|\\
		&\le \left((1+a)\lambda_0^{-1}\right)^{t-j}|u|
	\end{align*}
	where the last inequality follows from the assumption that $x_t$ is a $(\lambda_0,cu)$-hyperbolic time under the scaled linear Poincar\'e flow. This shows that 
	$$
	|D_{\overline y^{t}} P_{-(t-j), x_t}(u)|\le \frac{|X(x_j)|}{|X(x_t)|}\left((1+a)\lambda_0^{-1}\right)^{t-j}|u| =  \frac{|X(x_j)|}{|X(x_t)|}\left(\lambda_1^{-1}\right)^{t-j}|u|,
	$$
	and the lemma follows.

\end{proof}

From now on, $\rho$, the scale for the shadowing property, will be taken from $(0,\rho_1]$. 
Let $\rho'$ be given by Lemma~\ref{p.tubular4} applied to $\rho K_0^{-2}$. In particular, if $y\in \cN_{\rho'|X(x)|}(x)$ then $\tilde y_{-1}:=\cP_{-1,x_{-1}}(y)$ exists and is contained in $\cN_{\rho K_0^{-2}|X(x_{-1})|}(x_{-1})$; consequently, $\tilde y_{-1}$ is $\rho$-scaled shadowed by the orbit of  $x_{-1}$ up to time $T=2$. We also assume that $\alpha_0$ obtained from the previous lemma satisfies the discussion in the previous section, namely $1000\alpha_0\cdot \max\{L_0, L_1,L_2\}<0.01$. We therefore obtain $r_0$ and $\overline r$ from Lemma~\ref{l.Elarge} and~\ref{l.Flarge} for $\alpha=\alpha_0$.

The next proposition is the key for the Bowen property on $\cG_B$. 

\begin{proposition}[The Main Proposition]\label{p.key}
	There exists $\beta_0\in(0,1)$, such that for every $\beta\in(0,\beta_0]$ and $r_0>0$ small enough, there exists $\overline r\in (0,{r_0})$, such that for every $r \in (0,\overline r]$, there exists $\vep_2>0$ such that for every  $(x,t)\in \cG_B(\lambda_0,\beta, r_0, W_{r })$, we have that 
	every $y\in \cP_{x}\left(B_{t,\vep_2}(x)\right)$ is $\rho$-scaled shadowed by the orbit of $x$ up to  time $t$.
\end{proposition}
\begin{remark}
	Thinking of the Bowen ball $B_{t,\vep}(x)$ as
	$$
	B_{t,\vep}(x)=\bigcap_{s=0}^tf_{-s}(B_\vep(f_s(x))),
	$$ 
	we see the discrepancy between the Bowen ball and Liao's tubular neighborhood: while the Bowen ball consists of points $y$ that are shadowed by $x$ at a uniform scale $\vep$,  Liao's tubular neighborhood consists of points that are shadowed at a uniformly relative scale, proportional to the flow speed.  The Main Proposition states that, despite the difference between their scales, for ``good'' orbits in $\cG$, points in the Bowen ball are indeed  {\bf scaled} shadowed by $x$.
\end{remark}

\subsection{Preparation and estimates outside $W_r$}

We make a few remarks and observations before starting the proof of Proposition~\ref{p.key}.

\begin{enumerate}
	\item Throughout the proof we will assume that $r_0$ is sufficiently small so that Lemma~\ref{l.Elarge} and~\ref{l.Flarge} are applicable. In particular, $\overline r$ will be given by Lemma~\ref{e.Elarge} applied to $r_0$ and $\rho$.  
	\item Given $r$ and $r_0$, by Remark~\ref{r.RecPliss}, for every $(x,t)\in\cG_B$ we have $x_t\notin W_{r }$. Consequently, such $x_t$'s are bounded away from $\Sing(X)$. As a result, there exists $\vep_2=\vep_2(r,r_0)>0$ such that the foliations $\cF_{x_t, \cN}^i,i=s,cu$ are well-defined on $\cN_{\vep_2}(x_t)$; furthermore, $\cF_{x_t, N}^i,i=s,cu$ form a local product structure in $\cN_{\vep_2}(x_t)$, for every $(x,t) \in \cG_B$. As a result, for $y\in B_{t,\vep_2}(x)$, the point $\tilde y_t:=\cP_{x_t}(y_t)$ belongs to $\cN_{\vep_2}(x_t)$ and, therefore, the points $(\tilde y_t)^E, (\tilde y_t)^F$ exist and satisfy $\tilde y_t = [(\tilde y_t)^E, (\tilde y_t)^F]$.   A similar argument can be made near the starting point $x$ since those points are also outside $W_{r }$.
	\item As in the previous item, points outside $W_{r }$ are bounded away from all singularities; this further indicates that $\inf_{z\in (W_{r })^c}\{|X(z)|\}>0.$ Consequently, we may resort to a smaller $\vep_2$ such that 
	$$
	\cP_z(B_{\vep_2}(z))\subset \cN_{\rho K_0^{-1}|X(z)|}(z), \,\,\forall z\in (W_{r })^c. 
	$$
	This forces every $y\in \cP_z(B_{\vep_2}(z))$ to be $\rho$-scaled shadowed by the orbit of  $z$ up to time $T = 1$. The same argument can be applied again at $z_1$, as long as $z_1 \in (W_{r })^c$.
	As a result, if an orbit segment $(x,t)$ does not enter $W_{r }$, then the choice of $\vep_2$ above guarantees that 	every $y\in \cP_{x}\left(B_{t,\vep_2}(x)\right)\subset  \cP_{x}\left(B_{\vep_2}(x)\right)\subset \cN_{\vep_2}(x)$ is $\rho$-scaled shadowed by the orbit of $x$ up to time $t$ (recall Lemma~\ref{l.shadowing2}). In other words, the conclusion of the proposition may only fail for those orbit segments $(x,t)$ that visit $W_{r }$. 
	\item Because of the previous item, given any orbit segment $(x,t)\in\cG_B$, we have a collection of times 
	$$
	0\le t_1^-< t_1^+ < t_2^-<t_2^+<\cdots < t_k^-<t_k^+\le t,
	$$
	where  $t_i^\pm$ are chosen so that for every $i=1,\ldots, k$:
	\begin{itemize}
		\item  there is a singularity $\sigma_i\in\Sing(X)$ such that $x_{t_i^-}\in \partial W^-_{r }(\sigma_i)$, and $x_{t_i^+}\in \partial W^+_{r }(\sigma_i)$;
		\item the (open) orbit segment $(x_{t_i^-}, t_i^+-t_i^-)$ belongs to $W_{r }(\sigma_i)$; and
		\item the (closed) orbit segment $(x_{t_i^+}, t_{i+1}^--t_i^+)$ are outside $W_{r }$. 
	\end{itemize}
	 Later we will see that if $x_t$ is a $\beta$-recursive Pliss time, then the time ratio $\frac{t^+_i - t^-_i}{t-t_i^-}$ must be small. This allows the hyperbolicity from $(t_i^+,t)$ to overcome the lack of estimate inside $(t^-_i,t^+_i)$ for every $i$.
\end{enumerate}

It should be noted that $d^E_x(y)$ is defined using $y^F$ and $y$, as opposed to $x$ and $y^E$ (although our proof can be modified if one chooses to use $y^E$ instead). As a result, the point $y^E$ is not used in the definition. This is due to how Lemma~\ref{l.Elarge} and~\ref{l.Flarge} are stated. 

Also, it is worth keeping in mind that $y\in B_{t,\vep_2}(x)$ does not imply {\em a priori} that $d^i_{x_s}(\cP_{x_s}(y_s)),i=E,F$ are well-defined for every $s\in [0,t]$. This is because the projection $\cP_{x_s}(\cdot )$ is only defined in $B_{\rho_0|X(x_s)|}(x_s)$, and its domain could be smaller than $\vep_2$ if $x_s$ is close to $\Sing(X)$ (the goal of the Main Proposition is to show that this does not happen); also recall that the product structure on $\cN(x_s)$ only exists at a scale of $\rho_0|X(x_s)|$ which could also be less than $\vep_2$. To solve this issue, we will resort to Lemma~\ref{l.Elarge} and~\ref{l.Flarge} when dealing with orbit segments inside $W_{r }$. 

We take 
\begin{equation}\label{e.vep2}
	\vep_2 = \vep_2(r ,r_0,\rho) = \frac12 \min\left\{ \rho K_0^{-1} \inf_{z\in W_{r }^c}\{|X(z)|\},\,\,\vep_1\right\}.
\end{equation}
where $\vep_1$ is given by Lemma~\ref{l.Elarge} and~\ref{l.Flarge}.

The next lemma summarizes (3) in the previous discussion and provides the desired expansion/contraction on $\cF_{x,\cN}^i,i=s,cu$ for orbit segments that do not enter $W_{r }$. 
\begin{lemma}\label{l.reg}
	There exists  $\lambda_R>1$, such that for every $\rho\in(0,\rho_1)$ and $0<r <\overline r$ and every orbit segment $(x,t)$ that does not intersect with $W_{r }$, we have 
	\begin{enumerate}[label=(\alph*)]
		\item for every $y\in B_{t,\vep_2} (x)$, $\cP_x(y)$ is $\rho$-scaled shadowed by the orbit of $x$ up to  time $t$;
		\item assume that 
		\begin{equation}\label{e.Flarge1}
			d^F_x(P_x(y))\ge d^E_x(P_x(y)), \hspace{1cm} \mbox{(large $F$-length  	near $x$)}
		\end{equation}
		then we have 
		\begin{equation*}
			d^F_{x_t}(P_{x_t}(y_t))\ge \lambda_R^t\cdot d^E_{x_t}(P_{x_t}(y_t)); \hspace{1cm} \mbox{(exponentially large $F$-length 	near $x_t$)}
		\end{equation*}
		\item  assume that 
		\begin{equation}\label{e.Elarge1}
			d^E_{x_t}(P_{x_t}(y_t))\ge d^F_{x_t}(P_{x_t}(y_t)), \hspace{1cm} \mbox{(large $E$-length  	near $x_t$)}
		\end{equation}
		then we have 
		\begin{equation*}
			d^E_x(P_x(y))\ge \lambda_R^t\cdot d^F_x(P_x(y)). \hspace{1cm} \mbox{( exponentially large $E$-length	near $x$)}
		\end{equation*}
	\end{enumerate} 
\end{lemma}

\begin{proof}
	As explained earlier, we take 
	$$
	\vep_2< \frac{\rho}{K_0} \inf_{z\in W_{r }^c}\{|X(z)|\}. 
	$$
	This choice of $\vep_2$ indicates that if $y\in B_{t,\vep_2}(x)$, then for every $s\in [0,t]$ we have  
	$$
	\cP_{x_s}(y_s) \in \cN_{\vep_2}(x_s)\subset \cN_{\rho K_0^{-1} |X(x_s)|}(x_s). 
	$$
	Also note that if $s\in [0,t]\cap \NN$, then $\cP_{x_s}(y_s) = \cP_{1,f_{s-1}(x)}\circ \cdots\circ \cP_{1,x}(P_x(y))$. Therefore $\cP_x(y)$ is $\rho$-scaled shadowed by the orbit of $x$ up to  time $t$, according to Lemma~\ref{l.shadowing2}.
	
	(b) and (c) immediately follow from the uniform continuity of $D\cP_{1,t}$ (Proposition~\ref{p.tubular3}) and the fact that $E_N^s\oplus F^{cu}_N$ is a dominated splitting for the scaled linear Poincar\'e flow $\psi^*_t$; see Lemma~\ref{l.splitting}. The reason that we do not need a constant $C$ is due to the one-step domination~\eqref{e.onestepDS1} with constant $\frac12$. In particular, $\lambda_R$ can be taken close to 2.  
\end{proof}

Here we remark that the {\em  ratio} between the $E$ and $F$-length is a result of the dominated splitting between $E_N^{s}$ and $F^{cu}_N$, and does not rely on the is unrelated to the flow speed $|X(x)|$. The flow speed is only relevant when calculating the actual {\em value} of the $E$ and $F$-length.

\subsection{An important lemma}

The rest of this section is devoted to the proof of Proposition~\ref{p.key}. 
As explained before, every orbit segment $(x,t)\in \cG_B$ can be parsed as 
$$
0\le t_1^-< t_1^+ < t_2^-<t_2^+<\cdots < t_k^-<t_k^+\le t.
$$
In particular,  $x_{t_1^+}$ is on $\partial W_{r }^+$. The choice of $\vep_2$ by~\eqref{e.vep2} implies that  $\cP_{x_{t_1^+}}$ is well-defined in $B_{\vep_2}(x_{t_1^+})$; furthermore, $\cP_{x_{t_1^+}}(y_{t_1^+})$ is contained in $\cN_{\rho K_0^{-1}|X(x_{t_1^+})|}(x_{t_1^+})$; this makes the $E$ and $F$-length of $\cP_{x_{t_1^+}}(y_{t_1^+})$ well-defined. 

We first consider a special case, namely the $F$-length is larger than the $E$-length at $x_{t_1^+}$.

\begin{lemma}\label{l.Flarge1}
	Proposition~\ref{p.key} holds   under the extra assumption that 
	\begin{equation}\label{e.Flarge2}
		d^F_{x_{t_1^+}}\left(\cP_{x_{t_1^+}}(y_{t_1^+})\right)\ge d^E_{x_{t_1^+}}\left(\cP_{x_{t_1^+}}(y_{t_1^+})\right). \hspace{1cm} (\mbox{Large $F$-length	at $x_{t_1^+}$})
	\end{equation}
\end{lemma}

\begin{proof}
	We apply Lemma~\ref{l.Flarge} to obtain $\overline r$, and fix any $r <\overline r$.  Also note that by~\eqref{e.vep2}, $\vep_2<\vep_1$ where $\vep_1$ is given by Lemma~\ref{l.Flarge}.  We will assume, without loss of generality, that $t\in\NN$. 
	We also take $\beta_0$ sufficiently close to zero, such that 
	\begin{equation}\label{e.beta0'}
		\left((K_1')^{\frac{\beta_0}{1-\beta_0}}\lambda_1^{-1}\right)^{\beta_0^{-1}-1}\rho_1<\frac{1}{16}\rho', 
	\end{equation}
	where $K_1'$ is given by Proposition~\ref{p.tubular3} and $\lambda_1$ is given by~\eqref{e.lambda1}. 
	Here~\eqref{e.beta0'} is possible because the base 
	$$
	(K_1')^{\frac{\beta_0}{1-\beta_0}}\lambda_1^{-1}<1
	$$
	as long as $\beta_0$ is close to zero, meanwhile $(\beta_0^{-1}-1)$ can be made arbitrarily large. Then we fix any $\beta\in(0,\beta_0]$. 
	The reason behind this choice will be clear near the end of the proof.

	To start with, we apply Lemma~\ref{l.reg} (b) to the orbit segment $(x_{t_1^+}, t_2^--t_1^+)$. This orbit segment is outside $W_{r }$. 	  We therefore  have 
	$$
	d^F_{x_{t_2^-}}\left(\cP_{x_{t_2^-}}(y_{t_2^-})\right)\ge d^E_{x_{t_2^-}}\left(\cP_{x_{t_2^-}}(y_{t_2^-})\right). \hspace{1cm} (\mbox{Large $F$-length at } x_{t_2^-})
	$$
	
	Now we can apply Lemma~\ref{l.Flarge} to the orbit segment $(x_{t_2^-}, t_2^+-t_2^-)$ which is inside $W_{r }$ and obtain 
	$$
	d^F_{x_{t_2^+}}\left(\cP_{x_{t_2^+}}(y_{t_2^+})\right)\ge d^E_{x_{t_2^+}}\left(\cP_{x_{t_2^+}}(y_{t_2^+})\right). \hspace{1cm} (\mbox{Large $F$-length at } x_{t_2^+})
	$$
	Then we are in a position to apply Lemma~\ref{l.reg} (b) again on the orbit segment $(x_{t_2^+}, t_3^- - t_2^+)$ which is outside $W_{r }$. Following this recursive argument, we see that for $i=1,\ldots, k$, 
	\begin{equation}\label{e.Flarge3}
		d^F_{x_{t_i^\pm}}\left(\cP_{x_{t_i^\pm}}(y_{t_i^\pm})\right)\ge 	d^E_{x_{t_i^\pm}}\left(\cP_{x_{t_i^\pm}}(y_{t_i^\pm})\right). \hspace{1cm} (\mbox{Large $F$-length at every } x_{t_i^\pm})
	\end{equation}	
	We remark that we indeed get that $y_s$ has large $F$-length at every $s\in [t_i^+,t_{i+1}^-]$ as well as at $x_t$, but we will not use this fact.
	
	For contradiction's sake, we define 
	\begin{align*}
	s_1=\min\{s\in[0,t]: \exists \tilde s\in(0,t) \mbox{ such that } y_{\tilde s} \in \cN_{\rho_0|X(x_s)|}(x_s) \\
		\mbox{ is $\rho$-scaled shadowed by $x_s$ up to time $t-s$}\}. 
	\end{align*}
	Due to Lemma~\ref{l.reg} (a), $s_1\in [t_i^-,t_i^+]$ for some $i\in[1,k]\cap\NN$. 


	Without loss of generality we may assume that $s_1\in\NN$ (otherwise round $s_1$ to the next integer, and the proof below still follows). We also let 
	$$
	s_0: = s_1 -1. 
	$$
	Below we will prove that 
	$$
	y_{\tilde s}\in \cN_{\rho'|X(x_{s_1})|}(x_{s_1})
	$$
	where $\rho'$ is given by Lemma~\ref{p.tubular4} applied to $\rho K_0^{-2}$; this shows that  $\cP_{1,x_{s_0}}^{-1}\left(y_{\tilde s}\right)$ exists and is contained in $\cN_{\rho K_0^{-2}|X(x_{s_0})|}(x_{s_0}); 
	$
	this, together with Lemma~\ref{l.shadowing2}, will give the desired contradiction.

	To this end, we write 
	$$
	t^+ = \roof{t_i^+} < t.
	$$
	Since $t^+>s_1$, the orbit of $\cP_{x_{t^+}}(y_{t^+})$ (which is well-defined because $x_{t^+}$ is not in $W_r$) is $\rho$-scaled shadowed by the orbit of $x_{t^+}$ up to the  time $t-t^+$. Also recall that $x_t$ is a $(\lambda_0, cu)$-hyperbolic time for the orbit segment $(x,t)$ under the scaled linear Poincar\'e flow $\psi^*$. This allows us to apply Lemma~\ref{l.cuhyptime1} with $j=t^+$
	and obtain
	\begin{align*}
		d_{x_{t^+}}^F\left(\cP_{x_{t^+}}(y_{t^+})\right)&\le4 \frac{|X(x_{t^+})|}{|X(x_t)|}\lambda_1^{-(t-t^+)}d_{x_{t}}^F\left(\cP_{x_{t}}(y_{t})\right)\\
		&\le4  \frac{|X(x_{t^+})|}{|X(x_t)|}\lambda_1^{-(t-t^+)}\vep_2\\
		&\le 4\frac{|X(x_{t^+})|}{|X(x_t)|}\lambda_1^{-(t-t^+)}\rho_1|X(x_t)|\\
		&=4\lambda_1^{-(t-t^+)}\rho_1|X(x_{t^+})|.
	\end{align*}

	Here the second line is due to  $\cP_{x_{t}}(y_{t})\in \cN_{\vep_2}(x), $ and the next line is due to the choice of $\vep_2$; see~\eqref{e.vep2}.
	In particular, by~\eqref{e.Flarge3} we have 
	\begin{align*}
		d_{\cN(x_{t^+})}\left(x_{t^+}, \cP_{x_{t^+}}(y_{t^+}) \right) &\le d_{x_{t^+}}^F\left(\cP_{x_{t^+}}(y_{t^+})\right) + d_{x_{t^+}}^E\left(\cP_{x_{t^+}}(y_{t^+})\right)\\
		&\le 2 d_{x_{t^+}}^F\left(\cP_{x_{t^+}}(y_{t^+})\right)\\
		\numberthis\label{e.dist}	&\le 8\lambda_1^{-(t-t^+)}\rho_1|X(x_{t^+})|. 
	\end{align*}
	
	To estimate the change of distance on the normal plane along the orbit segment between the times $s_1$ and $t^+$, recall that $K'_1>1$ is the upper-bound of $|D(P^*_{1,x})^{-1}|$ given by Proposition~\ref{p.tubular3}. Then we have  
	\begin{equation}\label{e.Dsmall}
		\|D(\cP_{(t^+-s_1), x_{s_1}})^{-1}\|\le \frac{|X(x_{s_1})|} {|X(x_{t^+})|}(K_1')^{t^+-s_1}. 
	\end{equation}
	Combining~\eqref{e.dist} and~\eqref{e.Dsmall}, we obtain
	\begin{align*}
		d_{\cN(x_{s_1})}\left(x_{s_1}, y_{\tilde s}) \right) &\le
		\frac{|X(x_{s_1})|} {|X(x_{t^+})|}(K_1')^{t^+-s_1} d_{\cN(x_{t^+})}\left(x_{t^+}, \cP_{x_{t^+}}(y_{t^+}) \right) \\
		&\le 8\frac{|X(x_{s_1})|} {|X(x_{t^+})|}(K_1')^{t^+-s_1}\lambda_1^{-(t-t^+)}\rho_1|X(x_{t^+})|\\
		\numberthis\label{e.dist1}	&=8(K_1')^{t^+-s_1}\lambda_1^{-(t-t^+)}\rho_1|X(x_{s_1})|.
	\end{align*}
	
	On the other hand, since $x_t$ is a $\beta$-recurrence Pliss time for the orbit segment $(x,t)$, we have 
	$$
	\frac{t^+-s_1}{t-s_1}<\beta,
	$$
	which leads to 
	\begin{equation}\label{e.smalltime}
		t^+-s_1\le \frac{\beta}{1-\beta} (t-t^+). 
	\end{equation}
	Combining~\eqref{e.smalltime} with~\eqref{e.dist1}, we see that 
	\begin{align*}
		d_{\cN(x_{s_1})}\left(x_{s_1}, y_{\tilde s} \right) &\le 8\left((K_1')^{\frac{\beta}{1-\beta}}\right)^ {(t-t^+)}\lambda_1^{-(t-t^+)}\rho_1|X(x_{s_1})|\\
		&=8\left((K_1')^{\frac{\beta}{1-\beta}}\lambda_1^{-1}\right)^ {(t-t^+)}\rho_1|X(x_{s_1})|\\
		&	\le\frac12 \rho'|X(x_{s_1})|,
	\end{align*}
	where the last inequality follows from~\eqref{e.beta0'} and the observation that 
	$$
	t-t^+>\beta^{-1}-1
	$$
	which is due to Remark~\ref{r.RecPliss}. As a result of Lemma~\ref{p.tubular4}, the point 
	$$
	\overline y:=\cP_{1,x_{s_0}}^{-1}\left(y_{\tilde s}\right) \in \cP_{1,x_{s_0}}^{-1}\left(\cN_{\rho'|X(x_{s_1})|}(x_{s_1})\right)
	$$ 
	exists and is contained in $\cN_{\rho K_0^{-2}|X(x_{s_0})|}(x_{s_0})$. By Proposition~\ref{p.tubular} and Remark~\ref{r.shadowing}, $\overline y$ is $\rho$-scaled shadowed by the orbit of $x_{s_0}$ up to  time $s=2$. Since $s_1=s_0+1$, and  $y_{\tilde s}$ is $\rho$-scaled shadowed by the orbit of $x_{s_1}$ up to  time $t-s_1$, by Lemma~\ref{l.shadowing2} we see that $\overline y$ is indeed $\rho$-scaled shadowed by the orbit of $x_{s_0}$ up to  time $t-s_0$. Because $s_0<s_1$, this contradicts the minimality of $s_1$, and concludes the proof of Lemma~\ref{l.Flarge1}.

\end{proof}

\subsection{Proof of Proposition~\ref{p.key}}
Finally we are ready to prove Proposition~\ref{p.key}. Assuming that the conclusion is not true, we write 
$$
s_1=\sup\left\{s\in[0,t]: \cP_{x}(y) \mbox{ is $\rho$-scaled shadowed by $x$ up to  time $s$}\right\}
$$	
(not to be confused with $s_1$ in the proof of Lemma~\ref{l.Flarge1}). Due to Lemma~\ref{l.reg} (a), there exists $i\in [1,k]\cap \NN$ such that 
$$
t_i^-<s_1<t_i^+.
$$
In particular, we see that $\cP_x(y)$ is $\rho$-scaled shadowed by the orbit of $x$ up to  time $t_i^-$, due to Lemma~\ref{l.shadowing} (1).

Now we consider the $E$ and $F$-length of $\cP_{x_{t_i^+}}(y_{t_i^+})$.  They exist due to the choice of $\vep_2$ by~\eqref{e.vep2}, and hence $\cP_{x_{t_i^+}}(y_{t_i^+})\in \cN_{\rho K_0^{-1}|X(x_{t_i^+})|}(x_{t_i^+})$. In particular, there are two possibilities:

\medskip
\noindent Case 1: Large $E$-length, i.e., 
\begin{equation*}
	d^E_{x_{t_i^+}}\left(\cP_{x_{t_i^+}}(y_{t_i^+})\right)\ge d^F_{x_{t_i^+}}\left(\cP_{x_{t_i^+}}(y_{t_i^+})\right).
\end{equation*}

In this case we apply Lemma~\ref{l.Elarge} to the orbit segment $(x_{t_i^-}, t_i^+ - t_i^-)\subset \overline W_{r }$ and the point $\cP_{x_{t_i^-}}(y_{t_i^-}) \in \cN_{\rho K_0^{-1}|X(x_{t_i^-})|}(x_{t_i^-})$. By Lemma~\ref{l.Elarge} (3), $\cP_{x_{t_i^-}}(y_{t_i^-})$ is $\rho$-scaled shadowed by the orbit of $x_{t_i^-}$ up to  time $t_i^+ - t_i^-$. As a result, $\cP_x(y)$ must be $\rho$-scaled shadowed by the orbit of $x$ up to time $t_i^+$, contradicting the maximality of $s_1$.

\medskip
\noindent Case 2: Large $F$-length, i.e., 
\begin{equation*}
	d^F_{x_{t_i^+}}\left(\cP_{x_{t_i^+}}(y_{t_i^+})\right)\ge d^E_{x_{t_i^+}}\left(\cP_{x_{t_i^+}}(y_{t_i^+})\right).
\end{equation*}

In this case we consider the orbit segment $(x_{t_i^-}, t-t_i^-)$, and note that it is contained in $\cG_B(\lambda_0,\beta,r_0,W_{r })$, since $x_t$ is a simultaneous Pliss time for $(x,t)$ and therefore also a simultaneous Pliss time for the orbit segment $(x_{t_i^-}, t-t_i^-)$. This enables us to apply Lemma~\ref{l.Flarge1} (and note that $t_i^+$ for the orbit segment $(x,t)$ corresponds to $t_1^+$ for the orbit segment $(x_{t_i^-}, t-t_i^-)$), and obtain that $\cP_{x_{t_i^-}}(y_{t_i^-})$ is $\rho$-scaled shadowed by the orbit of $x_{t_i^-}$ up to  time $t-t_i^-$. By Lemma \ref{l.shadowing}, $\cP_x(y)$ is $\rho$-scaled shadowed by the orbit of $x$ up to time $t$, once again contradicting the maximality of $s_1$. 

Since we have obtained contradictions in both cases, the proof of the main proposition is complete.\qed

\subsection{The Bowen property}
We conclude this section by proving the Bowen property on $\cG_B$.  To save letters and simplify notations, we use the lower case letters $c, c_1,c_2,\ldots$ to denote {\em local variables} that appear only during the proof of various lemmas. Their meaning may vary from the proof one lemma to the next.

\begin{proposition}[The Bowen Property]\label{p.Bowen}
	For any H\"older continuous function $\phi:\bM\to\RR$, $\cG_B$ has the Bowen property at scale $\vep_2$.
\end{proposition}

Before proving the proposition, we first state the following lemma.  Recall that 
$\displaystyle
\Phi_0(x,t) =\int_0^t\phi(f_s(y))\,ds
$
is the Birkhoff integral along the orbit $(x,t).$
\begin{lemma}\label{l.holder}
	Let $\phi$ be a H\"older continuous function with H\"older index $\gamma$. Then $\Phi_0(x,1)$, as a function of $x$, is also H\"older continuous with H\"older index $\gamma$.
\end{lemma}

\begin{proof}
	First we take $c>0$ large enough such that for every $x,y\in\bM$ and every $s\in [0,1]$, we have 
	$$
	d(x_s,y_s)<c\cdot d(x,y). 
	$$
	Let $c_1$ be the H\"older constant of $\phi$. Then we write
\begin{align*}
		|\Phi_0(x,1) - \Phi_0(y,1)|&\le \int_0^1  |\phi(x_s)-\phi(y_s)|  \,ds
		\le \int_0^1 c_1 \left(d(x_s,y_s)\right)^\gamma\,ds\\
		&\le c_1\int_0^1 c^\gamma \left(d(x,y)\right)^\gamma\,ds
		\le c_1c^\gamma \cdot \left(d(x,y)\right)^\gamma.
	\end{align*}

	This shows that $\Phi_0(x,1)$ is H\"older continuous with H\"older index $\gamma$, as desired. 
\end{proof}

\begin{proof}[Proof of Proposition~\ref{p.Bowen}]
	Let $(x,t)\in\cG_B$ and $y\in B_{t, \vep_2}(x)$. 
	Below we shall assume w.l.o.g.\ that $t\in \NN$ (otherwise let $t'=\floor{t}$ and consider the orbit segment $(x_{t-t'}, t')$ and $y_{t-t'}$; it is easy to see that this causes $G_0(x,t)$ to differ by  $\left|\int_0^{t-t'} g(x_s)\, ds\right| \le   \| g\|_{C^0}$ which does not depend on $(x,t)$ or $y$; the same can be said for $G_0(y,t)$). We shall also assume that $y\in \cN(x)$; this results an error of $G_0(y,t)$ by no more than $\|g\|_{C^0}$, which is again independent of $(x,t)$ and $y$.

	By Proposition~\ref{p.key}, $y$ is $\rho$-scaled shadowed by the orbit of $x$ up to  time $t$.  In particular, there exists a strictly increasing continuous function $\tau_{x,y}(t)$ that is differentiable (see Lemma~\ref{l.tau}) over $y$ with $\tau_{x,y}(0)=0$ such that 
	$$
	y_{\tau_{x,y}(s)} \in \cN_{\rho|X(x_s)|}(x_s), \forall s\in[0,t]. 
	$$
	To simplify notation, for the moment we shall drop the sub-indices and write $\tau(s)$. 
	Then for $i\in [0,t]\cap \NN$, the points $y_{\tau(i)}^j, j=E,F$ are well-defined and satisfy $y_{ \tau(i)} = [y_{ \tau(i)}^E,y_{ \tau(i)}^F]$. Furthermore, we have 
	$$
	\cP_{x_t}(y_t) = y_{ \tau(t)}, 
	$$
	and consequently, $|t-\tau(t)|<1$ (see~\eqref{e.Px2}).

		Now let $\phi$ be any H\"older continuous function with H\"older index $\gamma\in (0,1)$. Then we have\footnote{Here we slightly abuse notation and let $\Phi_0(y,t) = -\int^0_{t} g(y_s)\,ds$ if $t<0$. }
		\begin{align*}
			|\Phi_0(x,t) - \Phi_0(y,t)|&\le \sum_{i=0}^{t-1} \left| \Phi_0(x_i,1) - \Phi_0\left(y_{ \tau(i)}, \tau(i+1) -  \tau(i)\right)\right|\\&\quad + |\Phi_0(y_{ \tau(t)}, t-\tau(t))|\\
			\numberthis\label{e.step1}	&\le \|\phi\|_{C^0} +\sum_{i=0}^{t-1} \left| \Phi_0(x_i,1) - \Phi_0\left(y_{ \tau(i)}, \tau(i+1) -  \tau(i)\right)\right|.
		\end{align*}
		To control each summand on the right-hand side, we write
		\begin{align*}
			&\left|\Phi_0(x_i,1) - \Phi_0\left(y_{ \tau(i)}, \tau(i+1) -  \tau(i)\right)\right|\\ &\le  \left|\Phi_0(x_i,1) - \Phi_0(y_{\tau(i)}, 1)\right| + \left| \int_{1}^{\tau(i+1)-\tau(i)} \phi(y_s)\,ds\right|\\
			&\le  \left|\Phi_0(x_i,1) - \Phi_0(y_{\tau(i)}^F, 1)\right| +   \left|\Phi_0(y_{\tau(i)}^F,1) - \Phi_0(y_{\tau(i)}, 1)\right| \\
			&\quad + \|\phi\|_{C^0}\cdot  \left|\tau(i+1)-\tau(i)-1\right|\\
			&= I + II + III.
		\end{align*}

		To estimate I, we use the fact that $x_t$ is a $(\lambda_0,cu)$-hyperbolic time for the orbit segment $(x,t)$ under the scaled linear Poincar\'e flow. By Lemma~\ref{l.cuhyptime1}, for all $i\in [0,t]\cap \NN$:
		\begin{align*}
			d(x_i, y_{\tau(i)}^F) &\le d_{\cN(x_i)}(x_i, y_{\tau(i)}^F) \\
			&\le \frac{|X(x_i)|}{|X(x_t)|} c_2\lambda_1^{i-t} \cdot d_{\cN(x_t)}(x_t, y_{\tau(t)}^F)\\
			&\le \frac{|X(x_i)|}{|X(x_t)|} c_2\lambda_1^{i-t}\cdot\rho |X(x_t)|\\
			\numberthis\label{e.FF}	&=c_2\rho\lambda_1^{i-t}|X(x_i)|. 
		\end{align*} 
		Consequently, by Lemma~\ref{l.holder} ($c_1$ is the H\"older constant of $\Phi_0(x,1)$):
		\begin{align*}
			\left|\Phi_0(x_i,1) - \Phi_0(y_{\tau(i)}^F, 1)\right| &\le c_1  d(x_i,y_{\tau(i)}^F)^\gamma \le c_1c_2^\gamma \rho^\gamma \left(\lambda_1^\gamma\right)^{i-t} \sup_{z\in\bM}\{|X(z)|\}^\gamma\\
			\numberthis\label{e.xyF}	&= c_3 \left(\lambda_1^\gamma\right)^{i-t},
		\end{align*}
		where the constant $c_3$ is independent of $x$ and $t$.
		
		To estimate II, we use the uniform contraction of the scaled  linear Poincar\'e flow $\psi^*$ on the $E_N$ bundle (Lemma~\ref{l.basic}) and obtain, for  all $i\in[0,t]\cap\NN$:
		\begin{align*}
			d(y_{ \tau(i)}^F,y_{ \tau(i)}) \le d_{\cN(x_i)}(y_{ \tau(i)}^F,y_{ \tau(i)}) &\le \frac{|X(x_i)|}{|X(x)|} c_4\lambda_E^i \cdot d_{\cN(x)}( y_{\tau(0)}^E, y_{\tau(0)})\\
			&\le  \frac{|X(x_i)|}{|X(x)|} c_4\lambda_E^i \cdot\rho |X(x)|\\
			\numberthis\label{e.EE}	&=c_4\rho\lambda_E^i |X(x_i)|, 
		\end{align*} 
		where $c_4>0$, $\lambda_E\in (0,1)$ are some constants that do not depend on $y$ or $(x,t)$. As a result, we have 
		\begin{align*}
			\left|\Phi_0(y_{\tau(i)}^F,1) - \Phi_0(y_{\tau(i)}, 1)\right| &\le c_1 d(y_{\tau(i)}^F,y_{\tau(i)})^\gamma  \le c_1 c_4^\gamma \rho^\gamma \left(\lambda_E^\gamma \right)^i \sup_{z\in\bM}\{|X(z)|\}^\gamma \\
			\numberthis\label{e.yFy}	&= c_5  \left(\lambda_E^\gamma \right)^i,
		\end{align*}
		where the constant $c_5$ does not depend on $x$ or $t$.
		
		We are left with III. First, note that 
		$$
		\tau_{x,y}(i+1) - \tau_{x,y}(i)  = \tau_{x_i,y_{\tau(i)}}(1).  
		$$
		Then we apply Lemma~\ref{l.tau} to obtain
		\begin{align*}
			&\quad \left|\tau_{x,y}(i+1)-\tau_{x,y}(i)-1\right|\\ 
			&= \left| \tau_{x_i,y_{\tau(i)}}(1) - \tau_{x_i,x_i}(1)   \right|\\
			&\le  \sup\left\{\left\|D_z \tau_{x_i,z}(1)\right\| : z\in\cN_{\rho_0K_0^{-1}|X(x_i)|}(x_i) \right\}\cdot d_{\cN(x_i)}(y_{\tau(i)}, x_i)\\
			&\le \frac{1}{|X(x_i)|} K_\tau\cdot d_{\cN(x_i)}(y_{\tau(i)}, x_i)\\
			&\le \frac{1}{|X(x_i)|} K_\tau\left(d_{\cN(x_i)}(x_i,y_{\tau(i)}^F) + d_{\cN(x_i)}(y_{\tau(i)}^F,y_{\tau(i)}) \right)\\
			&\le  \frac{1}{|X(x_i)|} K_\tau\left(c_2\rho \lambda_1^{i-t} + c_4\rho \lambda_E^i\right)|X(x_i)|\\
			\numberthis\label{e.III}	&= c_6 \lambda_1^{i-t} + c_7 \lambda_E^i,
		\end{align*}
		where the fourth line follows from~\eqref{e.FF} and~\eqref{e.EE}. 
		
		Collecting~\eqref{e.xyF},~\eqref{e.yFy} and~\eqref{e.III}, we have 
		\begin{align*}
			& \left|\Phi_0(x_i,1) - \Phi_0\left(y_{ \tau(i)}, \tau(i+1) -  \tau(i)\right)\right|\\
			&\le  c_3 \left(\lambda_1^\gamma\right)^{i-t}+c_5  \left(\lambda_E^\gamma \right)^i+\|\phi\|_{C^0}\left( c_6 \lambda_1^{i-t} + c_7 \lambda_E^i\right),
		\end{align*} 
		where all the involved constants are independent of $x$ and $t$. 
		
		Summing over $i\in [0,t]\cap \NN$, we finally obtain
		\begin{align*}
			&\quad|\Phi_0(x,t) - \Phi_0(y,t)|\\
			&\le  \|\phi\|_{C^0} +\sum_{i=0}^{t-1} \left(c_3 \left(\lambda_1^\gamma\right)^{i-t}+c_5  \left(\lambda_E^\gamma \right)^i+\|\phi\|_{C^0}\left( c_6 \lambda_1^{i-t} + c_7 \lambda_E^i\right)\right)\\
			&\le c_8,
		\end{align*}
		for some constant $c_8>0$ that does not depend on $(x,t)\in\cG_B$ and $y\in B_{\vep_2,t}(x)$. We conclude that $\cG_B$ has the Bowen property at scale $\vep_2$ for every H\"older continuous function $\phi$.

	\end{proof}

\section{Specification}\label{s.spec}
In this section we prove Theorem~\ref{t.spec}, the specification on $\cG_{spec}$. The idea of the proof is to use the hyperbolic periodic orbit $\gamma$ provided by Assumption (B) as a ``bridge'' for the transition from one orbit segment to the next. The key observation of the proof is that, in order to achieve specification at scale $\delta$, a compact neighborhood $\CP_\sigma=\CP_\sigma(\delta)\subset W^u(\sigma)$ must be avoided by the starting point of each orbit segment (recall the definition of $\cG_{spec}$ from~\eqref{e.D1}, more specifically, item (2)). This means that the orbit collection $\cG_{spec}$ must depend on $\delta$; this is why the original CT-criterion does not apply to our case, and we need the improved version Theorem \ref{t.improvedCL}.

Let us also take this chance to explain why the original CT criterion~\cite{CT16} cannot be used in our setting. To begin with, note that Assumption (I) ``$\cG^M$ has specification at scale $\delta$ for every $M>0$'' can be replaced by Assumption (I') ``$\cG$ has specification at every $\delta>0$'' (\cite[Lemma 2.10]{CT16}). However, neither of these assumptions is applicable to singular flow. Let us explain.

\noindent {\em(I). $\cG^M$ has specification at scale $\delta$ for every $M>0$}. Recall that $\cG^M$ is obtained by adding a prefix and a suffix with length at most $M$ to orbit segments in $\cG$. For the Bowen property to hold, it is imperative that for orbit segment $(x,t)\in\cG$, both $x$ and $x_t$ are away from singularities. However, for $M$ sufficiently large, the endpoints of orbit segments in $\cG^M$ may approach singularities; this implies that the size of the local product structure, given by the fake foliations $\cF_{x,\cN}^{s/cu}$ may tend to zero. As a result, one should not expect $\cG^M$ to have specification at the same scale for every $M>0$.

\noindent {\em(I'). $\cG$ has specification at scale $\delta$ for every $\delta>0$}. For a singularity in a sectional-hyperbolic attractor, one has  $\left(W^{ss}(\sigma)\cap\Lambda\right)\setminus\{\sigma\}=\emptyset$ (see Proposition~\ref{p.sh}). For every $y\in W^u(\sigma)$, since $\cF^{s}(y)$ accumulates on $W^{ss}(\sigma) = \cF^{s}(\sigma)$ under backward iteration of the flow, there exists a continuous function $\delta(y)>0$ such that 
$
\left(\cF^{s}_{\delta(y)}(y)\cap \Lambda\right)\setminus \{y\}=\emptyset. 
$
In other words, $\delta(y)$ is an ``optimal'' scale that $y\in W^u(\sigma)$ can be shadowed by points in $\Lambda$. As a result, to prove that $\cG$ has specification at every scale, one must require that for every $(x,t)\in \cG$ one has $x\notin W^u(\sigma)$. However, $W^u(\sigma)$ is dense in $\Lambda$, and therefore the ``bad'' segments will not have a smaller pressure.

Therefore, for given $\delta>0$ we must remove a compact set $\CP_\sigma\subset W^u(\sigma)$ in which $\delta(y)>\delta$; then, after carefully removing a small neighborhood $U$ of $\CP_\sigma$ (this is needed for the existence of the maximal gap size $\tau$) one can prove that those good orbit segments $(x,t)$ with $x\notin U$ and $x_t$ a $(\lambda_0,cu)$-hyperbolic time have specification property.

Now we start the proof of Theorem~\ref{t.spec}. 

Recall that the collection  $\cG_{spec}(\lambda_0,r_0,W_\Sing, U_\Sing)$ consists of orbit segments $(x,t)$ such that 
\begin{itemize}
\item $x\notin W_\Sing\cup U_\Sing$, and $x_t\notin W_\Sing$ where $W_\Sing$ is an open neighborhood of $\Sing(X)\cap\Lambda$ that is contained in $B_{r_0}(\Sing(X)\cap\Lambda)$,\footnote{One can take $W_\Sing$ to be $W_r$ as defined in~\eqref{e.W.def}} and $U_\Sing$ is an open neighborhood of (relative) compact sets $\CP_\sigma\subset W^u(\sigma)$;  
\item $x_t$ is a $(\lambda_0,cu)$-hyperbolic time for the orbit segment $(x,t)$  under the scaled linear Poincar\'e flow.
\end{itemize}
Here the compact sets $\CP_\sigma$ are given by Lemma~\ref{l.(W^u)} and depend on $\delta$, and $\lambda_0$ is given by Lemma~\ref{l.hyptime}. 

Since $\Lambda$ is sectional-hyperbolic, the $E^{ss}$ bundle is uniformly contracted by the tangent flow $Df_t$; as a result, there exists a stable foliation $\cF^s$ whose leave at $x\in\Lambda$ will be denoted by $\cF^{s}(x)$. For a regular point $x$, this manifold can be pushed to $\cN_{\rho_0|X(x)|}(x)$ using the map $\cP_x$, and its image will be denoted by $W^{s,N}(x)$. Since we only consider orbit segments  $(x,t)$ with $x,x_t\notin W_\Sing$, $\cN_{\rho_0|X(x)|}(x)$ and $W^{s,N}(x)$ have uniform size at said points. The parameters $\xi$ and $\zeta$ below will be taken to be much smaller than this uniform size. Also note that vectors in $E^s_N$ are uniformly contracted by $\psi^*_t$ due to Lemma~\ref{l.basic}. In particular, this implies that every point in $W^{s,N}(x)$ is $\rho_0$-scaled shadowed by the orbit of $x$ up to all times $T>0$. 


To simplify notation, for fixed $W_\Sing$ we define 
$$
\epsilon = \inf_{z\notin W_\Sing}\rho_0 |X(z)|
$$
the uniform size for the normal planes outside $W_\Sing$. Throughout this section, unless otherwise specified, all $(\alpha,cu)$-disks are disks contained in an appropriate normal plane $\cN_{\epsilon}(x)$ at some $x\notin W_\Sing$, and are tangent to the $\alpha$-cone of the $F^{cu}_N$ bundle extended to every point on $\cN_\epsilon(x)$. 

We start with the following lemma on $(\lambda_0,cu)$-hyperbolic times, whose proof is an adapted version of the Hadamard-Perron theorem and can be found in~\cite{ABV00}, and is therefore omitted. 

\begin{lemma}\label{l.udisk}
For any open set $W_\Sing$ that satisfies  $\Sing(X)\cap\Lambda\subset  W_\Sing\subset  B_{r_0}(\Sing(X)\cap\Lambda)$, any $\lambda_2\in (1,\lambda_1)$ with $\lambda_1$ given by~\eqref{e.lambda1}, and any $\alpha>0$, $\xi>0$ small enough, there exists $\zeta_1>0$, such that for every $\zeta>0$,  there exists  $L>0$ such that the following property holds:

Let $D\subset \cN_\epsilon(x)$ be a $(\alpha,cu)$-disk and $x_t$ be a $(\lambda_0,cu)$-hyperbolic time for the orbit segment $(x,t)$ with $t>L$, such that $x\notin W_\Sing$ and $x_t\notin W_\Sing$; furthermore, assume that $D$ is centered at a point $z\in D\cap W^{s,N}_{3\xi}(x)$ with size at least $\zeta$. Then $D$ contains a smaller disk $D'$ such that every point in $D'$ is  $\rho_0$-scaled shadowed by the orbit of $x$ up to time $t$. Furthermore, $\cP_{t,x}(D')\subset \cN_{\epsilon}(x_t)$  is a $(\alpha,cu)$-disk with uniform size $\zeta_1>0$  
and has transverse intersection with $W^{s,N}_\xi(x_t)$. 

\end{lemma}

We also need the following lemma regarding the shadowing at the uniform scale $\delta$.  This does not immediately follow from the previous lemma due to the time change $\tau_{x,y}(t)$ involved in the scaled shadowing. In particular, we need to obtain an upper bound on $|\tau_{x,y}(t)-t|$.

\begin{lemma}\label{l.bowen1}
Let $\delta>0$. Under the assumptions of the previous lemma, one can decrease $\xi,\zeta$ such that for every $y\in D'$, it holds that  $y\in B_{t,\delta}(x)$.

\end{lemma}

\begin{proof}
Let $D'$ be the disk given by the previous lemma, with $\cP_{t,x}(D')$ having size $\zeta_1$. For $y\in D'$, the previous lemma shows that $y$ is $\rho_0$-scaled shadowed by the orbit of $x$ up to time $t$; consequently, there exists an increasing function  $\tau_{x,y}$ such that for $s\in[0,t]$,
$$
y_{\tau_{x,y}(s)} \in \cN_{\rho_0|X(x_s)|}(x_s). 
$$
Let $z$ be the transverse point of intersection between  $D$ and $W^{s,N}_{3\xi}(x)$. Then the previous argument also applies to $z$ (with a different function $\tau_{x,z}$). Write, for $i\in [0,t]\cap\NN$, 
$$
z^i = z_{\tau_{x,z}(i)} = \cP_{x_i}(z_i)\in \cN_{\rho_0|X(x_t)|}(x_i), \mbox{ and }y^i = y_{\tau_{x,y}(i)}; 
$$
we will use $z^i$ as a reference point when calculating $d_{\cN(x_i)}(y^i, x_i)$. For simplicity, we shall assume that $t\in\NN$.

First we need to estimate $|\tau_{x,y}(t) - t|.$ The proof is similar to the Proof of Proposition~\ref{p.Bowen} at the end of Section~\ref{s.bowen}, so we will only explain the structure of the proof and skip some details. 

We first write 
\begin{align*}
	|\tau_{x,y}(t) - t| &\le  \sum_{i=0}^{t-1}|\tau_{x,y}(i+1) - \tau_{x,y}(i)-1| 
	=\sum_{i=0}^{t-1} |\tau_{x_i,y^i}(1) - 1 |\\
	& = \sum_{i=0}^{t-1}|\tau_{x_i,y^i}(1) - \tau_{x_i,x_i}(1) |\\
	&\le\sum_{i=0}^{t-1}  {\frac{1}{|X(x_i)|}} K_\tau  d_{\cN(x_i)}(y^i,x_i)\\
	& \le \sum_{i=0}^{t-1}  {\frac{1}{|X(x_i)|}} K_\tau\left(d_{\cN(x_i)}(y^i,z^i)+d_{\cN(x_i)}(z^i,x_i)\right).
\end{align*}

Note that $y$ and $z$ are both contained in $D'$. Then, Lemma~\ref{l.cuhyptime1} gives the backward contraction: for $\lambda_1>1$ we have
$$
d_{\cN(x_i)}(y^i,z^i) \le \lambda_1^{-(t-i)} \zeta_1 |X(x_i)|; 
$$
on the other hand, $z\in W^{s,N}(x)$ which is exponential contracted by $\cP_{t,x}$; this leads to
$$
d_{\cN(x_i)}(z^i,x_i) \le 3\lambda_E^{i} \xi|X(x_i)|,
$$
where $\lambda_E\in (0,1)$. We thus obtain 
\begin{equation}\label{e.tau111}
	|\tau_{x,y}(t) - t| < c\max\{3\xi,\zeta_1\} 
\end{equation}
for some constant $c$ that does not depend on $\xi$ and $\zeta_1$. It is also clear from the proof that the same bound holds for every $s\in [0,t]$ (in which case the summation in the previous calculation only involves the first $s$ terms). 

Consequently, for every $s\in[0,t]\cap\NN$ we have
\begin{align*}
	d( f_s(y), f_s(x))&\le 	d(y_s, y^s) + d(y^s,x_s)\\
	& \le |\tau_{x,y}(s) - s|\cdot \max|X| + d_{\cN(x_i)}(y^s,z^s) + d_{\cN(x_s)}(z^s,x_s) \\
	&\le  c\max\{3\xi,\zeta_1\}\max|X| + \lambda_1^{-(t-s)} \zeta_1 |X(x_s)| + 3\lambda_E^{s} \xi|X(x_s)|\\
	&\le c'\max\{3\xi,\zeta_1\}.
\end{align*}
Take $Lip>0$ defined by~\eqref{e.lip}. Also recall that $\zeta_1$ is a constant multiple of $\zeta$.  Then, if $\xi$ and $\zeta$ are small enough so that  $c'\cdot Lip \max\{3\xi,\zeta_1\} < \delta$, we have $y\in B_{t,\delta}(x)$. The proof is complete.
\end{proof}

Next, we turn our attention to the transverse intersection between the invariant manifolds of regular points and the hyperbolic periodic point $p$.  
Recall that $p\in\Lambda$ is a hyperbolic periodic point with the following properties (Assumption~(B)):
\begin{itemize}
\item $\cF^{s}(\Orb(p)) = \cup_{y\in\Orb(p)} \cF^{s}(y)$ is dense in $\Lambda$;
\item for every $y\in\Lambda\setminus \Sing(X)$, we have $W^u(\Orb(p))\pitchfork \cF^{s}(y)\ne\emptyset$.
\end{itemize}

\begin{lemma}\label{l.(W^u)^c}
Given $\xi>0$ and $x\in (\cup_\sigma W^u(\sigma))^c{\cap\Lambda}$, there exists $K^u(x)>0$ such that $$W^u_{K^u(x)}(\Orb(p))\pitchfork \cF^{s}_\xi(x)\ne\emptyset.$$ Furthermore, it is possible to choose $K^u(x)$ locally continuously with respect to $x$.
\end{lemma}

\begin{proof}
Since $x\in (\cup_\sigma W^u(\sigma))^c\cap\Lambda$, it follows that $\alpha(x)$, the $\alpha$-limit set of $x$, is not a subset of $\Sing(X)$. We fix any $y\in \alpha(x)\setminus \Sing(X)$. 
By Assumption (B) there exist $K_1(y), K_2(y)>0$ such that 
\begin{equation}\label{e.(W^u)^c}
	W^u_{K_1(y)}(\Orb(p))\pitchfork \cF^{s}_{K_2(y)}(y)\ne\emptyset.
\end{equation}
Now we take a sequence of times $t_i\nearrow+\infty$ such that $f_{-t_i}(x)\to y$. Then by continuity of the transverse intersection, for $i$ sufficiently large one has 
\begin{equation}\label{e.(W^u)^c11}
	W^u_{2K_1(y)}(\Orb(p))\pitchfork \cF^{s}_{2K_2(y)}(f_{-t_i}(x))\ne\emptyset.
\end{equation}

Since vectors in the $E^{ss}$ bundle are  uniformly contracted by the tangent flow, there exists $T(y)>0$ such that whenever $t>T(y)$, we have for every $z\in\Reg(X)\cap\Lambda$,
$$
f_t(\cF^{s}_{2K_2(y)}(z))\subset \cF^{s}_\xi( f_t(z)).
$$
Let $t_n$ be the smallest term in the sequence $(t_i)^\infty_{i=1}$ with $t_n>T(y)$, and take $K^u(x)$ large enough such that $ f_{t_n}(W^u_{2K_1(y)}(\Orb(p)))\subset W^u_{K^u(x)}(\Orb(p))$. Then~\eqref{e.(W^u)^c11} implies that 
$$
W^u_{K^u(x)}(\Orb(p))\pitchfork \cF^{s}_\xi(x)\ne\emptyset,
$$
as desired. The continuity of $K(x)$ easily follows from the construction and the continuity of the invariant manifolds.
\end{proof}

The next lemma states that once a relative compact neighborhood $\CP_\sigma$ is removed from the unstable manifold of each singularity, one obtains the transverse intersection between $\cF^{s}_\xi(x)$ and $W^u(\Orb(p))$; note that according to Lemma~\ref{l.bowen1}, $\xi$ depends on $\delta$, and so is $\CP_\sigma$.

\begin{lemma}\label{l.(W^u)}
Given $\xi>0$, there exists a compact (under the relative topology) neighborhood $\CP_\sigma\subset W^u(\sigma)$ for each $\sigma\in\Sing(X)$, such that for every $x\in \cup_\sigma\left(W^u(\sigma)\setminus \CP_\sigma\right)$, there exists $K(x)>0$ such that 
$$
W^u_{K(x)}(\Orb(p))\pitchfork \cF^{s}_\xi(x)\ne\emptyset. 
$$
\end{lemma}

\begin{proof}
For each $\sigma\in \Sing(X)$, denote by $\CP_\sigma^0$ the closure of the connected component of $B_{r_0}(\sigma)\cap W^u(\sigma)$  that contains $\sigma$. Fix $T>0$ and denote by 
$$
D^T_\sigma = \bigcup_{t\in[0,T]} f_t(D^0_\sigma).
$$
which is a compact subset of $W^u(\sigma)$. In what follows, we will prove that there exists $T>0$ sufficiently large such that the corresponding sets $D^T_\sigma$ have the desired property.

According to Assumption (B), there exist $K',K''>0$ such that
\begin{equation}\label{e.(W^u)}
	W^u_{K'}(\Orb(p))\pitchfork \cF^{s}_{K''}(y)\ne\emptyset
\end{equation}
for $y\notin B_{r_0}(\Sing(X)\cap\Lambda)$. On the other hand, if $x\in   \cup_\sigma\left(W^u(\sigma)\setminus D^T_\sigma\right)$ then there exist $T(x)>T$ and a singularity $\sigma\in \Sing(X)$ such that 
$$
f_{-T(x)}(x)\in \partial (\CP_\sigma)
$$
where $\partial (\CP_\sigma)$ refers to the relative boundary of $\CP_\sigma$ inside $W^u(\sigma)$. As such, we have 
\begin{equation}\label{e.(W^u)1}
	W^u_{K'}(\Orb(p))\pitchfork \cF^{s}_{K''}( f_{-T(x)}(x))\ne\emptyset.
\end{equation}

By the uniform contraction of $D f_t$ along the $E^s$ bundle, there exists $T_0>0$ such that for every $y\in\Lambda$, it holds for every $t\ge T_0$ that 
$$
f_t(\cF^{s}_{K''}(y)) \subset \cF^{s}_\xi( f_t(y)). 
$$
In particular, this inclusion applies to $y= f_{-T(x)}(x)$. 
Now we let $K(x)>0$ sufficiently large enough so that 
$$
f_{T(x)}(W^u_{K'}(\Orb(p))) \subset W^u_{K(x)}(\Orb(p)). 
$$
Finally we fix $T=T_0$ and consider the compact sets $ D^{T_0}_\sigma$. For  a point $x\in \cup_\sigma\left(W^u(\sigma)\setminus D^{T_0}_\sigma\right)$, we iterate~\eqref{e.(W^u)1} by $ f_{T(x)}$ and obtain
$$
W^u_{K(x)}(\Orb(p))\pitchfork \cF^{s}_\xi(x)\ne\emptyset,
$$
as desired.
\end{proof}

Combining the previous two lemmas, we obtain the following result concerning the transverse intersection between 
$\cF^{s}(x)$ and $W^u(\Orb(p))$ at a uniform scale. This requires that a small neighborhood of $\cup_\sigma \CP_\sigma$ be removed.
\begin{lemma}\label{l.p_intersection}
	For any open set $W_\Sing$ that satisfies  $\Sing(X)\cap\Lambda\subset  W_\Sing\subset  B_{r_0}(\Sing(X)\cap\Lambda)$ and any $\xi>0$ there exists a compact set $\CP_\sigma\subset W^u(\sigma)$ such that for every open neighborhood $U_\Sing$ of $\cup_\sigma \CP_\sigma$, there exists $K_0>0$ such that 
	$$
	W^u_{K_0}(\Orb(p))\pitchfork \cF^{s}_\xi(x)\ne\emptyset
	$$  
	holds for every $x\notin W_\Sing\cup U_\Sing$. 
\end{lemma}

\begin{proof}
	First, note that by the continuity of the invariant manifolds and replace $K(x)$ by $2K(x)$ if necessary, the conclusion of Lemma~\ref{l.(W^u)}  holds for every $y$ (not necessarily in $\cup_\sigma W^u(\sigma)$) sufficiently close to $x$. Then Lemma~\ref{l.p_intersection} follows from a standard compactness argument applied to the complement of $W_\Sing\cup U_\Sing $. 
\end{proof}

We remark that, since $\Lambda$ is an attractor, we have $W^u(\Orb(p))\subset \Lambda$. Consequently, the point of intersection $q_x\in 	W^u_{K_0}(\Orb(p))\pitchfork \cF^{s}_\xi(x)$ is a point in $\Lambda$. Since we already assume that the bundles $E^{ss}$ and $F^{cu}$ are orthogonal at every point in $\Lambda$, the transverse intersection in the statement of the lemma is indeed orthogonal.

For the next lemma, we use the notation $\pitchfork^N$ to denote the transverse intersection inside $\cN_\epsilon(z)$ between two submanifolds of $\cN_\epsilon(z)$. 

Let $\zeta_1>0$ be given by Lemma~\ref{l.udisk}.
\begin{lemma}\label{l.D_intersection}
	For any open set $W_\Sing$ that satisfies  $\Sing(X)\cap\Lambda\subset  W_\Sing\subset  B_{r_0}(\Sing(X)\cap\Lambda)$  and every $\alpha>0$, $\xi>0$ small enough, there exist  compact sets $\CP_\sigma\subset W^u(\sigma), \sigma\in\Sing(X)\cap \Lambda$ such that for every open neighborhood $U_\Sing$ of $\cup_\sigma \CP_\sigma$,  there exist $\zeta>0$ and
	$\tau>0$ with the following property:
	
	For every $y\notin W_\Sing$ and $z\notin  W_\Sing\cup U_\Sing$, let $D$ be an $(\alpha,cu)$-disk centered at some $y'\in W^{s,N}_\xi(y)$ with size at least $\zeta_1$. Then there exists $s\in[0,\tau]$ such that $ \cP_z(f_s(D)\cap B_\epsilon(z))$ contains a $(\alpha,cu)$-disk  $D''$  centered at $z'\in \cP_z(f_s(D)\cap B_\epsilon(z))\pitchfork^N W^{s,N}_{2\xi}(z)$ 
	with size at least $\zeta$.
\end{lemma}

\begin{figure}[h!]
	\centering
	\def\svgwidth{\columnwidth}
	\includegraphics[scale=0.93]{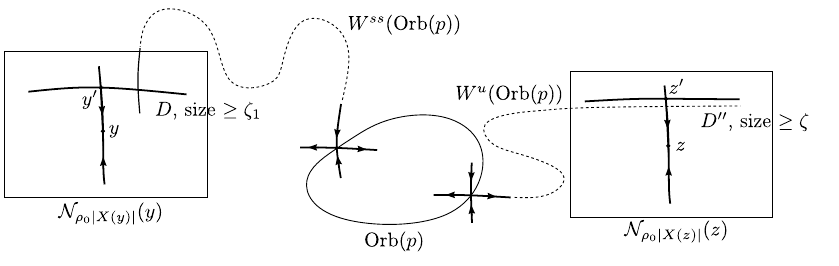}
	\caption{$\Orb(p)$ as a ``bridge''.}
	\label{f.lemma9.6}
\end{figure}

\begin{proof}
	We will use the hyperbolic periodic orbit $\Orb(p)$ as a ``bridge'' to build the transverse intersection. 
	
	 Recall that $\Lambda$ is a sectional-hyperbolic attractor, and the stable foliation $\cF^{s}$ is well-defined in an open neighborhood of $\Lambda$. 	
	Since $\cF^{s}(\Orb(p))$ is dense in a neighborhood of $\Lambda$, for every regular point $z$ in a small neighborhood of $\Lambda$, any disk $D$ on the normal plane of $z$ tangent to the $\alpha$-cone of $F^{cu}_N$ must have a nonempty transverse intersection with $\cF^{s}(\Orb(p))$ (which has dimension $\dim E^{ss}+1$)  at an angle within $(\frac\pi2-2\alpha,\frac\pi2+2\alpha)$. By the compactness of $W_\Sing^c$, we may take $K_3>0$ large enough such that for every $y\notin W_\Sing$, $\cF^{s}_{K_3}(\Orb(p))$ transversely intersect with any $(\alpha,cu)$-disk $D$ centered at some $y'\in W^{s,N}_\xi(y)$ with size at least $\zeta_1$.
	Furthermore, there is a sub-disk $D'$ contained in $D$ centered at $\tilde y'$ with size at least  $\zeta_1/2$. See Figure~\ref{f.lemma9.6}.
	
	By the previous lemma, we have  
	$$
	W^u_{K_0}(\Orb(p))\pitchfork \cF^{s}_\xi(z)\ne\emptyset.
	$$
	By the Inclination Lemma (see, for instance,~\cite[Proposition 6.2.23]{Katok}),  there exist $\tau>0$ depending on $\xi$ and $\alpha$ but is independent of $D$ and $y$, such that for some $s\in[0,\tau]$, $f_s(f_{[-\rho_0,\rho_0]}(D'))$ is  $\min\{\alpha, \xi\}$-approximated by $W^u_{2 K_0}(\Orb(p))\cap B_{\epsilon}(z)$ in $C^1$ topology, while $\cP_z\left(W^u_{2 K_0}(\Orb(p))\cap B_{\epsilon}(z)\right)$ has a transverse intersection with $W^{s,N}_{\xi}(z)$ inside $\cN_\epsilon (z)$. This shows that $\cP_z\left(f_s(D')\cap B_{\epsilon}(z)\right)$ has a transverse intersection with $W^{s,N}_{2\xi}(z)$ inside $\cN_\epsilon(z)$. Changing $s$  by no more than $\rho_0$ (recall the definition of $\cP_{z}$ from~\eqref{e.Px1}), we may assume that the point of intersection is $z' \in f_{s}(D').$
	In addition,  $\cP_{z}(f_s(D')\cap B_\epsilon(z)) \subset \cN_\epsilon (z)$ contains a disk $D''$ centered at $z'$ with size at least $\zeta$ and is tangent to the $(\alpha,cu)$-cone.

\end{proof}
\begin{remark}\label{r.timeshift}
	By~\eqref{e.Px1} and~\eqref{e.Px2} we see that for every $x\in D'$, there exists $\tilde x\in D$ such that  $x = f_{s(x)}(\tilde x)$ where $s(x)$ is a continuous function such that 
	$$
	|s-s(x)|<\rho_0.
	$$
	In particular, we have $s(x)\le\tau+\rho_0$ for all $x\in D'$.
\end{remark}

\begin{proof}[Proof of Theorem~\ref{t.spec}]
	In the proof of this theorem, to make our notation readable, we will stop using the notation $x_t = f_t(x)$ temporarily.

	Fix $r_0>0,\delta>0$ and an open set $W_\Sing$ with $\Sing(X)\cap\Lambda\subset W_\Sing\subset B_{r_0}(\Sing(X)\cap\Lambda)$. Then we take $\alpha>0$ small enough and  $\xi$ small enough such that Lemma~\ref{l.bowen1} holds. Then we apply Lemma~\ref{l.udisk} to obtain $\zeta_1$, and apply
	Lemma~\ref{l.D_intersection} on any small neighborhood $U_\Sing$ of $\cup_\sigma\CP_\sigma$ to obtain $\tau(U_\Sing)$ and $\zeta$. Decreasing $\zeta$ if necessary, we assume that $\zeta$ also satisfies Lemma~\ref{l.bowen1}. With now fixed $\zeta_1$ and $\zeta$, we use Lemma~\ref{l.udisk} to obtain $L$, the lower bound for the length of orbit segments that we will consider below.  

	Take orbit segments $(x^1,t_1)$, $\ldots, (x^k,t_k)\in \cgs^{>L}(\delta,r_0,W_\Sing,U_\Sing)$, we will construct a sequence of points $w^1,\ldots, w^k$ recursively such that $w^k$ has the desired shadowing property. 
	
	\noindent {\em Construction}:
	We start with $w^1=x^1=y^1$ and take $D^1$ a $(\alpha,cu)$-disk at $w^1$ with size at least $\zeta$.  Furthermore, we may assume that $D^1$ belongs to the unstable manifold of $\Orb(p)$; this is possible due to Assumption (B1) of Theorem \ref{m.B}. Since $\Lambda$ is an attractor and therefore contains the unstable manifold of every hyperbolic periodic orbit in $\Lambda$, this means that $D^1\subset \Lambda$. By Lemma~\ref{l.udisk}, $D^1$ contains a smaller disk $(D^1)'$ such that $(D^1)'': = \cP_{t_1,x}((D^1)')$ is a $(\alpha,cu)$-disk with size $\zeta_1>0$ and $(D^1)''\pitchfork^N W^{x,N}_\xi((x^1)_{t_1})\ne\emptyset$. We then 
	apply Lemma~\ref{l.D_intersection} to obtain a time $s_1\in[0,\tau]$ for which $\cP_{x^2}( f_{s_1}((D^1)''))$ contains a $(\alpha,cu)$-disk $D^2$ centered at a point $y^2\in \cP_{x^2}( f_{s_1}((D^1)'')\cap B_\epsilon(x^2))\pitchfork^N W^{s,N}_{2\xi}(x^2)$ with size  $\zeta$. 
	Then define
	$$
	w^2 =  \cP_{x_1}(f_{-s_1-t_1}(y^2)) = f_{-(t_1 + s_1^1)}(y^2)\in D^1
	$$
	for some $s_1^1 < s_1 + 2\rho_0 < \tau+2\rho_0$ (recall the definition of $\cP_{x}$ by~\eqref{e.Px1}). 
	Note that $	w^2\in B_{t_1, \delta}(x^1)$ by Lemma~\ref{l.bowen1}. Also since $  f_{s_1^1+t_1}(w^2)=y^2\in W^{s,N}_{2\xi}(x^2)$, we get $f_{s_1^1+t_1}(w^2)\in B_{\delta,t_2}(x^2)$. Indeed, for all $\overline y^2\in D^2$, by Remark~\ref{r.timeshift} we find $s(\overline y^2)\le \tau+\rho_0$ such that $f_{-s(\overline y^2)}(y^2)\in (D^1)''$; as a result, Lemma~\ref{l.bowen1} shows that $f_{-t_1-s(\overline y^2)}(\overline y^2)\in B_{t_1,\delta}(x^1)$. This observation will be useful when we later find the transition time for the next point $w^3$.

	Having constructed $w^1,\ldots, w^i$, a sequence of disks $D^j$ for $1\le j\le i$ which are $(\alpha,cu)$-disks with size $\zeta$ centered at a point of transverse intersection $y^j\in D^j\pitchfork^N W^{s,N}_{2\xi}(x^j)$, as well as a sequence of times\footnote{The superscript in $s_j^i$ is because, since all the points $\omega^i$ are defined using the sectional Poincar\'e map, each time we define the next $w^{i+1}$ we must make a small modification in the existing transition times $s_j$. However, as we shall see later, the modification made is exponentially small; therefore, all the transition times are uniformly bounded.} $s_1^{i-1},\ldots, s_{i-1}^{i-1}\in[0,\tau]$, we now construct $w^{i+1}, D^{i+1}$, $y^{i+1}$ and the new transition times $\{s_j^i,j=1,\ldots,i\}$. Let $(D^i)'$ be the $(\alpha, x^i_{t_i},cu)$-disk contained in $D^i$ such that $(D^i)'': = \cP_{t_i,x^i}((D^i)')$ is a $(\alpha,cu)$ disk with size $\zeta_1$ given by Lemma~\ref{l.udisk}, we then 
	apply Lemma~\ref{l.D_intersection} to obtain a time $s_i\in[0,\tau]$ such that 
	$$
	\cP_{x^{i+1}}\circ f_{s_{i}}((D^i)'')\cap B_\epsilon (x^{i+1})\pitchfork^N W^{s,N}_{2\xi}(x^{i+1}).
	$$
	Furthermore, $\cP_{x^{i+1}}\circ f_{s_{i}}((D^i)'')\cap B_\epsilon (x^{i+1})$ contains a $(\alpha,cu)$-disk centered at the transverse intersection point $y^{i+1}$ with size $\zeta$. Denote this disk by $D^{i+1}$. By changing $s_{i}$ by no more than $\rho_0$ we may assume that $y^{i+1} \in f_{s_i}((D^i)'')$. 
	
	Note that our construction above implies that for every $j\le i$, the disk $D^{j+1}$ is contained in the image of a sub-disk $(D^j)'\subset D^j$, under the map 
	\begin{equation}\label{e.F}
		F_j:=\cP_{x^{j+1}}\circ f_{s_j} \circ \cP_{t_j,x^{j}}  = \cP_{x^{j+1}}\circ f_{s_j} \circ \cP_{f_{t_j}(x^{j})}\circ f_{t_j} 
	\end{equation}	
	where the second equality follows from the alternate definition of $\cP_{t,x}$ by \eqref{e.Ptx1}; see also~\eqref{e.tau111}.
	Then, we consider the point 
	$$
	w^{i+1} : = F_1^{-1}\circ\cdots \circ F_i^{-1} (y^{i+1})\in D^1.
	$$
	Repeating this process, we find a point $w^k\in D^1$ which is the pre-image of a point $y^k\in D^k\pitchfork^N W^{s,N}_{2\xi} (x^k)$. Below we will prove that $w^k$ satisfies the desired shadowing property, with transition times $s_i(y)\le \tau +2\rho_0$ ($s_i(y)$'s may differ from  $s_i$'s constructed earlier).
	
	First, note that $y^k\in B_{t_k,\delta}(x^k)$ since it is contained in $W^{x,N}_{2\xi}(x^k)$. By \eqref{e.F}, \eqref{e.Px1}, \eqref{e.Px2} and Remark~\ref{r.timeshift}, there exists $s_{k-1}(y^k)\le\tau+2\rho_0$, such that 
	$$
	z^{k-1} :=f_{-s_{k-1}(y^k) - t_{k-1}} (y^k) =  F_k^{-1}(y^k)\in (D^{k-1})'\subset D^{k-1}.
	$$
	Furthermore, by Lemma~\ref{l.bowen1} we have $z^{k-1}\in B_{t_{k-1}, \delta}(x^{k-1})$. Next, we find $s_{k-1}(y^k) = s_{k-1}(z^{k-1}) \le \tau + 2\rho_0$ such that 
	$$
	z^{k-2} :=f_{-s_{k-2}(y^k) - t_{k-2}} (z^{k-1}) =  F_{k-1}^{-1}(z^{k-1})\in (D^{k-2})'\subset D^{k-2}.
	$$
	Again, Lemma~\ref{l.bowen1} shows that  $z^{k-2}\in B_{t_{k-2}, \delta}(x^{k-2})$. Recursively, we obtain a sequence of times $s_{j}(y^k) = s_{j}(z^{j+1}) \le\tau+2\rho_0$, $j=1,\ldots,k-1$ such that
	$$
	z^{j} = f_{-s_{j}(y^k) - t_{j}} (z^{j+1}) =  F_{j+1}^{-1}(z^{j+1})\in (D^{j})'\subset D^j.
	$$
	Furthermore, Lemma~\ref{l.bowen1} shows that  $z^{j}\in B_{t_{j}, \delta}(x^{j})$ for all $j=1,\ldots, k-1$. The construction above implies that $z^1 = w^k$. Therefore we have, for all $j=1,\ldots, k$,
	$$
	f_{\sum_{i=1}^{j-1} (t_i + s_i(y^k))}(w^k) = z^j\in B_{t_j,\delta}(x^j).
	$$
	Finally, we note that $w^k\in D^1\subset W^u(\Orb(p))\subset\Lambda.$	This concludes the proof of Theorem~\ref{t.spec}.
	
\end{proof}

\appendix

\section{Proof of Lemma~\ref{l.var.princ}}\label{ap.var}

\begin{proof}
	
	Recall that $h_\mu(X) = h_\mu(f_1)$. We consider $\Phi_0(x,1) = \int_0^1\phi(x_t)\,dt$ and note that for $n\in\NN$, 
	$$
	\nu_n=\frac{\sum_{x\in E_t}S_n^{f_1}\Phi_0(x,1)\cdot \delta_x}{\sum_{x\in E_t}S_n^{f_1}\Phi_0(x,1)}, 
	$$
	where $S_n^{f_1}\Phi_0(x,1)$ denotes the Birkhoff sum of the function $\Phi_0(x,1)$ under the time-one map $f_1$. Let 
	$$
	\mu_n^1 = \frac1n \sum_{k=0}^{n-1}(f_1)^k_* \nu_n.
	$$
	Then $\mu_n = \int_0^1(f_s)_*\mu_n^1\,ds$. Let $(n_k)$ be the sequence along which we have $\mu_{n_k}\to \mu$. Taking subsequence if necessary, we let 
	$$
	\mu^1 = \lim_{k\to\infty} \mu^1_{n_k}. 
	$$
	Note that $P(\cC, \phi,\delta,X) = P(\cC, \Phi_0(x,1),\delta, f_1)$ if we define the latter using  Bowen's $d^1$ metric. Below we will prove that 
	$$
	P(\cC, \Phi_0(x,1),\delta, f_1)\le h_{\mu^1}(f_1)+\int\Phi_0(x,1)\,d\mu^1,
	$$
	from which the statement of the lemma follows.
	
	For simplicity's sake we put $f = f_1$ and drop the superscription $1$ in $\mu^1$ and $\Phi_0(x,1)$. The proof of this inequality follows the proof of the variational principle by Walters~\cite{Wal}. One needs to carefully check that the proof holds when the pressure is only defined on the orbit collection $\cC$. However this is easy since the computation of $h_\mu(f)$ does not involve $\cC$.
	
	Following the second part of~\cite[Theorem 9.10]{Wal}, we choose a finite measurable partition $\cA = \{A_1,\ldots, A_k\}$ of $M$ with $\diam \cA<\delta$ and $\mu(\partial \cA) = 0.$ Then every element of $\bigvee_{j=0}^{n-1}f^{-j}\cA$ contains at most one element of $E_n$. It follows that 
	\begin{align*}
		&H_{\nu_n}\left(\bigvee_{j=0}^{n-1}f^{-j}\cA\right) + \int S_n\phi\,d\nu_n\\
		=&\sum_{x\in E_n}\nu_n(\{x\})\left(S_n\phi(x) - \log\nu_n(\{x\})\right)\\
		=&\log \sum_{x\in E_n}\exp(S_n\phi(x))\\
		\ge&  \log \Lambda(\cC, \phi, \delta,n,f)-1, 
	\end{align*}
	where the second equality follows from~\cite[Lemma 9.9]{Wal}, and the last line is due to our choice of $E_n$.
	
	For natural numbers $1\le q<n$ we define, for $0\le j\le q-1$, $a(j) =[(n-j)/q]$. Fix such $j$, we have
	$$
	\bigvee_{j=0}^{n-1}f^{-j}\cA = \bigvee_{r=0}^{a(j)-1}f^{-(rq+j)}\bigvee_{i=0}^{q-1}f^{-i}\cA\vee\bigvee_{l\in S}f^{-l}\cA
	$$
	and $S$ is a set with cardinality at most $2q$. 
	
	Therefore,
	\begin{align*}
		&\log \Lambda(\cC,\phi, \delta,n,f)-1\\
		\le& H_{\nu_n}\left(\bigvee_{j=0}^{n-1}f^{-j}\cA\right) + \int S_n\phi\,d\nu_n\\
		\le& \sum_{r=0}^{a(j)-1}H_{\nu_n}\left(f^{-(rq+j)}\bigvee_{i=0}^{q-1}f^{-i}\cA\right) + H_{\nu_n}\left(\bigvee_{l\in S}f^{-l}\cA \right)+ \int S_n\phi\,d\nu_n\\
		\le& \sum_{r=0}^{a(j)-1}H_{\nu_n\circ f^{-(rq+j)}}\left(\bigvee_{i=0}^{q-1}f^{-i}\cA\right)  + 2q\log k +  \int S_n\phi\,d\nu_n.
	\end{align*}
	Summing over $j$ from $0$ to $q-1$ and dividing by $n$ gives
	\begin{align*}
		&\frac qn\log \Lambda(\cC, \phi, \delta,n,f)-\frac qn\\\le & \frac1n\sum_{p=0}^{n-1}H_{\nu_n\circ f^{-p}}\left(\bigvee_{i=0}^{q-1}f^{-i}\cA\right)  + \frac{2q^2}{n}\log k +  q \int \phi\,d\mu_n\\
		\le &H_{\mu_n}\left(\bigvee_{i=0}^{q-1}f^{-i}\cA\right)  + \frac{2q^2}{n}\log k +  q \int \phi\,d\mu_n
	\end{align*}
	where the second inequality follows from the convexity of $H_*$ as a function of the measure (\cite[Section 8.2, Remark 2]{Wal}).
	Sending $n$ to infinity along the sequence $n_j$, we finally obtain
	$$
	qP(\cC,\phi,\delta,f)\le H_{\mu}\left(\bigvee_{i=0}^{q-1}f^{-i}\cA\right)  +  q \int \phi\,d\mu.
	$$
	Dividing by $q$ and sending it to infinity, we conclude that 
	$$
	P(\cC,\phi,\delta,f)\le h_{\mu}\left(f,\cA\right)  +  \int \phi\,d\mu \le h_{\mu}\left(f\right)  +  \int \phi\,d\mu.
	$$
	This finishes the proof of the lemma.
\end{proof}
 
\section{Proof of Lemma~\ref{l.comparecoord}}\label{s.A1}

\begin{proof}[Proof of Lemma~\ref{l.comparecoord}]
	Denote by $\pi^E: \RR^n\to E$ and $\pi^F: \RR^n\to F$ the projections to the first $k$ coordinates and to the last $n-k$ coordinates, respectively. We have $y^1 = \pi^E(y)$ and $y^2 = \pi^F(y)$. 
	
	We start the proof with the following general observation:  for all $\alpha>0$ small enough, if $v$ is a vector contained in the $(\alpha, E)$-cone, then it holds that 
	\begin{equation}\label{e.cone0}
		|\pi^E(v)|\le |v|\le (1+{O}(\alpha))\cdot|\pi^E(v)|.
	\end{equation}

	The same holds true for vectors contained in the $(\alpha,F)$-cone, in which case the projection is replaced by $\pi^F$.
	
	Let $y$ and $y^F$ be points in $\RR^n$ such that $y^F$ (as a vector) is in the $(\alpha, F)$-cone, $y-y^F$ is in the $(\alpha, E)$-cone and $|y|, |y^F|\le 1$. To simplify notation we will let $v^E=y-y^F$. See Figure~\ref{f.cone}. We observe that $y^2=\pi^F(y)=\pi^F(v^E)+\pi^F(y^F)$. 
	The key step in the proof is to obtain upper bounds for $|v^E|$ and $|\pi^F(v^E)|$.
	
	By projecting $y,y^F,v^E$ to $E$ and noting that $v^E=-y^F+y^1+y^2$, we get:
	\begin{align*}
		|\pi^E(v^E)|\le &|\pi^E(y^F)|+|\pi^E(y^1)|+|\pi^E(y^2)|\\
		\numberthis\label{e.vu}	\le &|y^1|+\mathcal{O}(\alpha)\cdot|y^F|
	\end{align*}
	where we used 
	the fact that $|\pi^E(y^F)|=\mathcal{O}(\alpha)|y^F|$ since $y^F$ is contained in the $(\alpha,E)$-cone, and $\pi^E(y^2) = \vec 0$. 
	Combining~\eqref{e.vu} and~\eqref{e.cone0} we obtain an upper bound for $|v^E|$:
	\begin{equation}\label{e.vuub}
		|v^E|\le(1+\mathcal{O}(\alpha))|\pi^E(v^E)|\le |y^1|+\mathcal{O}(\alpha)(|y^1|+|y^F|);
	\end{equation}	
	here we omit the higher order term $\cO(\alpha^2)|y^F|$ since $|y^F|\le 1$. 
	To obtain an upper bound for $|y^F|$ in terms of $|y^2|$, we project all vectors to $F$ (and note that $y^F=y^1+y^2-v^E$, and )
	\begin{align*}
		|y^F|\le& (1+\mathcal{O}(\alpha))|\pi^F(y^F)|\\
		\le &(1+\mathcal{O}(\alpha))\big( |y^2|+|\pi^F(v^E)|\big) && \mbox{since } \pi^F(y^1) = 0\\ 
		\numberthis\label{e.ysub0}	\le & |y^2|+ \mathcal{O}(\alpha)\big(|y^2|+|v^E|\big) && \mbox{since } \pi^F(v^E) = \mathcal{O}(\alpha)|v^E|\hspace{-1cm}\\
		\numberthis\label{e.ysub}	\le & |y^2|+\mathcal{O}(\alpha)(|y^1|+|y^2|),
	\end{align*}
	where the last inequality follows from~\eqref{e.vuub}. Substituting~\eqref{e.ysub} back in~\eqref{e.vuub} yields
	\begin{equation}\label{e.vuub1}
		|v^E|\le |y^1|+\mathcal{O}(\alpha)(|y^1|+|y^2|),
	\end{equation}
	as desired.
	
	\begin{figure}
		\centering
		\def\svgwidth{\columnwidth}
		\includegraphics[scale=1]{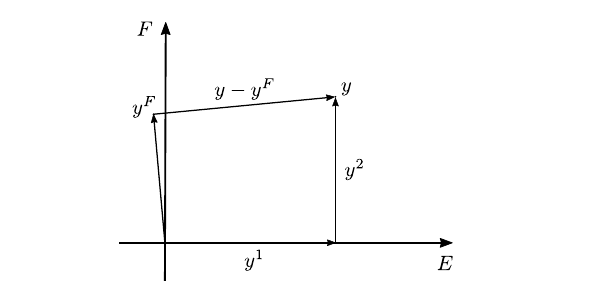}
		\caption{The vectors $y^F, v^E, y^1$ and $y^2$.}
		\label{f.cone}
	\end{figure}
	To obtain lower estimates on $y^F$ and $v^E$, we write
	
		\begin{align*}
		|y^F|\ge &|\pi^F(y^F)|\ge & |y^2|-|\pi^F(v^E)| & 
		\numberthis\label{e.yslb}\ge&|y^2|-\mathcal{O}(\alpha)|v^E| \ge |y^2|-\mathcal{O}(\alpha)|y^1|
	\end{align*}
	where we use~\eqref{e.vuub1} to obtain the last inequality. Similarly,

	\begin{align*}
		|v^E|\ge & |\pi^E(v^E)| \ge & |y^1|-|\pi^E(y^F)| \ge & |y^1|-\mathcal{O}(\alpha)|y^F|
		\numberthis\label{e.vulb}	\ge & |y^1|-\mathcal{O}(\alpha)|y^2|
	\end{align*}

	where the last inequality follows from~\eqref{e.ysub}. We conclude the proof for the first part of the proposition.
	
	Next we will prove case (1) under the extra assumption that $|y^F|\ge C_0|v^E|$. The proof of case (2) is analogous and therefore omitted.
	
	Under this extra assumption, we rewrite~\eqref{e.ysub0} as (assume that $\alpha\ll C_0$)
	$$
	|y^F|\le |y^2|+\mathcal{O}(\alpha)(|y^2|+C_0^{-1}|y^F|),
	$$
	where the constant in $\mathcal{O}(\cdot)$ does not depend on $C_0$. Rearranging terms, we obtain, for $\alpha$ sufficiently small,
	\begin{align*}
		|y^F|&\le \frac{1+\mathcal{O}(\alpha)}{1-C_0^{-1}\mathcal{O}(\alpha)}|y^2|\\
		&= (1+(C_0^{-1}+1)\mathcal{O}(\alpha))|y^2|,
	\end{align*}
	where the second inequality is due to the Taylor expansion.
	
	For the lower estimate, we use~\eqref{e.yslb} and the assumption that $|y^F|\ge C_0|v^E|$ to obtain
	\begin{align*}
		|y^F|&\ge|y^2|-\mathcal{O}(\alpha) |v^E|\\
		&\ge |y^2|- C_0^{-1}\mathcal{O}(\alpha)|y^F|,
	\end{align*}
	which, after rearranging terms and using the Taylor expansion, becomes
	$$
	|y^F|\ge(1-C_0^{-1}\mathcal{O}(\alpha))|y^2|,
	$$
	as required.
	
	We are left with the proof of $|y^2|\ge C_0(1-C_0'\alpha)|y^1|$. We write
	\begin{align*}
		|y^2|\ge&\frac{1}{1+(C_0^{-1}+1)\mathcal{O}(\alpha)}|y^F| \\
		= & (1-(C_0^{-1}+1)\mathcal{O}(\alpha))|y^F|\\
		\ge & C_0(1-(C_0^{-1}+1)\mathcal{O}(\alpha)) |v^E| \hspace{3.92cm}\mbox{ by assumption}\\
		\ge &C_0(1-(C_0^{-1}+1)\mathcal{O}(\alpha)) (|y^1|-\mathcal{O}(\alpha)|y^2|)\hspace{1.8cm}\mbox{ by \eqref{e.vulb}}.
	\end{align*}
	Rearranging terms, we get 
	$$
	|y^2|\ge C_0\cdot\left(\frac{1}{1+C_0\mathcal{O}(\alpha)}(1-(C_0^{-1}+1)\mathcal{O}(\alpha))\right)|y^1|.
	$$
	It is clear that the term in the parentheses is bounded from below by $(1- C_0'\alpha)$ for some constant $C_0'$ that depends on $C_0.$ With this we conclude the proof of the proposition. 
	
\end{proof}

\section{Proof of Lemma~\ref{l.Elarge} and~\ref{l.Flarge}}\label{s.A2}

\begin{proof}[Proof of Lemma~\ref{l.Elarge}]
	The key idea of the proof is to use Lemma~\ref{l.changecoord} to transition from the coordinate system $(\cF_{x^+,\cN}^s, \cF_{x^+,\cN}^{cu})$ to a proper coordinate system defined using $\cF_\sigma^j,j=s,c,u,cs,cu$ near the point $y^+$, use the dominated splitting $E^{ss}_\sigma\oplus E^c_\sigma\oplus E^u_\sigma$ to get the desired expansion/contraction,  and then transition back to $(\cF_{x,\cN}^s, \cF_{x,\cN}^{cu})$ near the point $x$.

	Without loss of generality, we will assume that $T:=t(x)$ is a positive integer. 
	
	Fix $\alpha>0$ such that all the lemmas in Section~\ref{ss.cone} hold with  $1000\alpha\cdot \max\{L_0, L_1,L_2\}<0.01$. Then we get $r_0$ and $\overline r \in (0,r)$ from Lemma~\ref{l.nearEc} (we also assume that the foliations $\cF^*_\sigma$ are well-defined in $B_{r_0}(\sigma)$). For simplicity's sake we will treat $B_{r_0}(\sigma)$ as $\RR^n$ ($n=\dim \bM$) with $\sigma$ sitting at the origin. The subspaces $E_\sigma^i, i=s,c,u$, which are assumed to be pairwise orthogonal, will be identified with the subspaces $\RR^s, \RR^c$ and $\RR^u$, respectively. Here we slightly abuse notation and write $s,c,u$ for the dimension of their corresponding subspace; consequently, we have $c=1$ and $s+c+u=n=\dim\bM$.

	\noindent {\em Preparation: a coordinate system near $\sigma$}. 
	
	Recall that since $\sigma$ is Lorenz-like, orbits in $\Lambda$ can only approach $\sigma$ along $E^c_\sigma$ direction, thanks to Lemma~\ref{l.lgw} and~\ref{l.nearEc}. We may assume that for $x\in \partial W_{r }^-$ the flow direction $X(x)$ is almost parallel to the one-dimensional subspace $\RR^c$ (henceforth, ``almost parallel'' means that the vectors are contained in the $\alpha$-cone of the corresponding bundle). This means that the normal plane $N(x)$ is almost parallel to the co-dimension one subspace $\RR^s\oplus \RR^u$. Furthermore, the subspaces $E_N^s(x)$ and $F_N^{cu}(s)$ (orthogonal to each other by assumption) are almost parallel to $\RR^s$ and $\RR^{u}$, respectively. On the other hand, at the point $x^+$ the flow direction is almost parallel to $\RR^u$. 
	
	We fix $r\le \overline r$; towards the end of the proof we will further decrease $\overline r,$ if necessary. In the compact set $W^c_r$ there is no singularity, and consequently, the flow speed is bounded away from zero. We take $\vep_0$ such that 
	$$
	0<\vep_0< \inf_{y\in W_r^c}\rho_0|X(y)|;
	$$
	this choice of $\vep_0$ implies that the map $\cP_x$ given by~\eqref{e.Px1} is well-defined in $B_{\vep_0}(x)$ for every $x\in \partial W^\pm_{r^1}\subset W_r^c$. To simplify notation, in the proof of this lemma we shall write $\cN(x)=\cN_{\vep}(x)\subset \cN_{\rho_0|X(x)|}(x)$,  and $\cN(x^+)=\cN_{\vep}(x^+)\subset \cN_{\rho_0|X(x^+)|}(x^+)$.
	Below we will always assume that $\vep\le \vep_0$.
	
	\begin{figure}
		\centering
		\def\svgwidth{\columnwidth}
		\includegraphics[scale=0.4]{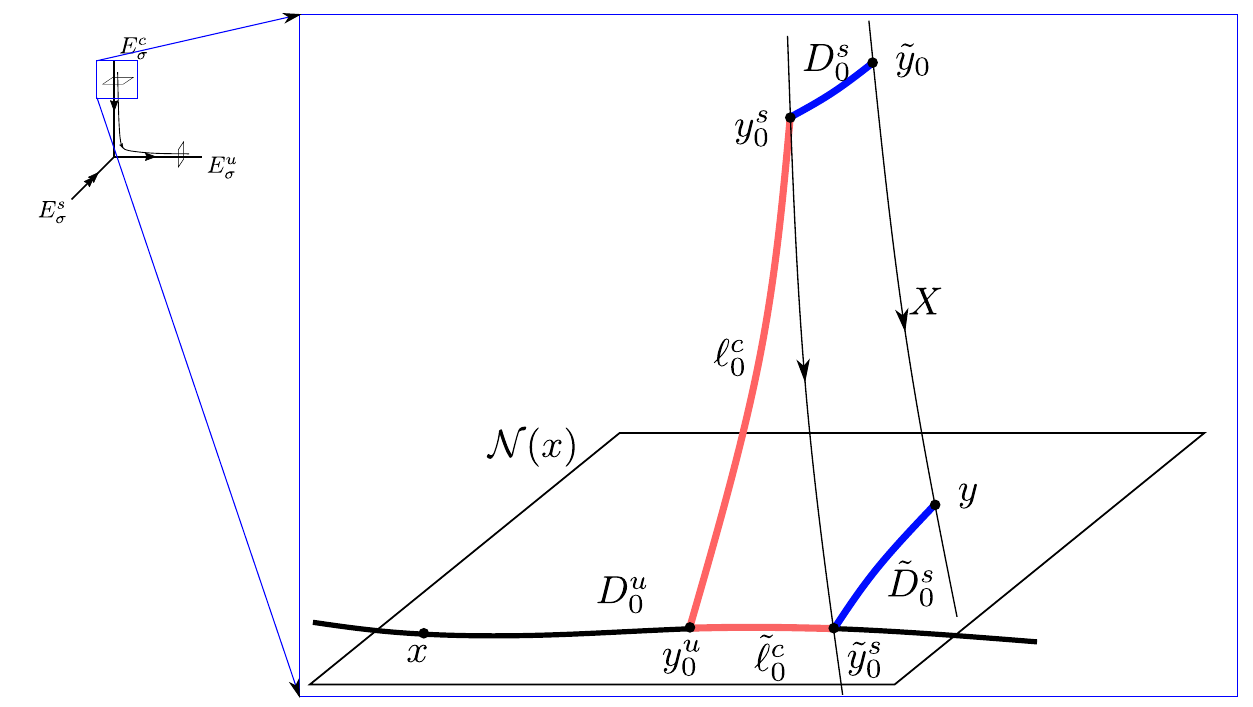}
		\caption{The triple $(D_0^s,\ell_0^c,D_0^u)$.}
		\label{f.nearw-}
	\end{figure}
	Following the discussion above, we may take a disk $D_0^u\subset \cN(x)$ with $\dim D_0^u=u$, such that $D_0^u$ is tangent to the $(\alpha, \RR^u)$-cone (this means that $D_0^u$ is the graph of a $C^1$ function $g:\RR^u\to \RR^s\oplus \RR^c$ with $\|Dg\|_{C^1}<\alpha$). Writing $D_0=f_{[-1,1]}(D_0^u)$, we see that $D_0$ is tangent to the $(\alpha,\RR^c\oplus \RR^u)$-cone. By the dominated splitting $E_\sigma^s\oplus_< (E_\sigma^c\oplus E_\sigma^u) $, we see that this cone field is forward invariant under the time-one map $f_1$. In particular, for every $0\le i\le T$, $D_i:=f_i(D_0)$ is a $(u+1)$-dimensional disk that is also tangent to the $(\alpha, \RR^c\oplus \RR^u)$-cone. 
	
	
	Now we consider the following points:
	\begin{align*}
		&\tilde y_0=f_{-T}(y^+),\\ 
		\{&y_0^s\} = D_0\pitchfork \cF_\sigma^s(\tilde y_0), \mbox{ and }\\
		\{&y_0^u\}=\cF_\sigma^{cs}(\tilde y_0)\pitchfork D_0^u.
	\end{align*}
	We also take:
	\begin{enumerate}
		\item an $s$-dimensional embedded disk $D^s_0\subset \cF_\sigma^s(\tilde y_0)$ center at $\tilde y_0$ and contains $y_0^s;$
		\item a one-dimensional arc $\ell^c_0$ in $\cF_\sigma^{cs}(\tilde y_0)\cap D_0$ that connects $y_0^s$ with $y_0^u$.\footnote{Recall that $\cF_\sigma^s$ sub-foliate $\cF_\sigma^{cs}$; therefore $y_0^s\in\cF_\sigma^{cs}(\tilde y_0)\cap D_0$.}
	\end{enumerate}
	For the second item, we remark that 
	$$
	\dim \cF_\sigma^{cs}(y_0^s)+\dim  D_0= s+1+u+1=\dim\bM+1,
	$$
	and therefore the intersection of $\cF_\sigma^{cs}(y_0^s)$ with $D_0$ contains a one-dimensional submanifold which is tangent to the $(\alpha,\RR^c)$-cone. In summary, the triple 
	$$
	(D_0^s,\ell_0^c,D_0^u)
	$$
	connects $x$ and $\tilde y_0$ via $y_0^u$ and $y_0^s$. See Figure~\ref{f.nearw-}. 
	
	For each $i=0,1,\ldots, T$ we define the points
	$$
	\tilde y_i = f_i(\tilde y_0),\footnote{This means that $\tilde y_T=y^+$} \,\mbox{ and }\, y_i^j = f_i( y_0^j), j=s,u
	$$
	together with the triple $(D_i^s,\ell_i^c,D_i^u)$ where
	$$
	\ell_i^c =f_i(\ell_0^c), \mbox{ and } D_i^j = f_i(D_0^j), j=s,u.
	$$
	The construction above implies the following facts:
	\begin{enumerate}
		\item by the forward invariance of the $(\alpha,\RR^u)$-cone, for every $i=0,\ldots, T$, $D^u_i$ is tangent to the $(\alpha,\RR^u)$-cone and connects $x_i$ with $y^u_i$; 
		\item by the invariant of the foliation $\cF_\sigma^s$, for every $i=0,\ldots, T$, $D^s_i$ is tangent to the $(\alpha,\RR^s)$-cone and connects $y_i^s$ with $\tilde y_i$;
		\item for every $i=0,\ldots, T$, 
		$$
		\ell^c_i\subset \cF_{\sigma}^{cs}(\tilde y_i)\cap  D_i
		$$
		is an arc that connects $ y^u_i$ with $ y^s_i$; 
		since $\cF_{\sigma}^{cs}(\tilde y_i)$ is tangent to the $(\alpha, \RR^s\oplus \RR^c)$-cone and $D_i$ is tangent to the  $(\alpha, \RR^c\oplus \RR^u)$-cone, we see that $\ell^c_i$ is a smooth arc tangent to the $(\alpha, \RR^c)$-cone;
		\item for every $i=0,\ldots, T$, $D^u_i$ and $\ell_i^c$ are transverse inside $D_i$; $\ell_i^c$ and $D_i^s$ are transverse inside $\cF_\sigma^{cs}(\tilde y_i)$.
	\end{enumerate}
	
	\begin{figure}
		\centering
		\def\svgwidth{\columnwidth}
		\includegraphics[scale=0.6]{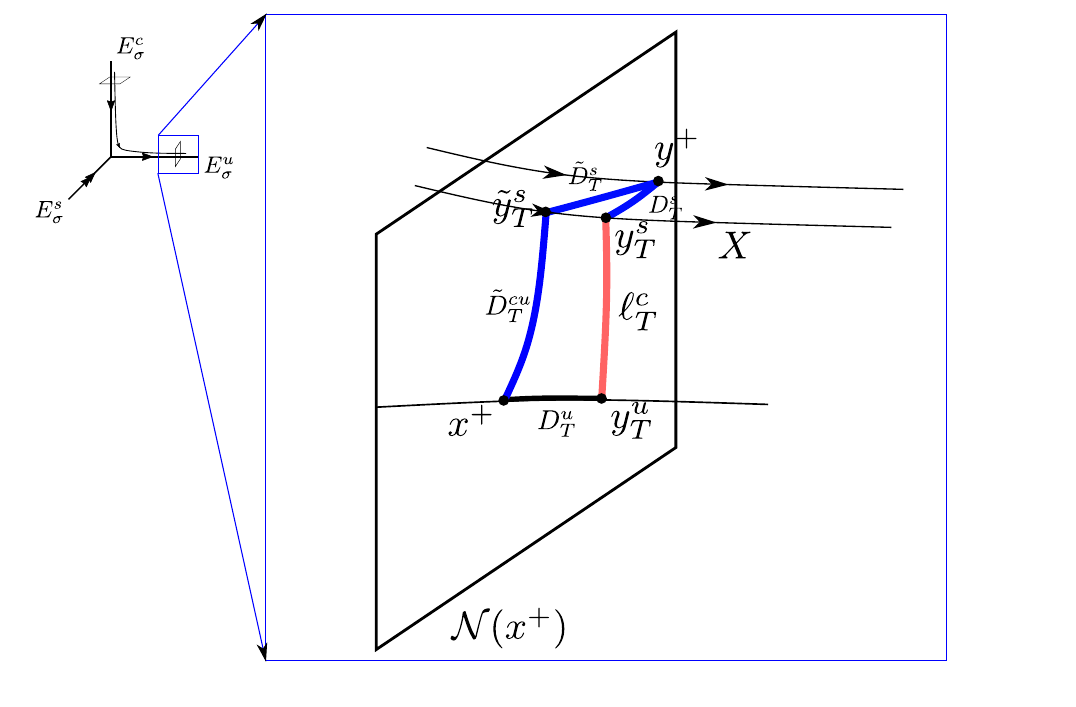}
		\caption{The triple $(D_T^s,\ell_T^c,D_T^u)$ and their projection to $\cN(x^+)$.}
		\label{f.nearw+}
	\end{figure}
	
	\medskip
	\noindent {\em Step 1: transition from $(\cF_{x^+,\cN}^s, \cF^{cu}_{x^+,\cN})$ to $(D^s_T,\ell^c_T,D^u_T)$.}
	
	Near the point $x^+$, we note that $y^s_T$ and $y^u_T$ (keep in mind that $\tilde y_T=y^+\in \cN(x^+)$), as well as the triple $(D_T^s,\ell_T^c,D_T^u)$, may not be on the normal plane $\cN(x^+)$. To solve this issue, we project  $D_T^s$ (which is almost parallel to $\cN(x^+)$) to $\cN(x^+)$ using the map $\cP_{x^+}(\cdot)$ defined by~\eqref{e.Px1}. This results in a disk $\tilde D^s_T$ tangent to the $(2\alpha,E^S_N(x^+))$-cone, together with a point $\tilde y^s_T=\cP_{x^+}(y^s_T)$. See Figure~\ref{f.nearw+}. 
	
	We could project $\ell^c_T$ to $\cN(x^+)$ which results in a one-dimensional arc; however, the same cannot be done for $D^u_T$ because it is not (almost) parallel to $\cN(x^+)$. This is because $D^u_T$ is tangent to the $(\alpha, \RR^u)$-cone which contains the flow direction $X(x^+)$, and the project of  $D^u_T$ to $\cN(x^+)$ along the flow lines may not even be a manifold. To get around this, note that both $\ell^c_T$ and $D^u_T$ are contained in $D_T=f_T(f_{[-1,1]}D^u_0)$, which is a $(u+1)$-dimensional disk tangent to the $(\alpha, \RR^c\oplus\RR^u)$-cone and consists of entire flow segments; every such flow segment is almost orthogonal to $\cN(x^+)$ (recall that at $x^+$, $\vep$ is less than the uniformly relative scale $\rho_0|X(x^+)|$) is taken small enough, and therefore the intersection of $D_T$ with $\cN(x^+)$ contains a $u$-dimensional disk tangent to the $(2\alpha, F^{cu}_N(x^+))$-cone;  this disk can be extended to a disk $\tilde D^{cu}_T$ that is tangent to the $(2\alpha, F^{cu}_{N}(x^+))$-cone. We remark that by construction we have $x^+,\tilde y^s_T\in \tilde D^{cu}_T$.
	
	In summary, on $\cN(x^+)$ we have the following objects:
	\begin{itemize}
		\item  $s$-dimensional disks $\tilde D^s_T$ and $\cF_{x^+,\cN}^s(y^+)$; they are both tangent to the $(2\alpha,E^s_N(x^+))$-cone and contain $y^+$;
		\item  $u$-dimensional disks $\tilde D^{cu}_T$ and $\cF_{x^+,\cN}^{cu}(x^+)$; they are both tangent to the $(2\alpha, F^{cu}_N(x^+))$-cone and contain $x^+$;
		\item we have 
		$$
		\tilde D^s_T\pitchfork \tilde D^{cu}_T= \{\tilde y^s_T\},\,\, \mbox{ and }\,\, \cF_{x^+,\cN}^s(y^+)\pitchfork\cF_{x^+,\cN}^{cu}(x^+)=\{y^{+,F}\}.
		$$ 
		
	\end{itemize}
	Therefore on $\cN(x^+)$, the assumptions of Lemma~\ref{l.changecoord}, Case (2) are satisfied. Recall our assumption that $1000\alpha\cdot \max\{L_0, L_1,L_2\}<0.01$. Thus, assumption~\eqref{e.Elarge} implies
	\begin{equation}\label{e.Ecomparable}
		d_{\tilde D^{s}_T}(\tilde y^s_T,y^+)\in(0.9,1.1)d_{\cF^{s}_{x^+,\cN}}(y^{+,F},y^+) = (0.9,1.1) d_{x^+}^E(y^+).	\footnote{The notation $(a,b)K$ refers to the interval $(aK,bK)$.}
	\end{equation}
	Since both $\tilde D^s_T$ and $D^s_T$ are contained in the $\vep$-ball of $x^+$ and are tangent to the $(3\alpha, \RR^s)$-cone,
	\begin{equation}\label{e.Ecomparable1}
		d_{D^s_T}(y^s_T,y^+)\in (1-c_\vep\alpha,1+c_\vep\alpha)d_{\tilde D^s_T}(\tilde y^s_T,y^+)
	\end{equation}
	where $c_\vep$ is a  constant that remains bounded as $\vep\to0$.  Next, note that by assumption~\eqref{e.Elarge}, there exists a constant $c_1>0$ which only depends on the integers $s,u$ and the dimension of $\cN(x^+)$ (which is $n-1$),  such that 
	$$
	d_{x^+}^E(y^+)\ge c_1d_{\cN(x^+)}(x^+,y^+). 
	$$
	Together with~\eqref{e.Ecomparable} and~\eqref{e.Ecomparable1} we see that
	\begin{equation}\label{e.Ecomparable2}
		d_{ D^s_T}( y^s_T,y^+)\ge \frac12c_1d_{\cN(x^+)}(x^+,y^+)\ge \frac12c_1d(x^+,y^+).
	\end{equation}
	
	{\em Claim}: we have  
	\begin{equation}\label{e.claim111}
		d_{ D^s_T}( y^s_T,y^+)\ge c_2 \max\left\{\length(\ell^c_T), d_{D^u_T}(x^+,y^u_T)\right\}	
	\end{equation}
	for some constant $c_2$ that only depends on the integers $s,u$ and $n$. 
	
	{\em Proof of the claim.} Keep in mind that the triple $(D^s_T,\ell^c_T,D^u_T)$ are pairwise (almost) orthogonal with $\dim D^s_T+\dim \ell^c_T+\dim D^u_T=n=\dim\bM$. This implies that 
	$$
	d_{D^u_T}(x^+,y^u_T)+\length(\ell^c_T)+d_{\tilde D^s_T}(\tilde y^s_T,y^+)\le c_3 d(x^+,y^+)
	$$ 
	for some $c_3$ that only depends on the integers $s,u$ and $n$. This further shows that 
	$$
	\max\left\{d_{D^u_T}(x^+,y^u_T),\length(\ell^c_T)\right\}\le  c_3d(x^+,y^+).
	$$
	Combining it with~\eqref{e.Ecomparable2} we have 
	$$
	d_{ D^s_T}( y^s_T,y^+)\ge\frac12 c_1c_3^{-1}\max\left\{d_{D^u_T}(x^+,y^u_T),\length(\ell^c_T)\right\},
	$$
	as claimed.
	
	\medskip
	\noindent {\em Step 2: backward expansion on $D^s$.}
	
	By~\eqref{e.claim111} and the dominated splitting $(E^{ss}_\sigma\oplus E^c_\sigma)\oplus_< E^u_\sigma$, there exist constants $c_4>0$ and $\tilde \lambda_\sigma>1$ such that 
	\begin{equation}\label{e.domination}
		d_{ D^s_{T-i}}( y^s_{T-i},\tilde y_{T-i})\ge c_2c_4\tilde\lambda_\sigma^i \max\left\{\length(\ell^c_{T-i}), d_{D^u_{T-i}}(x_{T-i},y^u_{T-i})\right\}, 
	\end{equation}
	for $i=0,1,\ldots, T.$ 
	Furthermore, since vectors contained in the $(\alpha, \RR^s)$-cone are exponentially expanded under the iteration of $f_{-1}$, we have
	\begin{equation}\label{e.expansion}
		d_{ D^s_{0}}( y^s_{0},\tilde y_{0})\ge c_4\tilde\lambda_\sigma^T d_{ D^s_{T}}( y^s_{T},y^+)\ge \frac12 c_4\tilde\lambda_\sigma^T d_{x^+}^E(y^+),
	\end{equation}
	where the second inequality follows from~\eqref{e.Ecomparable} and~\eqref{e.Ecomparable1}. 
	
	\medskip
	\noindent {\em Step 3: transition from $(D^s_0,\ell^c_0,D^u_0)$ back to $(\cF^s_{x,\cN},\cF^{cu}_{x,\cN})$.}
	
	Below we will make the transition from the triple $(D^s_0,\ell^c_0,D^u_0)$ back to the coordinate system $(\cF^s_{x,\cN},\cF^{cu}_{x,\cN})$ and compare $d_{ D^s_{0}}( y^s_{0},\tilde y_0)$ with ${\cF_{x,\cN}^s}(y^F, y)$. 
	
	To achieve this, we project $D^s_0$ (which is almost parallel to $\cN(x)$) using the map $\cP_x$ to $\cN(x)$. This gives us an $s$-dimensional disk $\tilde D^s_0$ tangent to the $(2\alpha,E^s_N(x))$-cone, together with a point $\tilde y^s_0$ which is the projection of $y^s_0$; see Figure~\ref{f.nearw-}. Since $y^s_0$ is contained in $D_0=f_{[-1,1]}(D^u_0)$ where the latter is saturated by local flow orbits, we see that $\tilde y^s_0\in D^u_0$. We also remark that the projection of $\tilde y_0$ to $\cN(x)$ is precisely $y$. Therefore we have (increase $c_\vep$ if necessary)
	\begin{equation}\label{e.Ecomparable3}
		d_{\tilde D^s_0}(\tilde y^s_0,y)\in (1-c_\vep\alpha,1+c_\vep\alpha)d_{D^s_0}(y^s_0,\tilde y_0).
	\end{equation}
	
	On $\cN(x)$ we have the following objects:
	\begin{itemize}
		\item $s$-dimensional disks $\tilde D^s_0$ and $\cF_{x,\cN}^s(y)$; they are both tangent to the $(2\alpha,E^s_N(x))$-cone and contain $y$;
		\item $u$-dimensional disks $D^{u}_0$ and $\cF_{x,\cN}^{cu}(x)$; they are both tangent to the $(2\alpha, F^{cu}_N(x))$-cone and contain $x$;
		\item we have 
		$$
		\tilde D^s_0\pitchfork  D^{u}_0= \{\tilde y^s_0\},\,\, \mbox{ and }\,\, \cF_{x,\cN}^s(y)\pitchfork\cF_{x,\cN}^{cu}(x)=\{y^{F}\}.
		$$ 
	\end{itemize}
	We also project $\ell^c_0$ to $\cN(x)$ which gives us an arc $\tilde \ell^c_0\subset D^u_0$ connecting $y^u_0$ and $\tilde y^s_0$. Inside $D^u_0$ we have 
	$$	
	d_{D^u_0}(x,\tilde y^s_0)\le d_{D^u_0}(x,y^u_0)+\length(\tilde \ell^c_0)\le d_{D^u_0}(x,y^u_0)+\length(\ell^c_0).\footnote{This explains the use of the triple $(D^s_i,\ell^c_i,D^u_i)$ instead of the pair $(D^s_i,D_i^u)$.}
	$$
	In view of~\eqref{e.domination} we conclude that 
	\begin{equation*}
		d_{D^u_0}(x,\tilde y^s_0)\le 2c_2^{-1}c_4^{-1}\tilde\lambda_\sigma^{-T}d_{D^s_0}(y^s_0, \tilde y_0)\le 4c_2^{-1}c_4^{-1}\tilde\lambda_\sigma^{-T}d_{\tilde D^s_0}(\tilde y^s_0,y), 
	\end{equation*}
	where the second inequality follows from~\eqref{e.Ecomparable3}. Note that the constants $c_2,c_4$ and $\tilde\lambda_\sigma$ do not depend on $r $ (the radius of the smaller ball). Therefore, by taking $r >0$ small enough and thus increasing $T$, we have 
	$$
	4c_2^{-1}c_4^{-1}\tilde\lambda_\sigma^{-T}<\frac12.
	$$
	In particular, the assumptions of Lemma~\ref{l.changecoord}, Case (2) is, once again, satisfied. We obtain, in view of Remark~\ref{r.1.5},
	\begin{equation}\label{e.Ecomparable4}
		\begin{split}
			d_{x}^E(y) = d_{\cF^{s}_{x,\cN}}(y^{F},y)\in(0.9,1.1)d_{\tilde D^{s}_0}(\tilde y^s_0,y),\,\,\mbox{ and }d_{x}^E(y)\ge d_{x}^F(y).
		\end{split}
	\end{equation}
	Collecting~\eqref{e.expansion},~\eqref{e.Ecomparable3} and~\eqref{e.Ecomparable4} we have
	$$
	d_{x}^E(y)\ge \frac18c_4\tilde\lambda_\sigma^T \cdot d_{x^+}^E(y^+).
	$$
	Note that as $r\to 0$, $T=t(x)$ goes to infinitely uniformly in $x\in \partial W_r^-\cap\Lambda$. Therefore, we can fix $\lambda_\sigma\in (1,\tilde\lambda_\sigma)$ and decreasing $\overline r\in (0,r)$ such that for all $r\le \overline r$ and $x\in \partial W_r^-\cap\Lambda$ it holds 
	$$
	\frac18c_4\tilde\lambda_\sigma^T\ge \lambda_\sigma^T.
	$$
	We conclude that 
	$$
	d_{x}^E(y)\ge \lambda_\sigma^T \cdot d_{x^+}^E(y^+),
	$$ 
	as desired.

	We are left with the scaled shadowing property. Let $\rho\in (0,\rho_0]$ be fixed. We recall Proposition~\ref{p.Fx} which shows that for every $z\in\Reg(X)$ (not necessarily in $W^c_r$), there exists $\tilde \rho\le \rho_0K_0^{-1}$ such that 
	\begin{equation}\label{e.tilderho}
		B_{\tilde\rho|X(z)|}(z)\subset F_x\left(U_{ 3\rho K_0^{-1}|X(z)|}(z)\right).
	\end{equation}
	This in particular implies that $\cP_{z}$ is well-defined on $B_{\tilde\rho|X(z)|}(z)$ and send it to $\cN_{\rho K_0^{-1}|X(z)|}(z)$. 
	
	Below we will prove that for every $i\in [0,T]\cap\NN$, the point $\tilde y_i$ is contained in $B_{\tilde \rho|X(x_i)|}(x_i)$, and consequently, can be projected to $\cN_{\rho K_0^{-1}|X(x_i)|} (x_i)$ by the map $\cP_{x_i}$. This shows that $\cP_{x_i}(\tilde y_i)$ is $\rho$-scaled shadowed by $x_i$ up to time $t=1$; then the conclusion of (3) follows from Lemma~\ref{l.shadowing2}.

To estimate $d(x_i,\tilde y_i)$, we use~\eqref{e.domination} to obtain
\begin{align*}
	d(x_i,\tilde y_i)&< d(x_i, y^u_i) + d(y^u_i, y^s_i) + d(y^s_i,\tilde y_i)\\ 
	&\le d_{D^u_{i}}(x_{i},y^u_{i}) +\length(\ell^c_{i})+ d_{D^s_i}(y^s_i, \tilde y_i)  \\
	&\le \left(1+2c_2^{-1}c_4^{-1} \tilde\lambda_\sigma^{-(T-i)}\right)	d_{D^s_i}(y^s_i, \tilde y_i)\\
	\numberthis\label{e.dd}	&\le c_5d_{D^s_i}(y^s_i, \tilde y_i). 
\end{align*}

On the other hand, note that for every $i\in[0,T]\cap \NN$, $X(x_i)$ is contained in the $(\alpha, \RR^c\oplus \RR^u)$-cone. By the dominated splitting $E^{ss}_\sigma \oplus_<\left(E^c_\sigma\oplus E^u_\sigma\right)$, for $\alpha>0$ small enough there exist constants $c_6>0$ and $\overline\lambda_\sigma>1$ such that, for every vector $u$ in the $(\alpha, \RR^s)$-cone and vector $v$ in the $(\alpha, \RR^c\oplus \RR^u)$-cone, we have 
$$
|Df_i (u)| \le c_6 (\overline\lambda_\sigma)^{-i}|Df_i (v)|. 
$$ 
Together with~\eqref{e.dd}, we obtain, using~\eqref{e.Ecomparable3}:
\begin{align*}
	d(x_i,\tilde y_i)&<c_5c_6(\overline\lambda_\sigma)^{-i}\cdot|Df_i (X(x))|\cdot d_{D^s_0}(y^s_0, \tilde y_0)\\
	&\le  2c_5c_6|X(x_i)|d_{\tilde D^s_0}(\tilde y^s_0,y)\\
	&\le c_7 |X(x_i)| d(x,y)\\
	&\le c_7 |X(x_i)|\vep,
\end{align*}
for some constant $c_7>0$ that does not depend on $\vep$. 

Taking $\vep\le \vep_1< \frac{\tilde \rho}{c_7}$ guarantees that $d(x_i,\tilde y_i)<\tilde \rho |X(x_i)|$. As a result, $\cP_{x_i}(\tilde y_i)$ is well-defined and is contained in $\cN_{\rho K_0^{-1}|X(x_i)|}(x_i)$ for every $0\le i\le T$. 
By Lemma~\ref{l.shadowing2}, the orbit of $y$ is $\rho$-scaled shadowed by the orbit of $x$ up to  time $T$. We conclude the proof of Lemma~\ref{l.Elarge}.

\end{proof}

\begin{proof}[Proof of Lemma~\ref{l.Flarge}]
First we show  $ d_{x^+}^F(y^+)\ge d_{x^+}^E(y^+)$ under the assumptions of this lemma. Assume that this is not the case; that is, 
$$
d_{x^+}^E(y^+)>  d_{x^+}^F(y^+)
$$
which is essentially~\eqref{e.Elarge}. Then (the same proof as) Lemma~\ref{e.Elarge} (1) shows that $d_{x}^E(y)> d_{x}^F(y)$, which contradicts the assumption~\eqref{e.Flarge}. 

Below we will prove the desired expansion on the fake $cu$-leaf. The proof is analogous to the previous lemma, with a major difficulty: vectors in $cu$-cone may not be expanded by $Df_1$. To solve this issue, we use the sectional-hyperbolicity of $Df_1$ to get the desired expansion on $\ell^c$. See Step 2 below for the detail.

\medskip
\noindent {\em Step 1: transition from $(\cF^s_{x,\cN},\cF^{cu}_{x,\cN})$ to $(D^s_0,\ell^c_0,D^u_0)$.}

We define the triple $(D^s_i,\ell^c_i,D^u_i)$ as well as the points $y^u_i,y^s_i$ and $\tilde y_i$ in the same way as before. 
We also write $\tilde D^s_0$, $\tilde \ell^c_0$ and $\tilde y^s_0$ for the projection of $D^s_0$, $\ell^c_0$ and $y^s_0$ to the normal plane $\cN(x)$, see Figure~\ref{f.nearw-}, and $\tilde D^s_T, \tilde D^{cu}_T$ for the projection of $D^s_T$, $D_T$ to the normal plane $\cN(x^+)$, see Figure~\ref{f.nearw+}. Then~\eqref{e.Ecomparable3} remains valid. Furthermore, the assumptions of Lemma~\ref{l.changecoord}, Case (1) is satisfied. We thus obtain
\begin{equation}\label{e.Fcomparable1}
	d_{D^u_0}(x,\tilde y^s_0)\in (0.9,1.1)d_{x}^F(y),\,\,	\mbox{ and }\,\, d_{D^u_0}(x,\tilde y^s_0)\ge d_{\tilde D^s_0}(\tilde y^s,y).
\end{equation}

\medskip
\noindent {\em Step 2: forward expansion.}

Note that inside $D^u_0$ we have 
$$
d_{D^u_0}(x,y^u_0)+\length(\tilde \ell^c_0)\ge d_{D^u_0}(x,\tilde y^s_0)\ge d_{\tilde D^s_0}(\tilde y^s,y)\ge\frac12 d_{D^s_0}(y^s_0,\tilde y_0),
$$
where we use~\eqref{e.Fcomparable1} for the second, and~\eqref{e.Ecomparable3} for the last inequality. This further implies that 
\begin{equation}\label{e.Fcomparable2}
	d_{D^u_0}(x,y^u_0)+\length( \ell^c_0)\ge \frac12 d_{D^s_0}(y^s_0,\tilde y_0).   \hspace{2cm} (u+c\mbox{ large at }x) 
\end{equation}

By the dominated splitting $E^{ss}_\sigma\oplus E^c_\sigma\oplus E^u_\sigma$, there exist $\tilde\lambda_\sigma>1$ and a constant $c_1>0$ such that~\eqref{e.Fcomparable2} becomes
\begin{equation}\label{e.Fcomparable2.51}\begin{split}
		&d_{D^u_i}(x_i,y^u_i)+\length( \ell^c_i)\\ &\ge \frac12c_1\tilde\lambda_\sigma^{i} d_{D^s_i}(y^s_i,\tilde y_i),\,\, i=0,1,\ldots, T.  \hspace{1.1cm} (u+c\mbox{ large at every }x_i ) \end{split}
\end{equation}
Furthermore, by the uniform expansion of vectors in the $(\alpha,\RR^u)$-cone, we have 
\begin{equation}\label{e.Fcomparable2.52}
	d_{D^u_i}(x,y^u_i)\ge c_1\tilde\lambda_\sigma^{i}d_{D^u_0}(x,y^u_0),\,\, i=0,1,\ldots, T.
\end{equation}
Next, we project $\ell^c_T$ to $\cN(x^+)$ along the flow lines. This gives us a one-dimensional arc $\tilde \ell^c_T\subset D_T$.
Note that  $\ell^c_i\subset D_i$ where the latter is tangent to the $(2\alpha, \RR^c\oplus \RR^u)$-cone for $i=1,\ldots, T$. By the sectional hyperbolicity applied to the area of the (almost) parallelogram with $\tilde\ell^c_0$ and $X(x)$ as its sides, we have (decrease $c_1$ and $\tilde\lambda_\sigma$ if necessary)
\begin{equation}\label{e.Fcomparable3}
	\length(\tilde \ell^c_T)|X(x^+)|\ge c_1\tilde\lambda_\sigma^T\length(\tilde \ell^c_0)|X(x)|.  \footnote{Here we can use $|X(x)|$ to replace $|X(y)|$ for some $y\in \tilde\ell^c_0$ because $\tilde\ell^c_0\subset B_\vep(x)\subset B_{\rho_0|X(x)|}(x)$, in which the flow speed differ by a uniform ratio which can be combined with $c_1$ (see Proposition~\ref{p.Fx}). The same holds at $x^+$.}
\end{equation}
Since the flow speed $|X(\cdot)|$ depends Lipschitz continuously on $d(\cdot,\sigma)$, we see that $$|X(x)|/|X(x^+)|>c_2$$ for some constant $c_2>0$, as $x$ and $x^+$ are both contained in $\partial B_{r_0}(\sigma)$. Then~\eqref{e.Fcomparable3} implies the following exponential expansion on $\ell^c$:
\begin{equation}\label{e.Fcomparable4}
	\length( \ell^c_T)\ge \frac12\length(\tilde \ell^c_T)\ge \frac12c_1c_2\tilde\lambda_\sigma^T\length(\tilde \ell^c_0).
\end{equation}
The first inequality holds because $\ell^c_T$ is almost parallel to $\cN(x^+)$, and $\ell^c_T$, $\tilde \ell^c_T$ are both contained in $B_\vep(x^+)$.

Adding~\eqref{e.Fcomparable2.52} and~\eqref{e.Fcomparable4} together and applying~\eqref{e.Fcomparable2}, we conclude the following exponential expansion on the $c$ and $u$ direction:
\begin{equation}\label{e.Fcomparable5}\begin{split}
		d_{D^u_T}(x^+,y^u_T)+\length( \ell^c_T)&\ge c_3\tilde\lambda_\sigma^T\left(d_{D^u_0}(x,y^u_0)+\length(\tilde \ell^c_0)\right)\\
		&\ge \frac14c_3\tilde\lambda_\sigma^Td_{D^u_0}(x,\tilde y^s_0).
	\end{split}
\end{equation}

\medskip
\noindent {\em Step 3: transition from   $(D^s_T,\ell^c_T,D^u_T)$ to $(\cF^s_{x^+,\cN},\cF^{cu}_{x^+,\cN})$.}

On $\cN(x^+)$ we already proved that $  d_{x^+}^F(y^+)\ge d_{x^+}^E(y^+)$. As in Step 1 of the previous lemma, we can apply Lemma~\ref{l.changecoord} to get 
\begin{equation}\label{e.Fcomparable5.5}\begin{split}
		&d_{\tilde D^{cu}_T}(x^+,\tilde y^s_T)\in (0.9,1.1)d_{x^+}^F(y^+),\,\,\mbox{ and }\\
		&d_{\tilde D^{cu}_T}(x^+,\tilde y^s_T)\ge 0.9 d_{\tilde D^s_T}(\tilde y^s_T,y^+).\end{split}
\end{equation}
Since $\tilde D^{cu}_T$ and $\tilde D^s_T$ are tangent to $(2\alpha, F^{cu}_N(x^+))$-cone and $(2\alpha, E^s_N(x^+))$-cone, respectively, following the previous inequality there exists a constant $c_4>0$ depending only on the integers $s,u$ such that 
\begin{equation}\label{e.Fcomparable6}
	d_{\tilde D^{cu}_T}(x^+,\tilde y^s_T)\ge c_4d_{\cN(x^+)}(x^+,y^+)\ge c_4d(x^+,y^+). 
\end{equation}
On the other hand, since $D^u_T$, $\ell^c_T$ and $D^s_T$ are (almost) pairwise orthogonal, there exists a constant $c_5>0$ depending only on $s,u,n$ such that 
\begin{align*}
	d(x^+,y^+)&\ge c_5 \max\left\{ d_{D^s_T}(y^s_T,y^+),\length(\ell^c_T),d_{D^u_T}(x^+,y^u_T) \right\}\\
	\numberthis\label{e.Fcomparable7}	&\ge c_6 \left(d_{D^u_T}(x^+,y^u_T)+\length(\ell^c_T)\right),
\end{align*}
where the second inequality is due to~\eqref{e.Fcomparable2.51} with $i=T$.  Collecting~\eqref{e.Fcomparable1}, \eqref{e.Fcomparable5}, \eqref{e.Fcomparable5.5}, \eqref{e.Fcomparable6} and~\eqref{e.Fcomparable7}, we see that
\begin{align*}
	&d_{x^+}^F(y^+) &&\\
	&\ge \frac12 d_{\tilde D^{cu}_T}(x^+,\tilde y^s_T)  &&\mbox{by}~\eqref{e.Fcomparable5.5}\\
	&\ge  \frac12c_4d(x^+,y^+) &&\mbox{by}~\eqref{e.Fcomparable6}\\
	&\ge \frac12c_4 c_6 \left(d_{D^u_T}(x^+,y^u_T)+\length(\ell^c_T)\right)&&\mbox{by}~\eqref{e.Fcomparable7}\\
	&\ge  \frac18 c_3c_4 c_6  \tilde\lambda_\sigma^Td_{D^u_0}(x,\tilde y^s_0)&&\mbox{by}~\eqref{e.Fcomparable5}\\
	&\ge\frac{1}{16}c_3c_4c_6\tilde\lambda_\sigma^T\cdot d_{x}^F(y).&&\mbox{by}~\eqref{e.Fcomparable1}
\end{align*}

Since the coefficient does not depend on $r $ or $T$, we can fix $\lambda_\sigma\in (1,\tilde\lambda_\sigma)$ and decrease $\overline r $ (therefore increase $T$) such that $\frac{1}{16}c_3c_4c_5\tilde\lambda_\sigma^T\ge \lambda_\sigma^T$. With this we conclude the proof of Lemma~\ref{l.Flarge}.

\end{proof}

\section{Notations and parameters list}\label{s.parameter}

For the convenience of our readers, in this section we summarize some important notations and parameters that will appear in this paper. It will be useful to print out this section for future reference.

\noindent {\bf Parameters:}

\begin{enumerate}
\item $r_0>0$ is the size of the (larger) ball at each singularity $\sigma\in\Sing(X)$.
\item $r\in (0,r_0)$ is the size of the smaller balls at $\sigma$. $r$ and $r_0$ are given by~\ref{l.nearEc} and will be modified for a finite number of times in subsequent lemmas (for example, in Lemma~\ref{l.Elarge} and~\ref{l.Flarge}).

\bigskip
\item $\alpha$ is the size of the cone on the normal bundle.
\item $\beta_0$ is the parameter for the recurrence Pliss time (for the definition, see Section~\ref{s.pliss}) to a small neighborhood $W$ of $\Sing(X)$. Its value is given by the hyperbolicity on the $F^{cu}_N$ bundle and the $C^1$ norm of $\psi^*_{-1}$, see Theorem~\ref{t.Bowen}. 
\item $\beta>0$ is taken after $\beta_0$ (see Lemma~\ref{l.rec}) such that if an orbit segment spends less than a $\beta$ proportion of its time inside the neighborhood  $W$, then there exist $\beta_0$-recurrence Pliss times for $W$.

\bigskip

\item $\rho_0>0$ is small enough such that 
\begin{itemize}
	\item the fake foliations $\cF_{x,N}^i,i=s,cu$ are well-defined on $\cN_{\rho_0|X(x)|}(x)$ for every $x\in\Reg(X)$ (Section~\ref{ss.fake.normal}); 
	\item the projection $\cP_{x}: B_{\rho_0|X(x)|}(x)\to \cN_{\rho_0|X(x)|}(x)$ given by~\eqref{e.Px1} is well-defined (Proposition~\ref{p.Fx})
	\item for every regular point $x$, the holonomy map along the flow lines, $\cP_{1,x}$, is well-defined from $\cN_{\rho_0K_0^{-1}|X(x)|}(x)$ to $\cN_{{\rho_0|X(f_1(x))|}}(f_1(x))$ (see Proposition~\ref{p.tubular}).
\end{itemize} 
\item $a>0$ and $\rho_1<\rho_0$ are taken w.r.t. the uniform continuity of $D\cP_{1,x}$ (Proposition~\ref{p.tubular3} and Lemma~\ref{l.cuhyptime1}).
\item $\rho<\rho_1$ is the scale for the $\rho$-scaled shadowing in the Main Proposition, Proposition~\ref{p.key}. 
\item $\rho'$ is such that the image of $\cN_{\rho K_0^{-2}|X(x)|}(x)$ under $\cP_{1,x}$ contains a $\rho'|X(x_1)|$ ball at $x_1$. Its existence is given by Lemma~\ref{p.tubular4}, and is used in the proof of the Main Proposition.

\bigskip	
\item $\vep_0$ is taken (after $r_0,r$) such that $\cP_{x}(\cdot)$ is well-defined in $B_{\vep_0}(x)$ for every $x\notin W_{r}$ (for $W_r$, see the list of notations below).
\item $\vep_1$ is the size of the Bowen balls given by Lemma~\ref{l.Elarge} and~\ref{l.Flarge}), with $\vep_1<\vep_0$. 
\item $\vep_2\le \vep_1$ is the size of the Bowen balls given by the Main Proposition~\ref{p.key} and is the scale of the Bowen Property on $\cG$.

\bigskip
\item $\delta>0$ is the scale for the specification on $\cG_{spec}$ in Theorem~\ref{t.spec}; after finding $\vep_2$, we will take $\delta = \vep_2/(1000Lip)$.
\item $\xi>0,\zeta>0$ are the size of the stable/unstable manifolds in Section~\ref{s.spec}. Both depend on $\delta$.  

\medskip 
\item $\lambda_\sigma>1$ is the expansion rate on the normal bundle near singularities, given by Lemma~\ref{l.Elarge} and~\ref{l.Flarge}.

\bigskip

\item $L_0, L_1, L_2$: they are constants that appear only in Section~\ref{ss.cone} concerning the change of coordinates in a cone.
\item $L>0$ is the time cut-off for the tail specification on $\cG_{spec}$; can be found in Section~\ref{s.spec}.

\bigskip
\item $K_0$ is the constant given by~Proposition~\ref{p.tubular} concerning the domain of $\cP_{1,x}$.
\item $K_1$ is the upper-bound of $\|D\cP_{1,x}\|$, Proposition~\ref{p.tubular3}.
\item $K^s(x)$ is the size such that the stable manifold of a regular point $x$ has a  transverse intersection with $W^u(\gamma)$, where $\gamma$ is a hyperbolic periodic orbit given by Theorem~\ref{t.top}.

\bigskip
\item $\tau$ is the maximal gap size for the specification on $\cG_{spec}$. Not to be confused with $\tau_{x}(y)$ (the time for $\cP_{x}(\cdot)$, see~\eqref{e.Px2}) or $\tau_{x,y}(t)$ (the time change in the scaled shadowing property,~\eqref{e.tau}).
\end{enumerate}

\noindent {\bf Notations:}
\begin{enumerate}
\item For a given regular point $x\in\bM$ and $t\in\RR$, we write 
$$
x_t = f_t(x). 
$$
Meanwhile, we use superscripts such as $x^1, x^2$, etc to denote different points.  
\item  As mentioned before, we use the normal math italic font $E,N, P,$ etc. for subspaces, bundle, maps, etc. defined on the tangent bundle, and the calligraphic font $\cC, \cN,\cP$, etc. for orbit segments/ submanifolds/ Poincar\'e maps on the manifold.
\item $B_{t,\vep}(x)$ is the $\vep$-Bowen ball of the orbit segment $(x,t)$. One exception of this rule is $W^*$, $*=ss,s,cu$, etc which denotes invariant manifolds. 
\item $E^s_N\oplus F^{cu}_N$ is the dominated splitting on the normal bundle $N=\bigcup_{x\in \Sing(X)^c} N(x)$ under  the scaled linear Poincar\'e flow $\psi_{t,x}$. It is obtained by projecting the dominated splitting $E^{ss}\oplus F^{cu}$ to the normal bundle (recall Lemma~\ref{l.splitting}). We will assume that $E^{ss}$ and $F^{cu}$ are orthogonal at every point. This implies that $E^s_N$ and $F^{cu}_N$ are also orthogonal at regular points.

\item $\cP_{t,x}$ is the holonomy map from $\cN_{\rho_0K_0^{-t}|X(x)|}(x)$ to $\cN_{\rho_0|X(x_t)|}(x_t)$ for $t>0$. See~\eqref{e.Ptx1}.
\item $P_{t,x}$ is the lift of $\cP_{t,x}$ to the normal bundle via the exponential map at the corresponding points.
\item $\cP_x(\cdot)$ is the projection from $B_{\rho_0|X(x)|}(x)$ to $\cN_{\rho_0|X(x)|}(x)$. See~\eqref{e.Px1}  and~\eqref{e.Px2}.
\item For $r \in (0,{r_0})$, $W_{r }(\sigma)$ is the union of local orbit segments in $B_{r_0}(\sigma)$ that intersect with $B_{r }(\sigma)$. See~\ref{e.W.def}.
\item $W_\Sing$ is a small neighborhood of $\sigma$ that is contained in $B_{r_0}(\sigma)$. One can think of it as $W_r$ for some proper choice of $r$ and $r_0$.
\item $B^*_\rho(x,t)$ is the $\rho$-scaled tubular neighborhood of the orbit segment $(x,t)$, Definition~\ref{d.liao}.
\item $\CP_\sigma, \sigma\in\Sing(X)\cap\Lambda$ are some compact neighborhoods of $\sigma$ in $W^u(\sigma)$; in order for the specification to hold at a certain scale $\delta$, the initial point of an orbit segment $(x,t)$ must avoid $\cup_\sigma\CP_\sigma$ (whose choice depends on $\delta$). See Theorem~\ref{t.spec}.
\item $U_\Sing$ is a neighborhood of $\cup_\sigma \CP_\sigma$ that appears in Theorem~\ref{t.spec}; in order to achieve specification (more precisely, to achieve a bounded maximum gap size $\tau$), the initial point $x$ must be taken outside of $U_\Sing$. Not to be confused with $U(\sigma)$ which only appeared in Lemma~\ref{l.lgw} and~\ref{l.nearEc}.

\item $\cG_B$ is the orbit segments collection with the Bowen property, Theorem~\ref{t.Bowen}.
\item $\cG_{spec}(\delta,r,W_\Sing, U_\Sing)$  is the orbit segments collection with specification property, Theorem~\ref{t.spec}. 
\end{enumerate}

\section*{Acknowledgments} The authors are grateful to the anonymous referees for their careful reading and helpful comments, which significantly improved the presentation of the current paper.  We also thank Vaughn Climenhaga, Daniel Thompson,  Shaobo Gan, Huyi Hu, Ming Li, Yi Shi, Rusong Zheng and Xingzhong Liu for their useful comments and suggestions.

\end{document}